\documentclass{article}
\usepackage[utf8]{inputenc}

\title{Observability estimates for the Schrödinger equation in the plane with periodic bounded potentials from measurable sets}
\author{Kévin Le Balc'h, Jérémy Martin}

\usepackage[top=3cm, bottom=2cm, left=3cm, right=3cm]{geometry}	

\usepackage{enumitem}
\usepackage{array}											

\usepackage{amsfonts}
\usepackage{amsmath}
\usepackage{amsthm}
\usepackage{amssymb}
\usepackage{mathtools}

\usepackage{multicol}

\usepackage{bbm}

\usepackage{stmaryrd}

\usepackage[scr]{rsfso}

\usepackage{comment}

\usepackage[dvipsnames]{xcolor}

\usepackage{hyperref}

\hypersetup{	
  colorlinks=true,
  breaklinks=true,
  urlcolor= red,
  linkcolor= red,
  citecolor=ForestGreen
}

\numberwithin{equation}{section}

\newtheorem{theorem}{Theorem}[section]
\newtheorem{definition}[theorem]{Definition}
\newtheorem{lemma}[theorem]{Lemma}
\newtheorem{proposition}[theorem]{Proposition}
\newtheorem{exemple}[theorem]{Example}
\newtheorem{corollary}[theorem]{Corollary}
\newtheorem{remark}[theorem]{Remark}
\numberwithin{equation}{section}

\DeclareMathOperator{\Supp}{Supp}
\DeclareMathOperator{\Reelle}{Re}

\newcommand{\norme}[1]{\left\lVert#1\right\rVert}

\newcommand{\ensemblenombre}[1]{\mathbb{#1}}
\newcommand{\N}{\ensemblenombre{N}}
\newcommand{\Z}{\ensemblenombre{Z}}

\newcommand{\R}{} 
\renewcommand{\R}{\ensemblenombre{R}}
\newcommand{\C}{\ensemblenombre{C}}

\newcommand{\T}{\ensemblenombre{T}}

\newcommand{\rr}{\mathbb{R}}
\newcommand{\nn}{\mathbb{N}}
\def\ops#1#2{{\text{op}_{#1}(#2)}}
\def\un{{\mathrm{1~\hspace{-1.4ex}l}}}

\newcommand{\Ima}{\text{Im}}

\begin{document}

\maketitle

\begin{abstract}
    The goal of this article is to obtain observability estimates for Schrödinger equations in the plane $\R^2$. More precisely, considering a $2 \pi \Z^2$-periodic potential $V \in L^{\infty}(\R^2)$, we prove that the evolution equation $i  \partial_t u  = -\Delta u  + V(x) u$, is observable from any $2 \pi \Z^2$-periodic measurable set, in any small time $T>0$. We then extend Taüffer's recent result \cite{Tau22} in the two-dimensional case to less regular observable sets and general bounded periodic potentials. The methodology of the proof is based on the use of the Floquet-Bloch transform, Strichartz estimates and semiclassical defect measures for the obtaining of observability inequalities for a family of Schrödinger equations posed on the torus $\rr^2/2\pi\mathbb Z^2$.
\end{abstract}
 \small
\tableofcontents
\normalsize

\section{Introduction}

\subsection{Observability inequalities for the Schrödinger equation in $\R^d$}
In this work, we are interested in the observability of the following Schrödinger equation
\begin{equation}
	\label{eq:SchrodingerObs}
		\left\{
			\begin{array}{ll}
				i  \partial_t u  = (-\Delta + V(x)) u & \text{ in }  (0,+\infty) \times \R^d, \\
				u(0, \cdot) = u_0 & \text{ in } \R^d,
			\end{array}
		\right.
\end{equation}
where $u_0 \in L^2(\R^d)$ and $V \in L^{\infty}(\R^d)$. In the case when $V$ is $2\pi\mathbb Z^d$-periodic, that is, satisfies 
\begin{equation*}
    \label{eq:periodicityV}
   V(x+2k \pi) = V(x),\qquad \forall x \in \R^d,\ \forall k \in \Z^d.
\end{equation*}

The equation \eqref{eq:SchrodingerObs} can describe the behaviour of an electron in a crystal, see for instance \cite{Kli12} and the references therein. The notion of observability is defined as follows:

\begin{definition}
Let $T>0$ be a positive time and $b : \rr^d \longrightarrow \rr_+$ be a non-negative measurable function. The equation \eqref{eq:SchrodingerObs} is said to be observable from $b$ in time $T>0$ if and only if there exists a positive constant $C>0$ such that 
\begin{equation*}
    \forall u_0 \in L^2(\rr^d), \quad \|u_0\|^2_{L^2(\rr^d)} \leq C \int_0^T \int_{\rr^d} b(z) |u(t,z)|^2 dz dt,
\end{equation*}
where $u$ is the mild solution of \eqref{eq:SchrodingerObs} with initial data $u_0$.
\end{definition}
When $b=\un_{\omega}$ with $\omega \subset \rr^d$ a measurable subset, some geometric conditions can be required to ensure the observability of \eqref{eq:SchrodingerObs}. In the one-dimensional case $d=1$, when $V=0$, it has been shown in \cite{MPS21} and \cite{HWW22} that the free Schrödinger equation is observable at some time $T>0$ from $b=\un_{\omega}$ if and only if $\omega$ is thick, that is,
\begin{equation*}
    \exists \gamma, L>0, \forall x \in \rr, \quad |\omega \cap (x+[0,L])|\geq \gamma,
\end{equation*}
where $|A|$ denotes the Lebesgue measure of $A\subset \rr$. In higher dimensions $d \geq 2$, this thickness condition turns out to be necessary (\cite[Theorem~2.6]{MPS21}). However, as explained by the authors of \cite{MPS21}, the question of its sufficiency remains open. A sufficient condition was first given in \cite[Proposition 2.11]{MPS21}.  Recently, the particular case of periodic sets have been investigated for the observability of the free Schrödinger equation. In \cite[Theorem~2]{Tau22}, Taüffer has shown that \eqref{eq:SchrodingerObs} with $V=0$ is observable in any time $T>0$ from $b=\un_{\omega}$, where $\omega$ is any non-empty $2\pi\mathbb Z^d$-periodic open subset of $\rr^d$. Notice that non-empty $2\pi\mathbb Z^d$-periodic open subsets are trivially thick subsets. It is also worth mentioning that periodic open subsets do not satisfy necessarily the well-known geometric control condition (GCC), roughly stating that every generalized geodesic meets $\omega$ in time $t \leq T$, that turns to be the necessary and sufficient geometric condition for the wave equation, see \cite{BJ16}. On the other hand, \cite{EV18} shows that the thickness condition is necessary and sufficient for the observability of the heat equation in dimensions $d \geq 1$. According to these results, the geometric condition ensuring the observability of the Schrödinger equation in $\R^d$ is strictly less restrictive than the one for the wave equation and more restrictive (in a large sense) than the one for the heat equation. 

\subsection{Main results in the two-dimensional case}
Our main result extends the result by Taüffer to the case of less regular observable function $b$ and general $2\pi\mathbb Z^2$-periodic potential $V \in L^{\infty}(\rr^2)$, in the two-dimensional case. In all the following, the $d$-dimensional torus $\T^d$ is defined by $$\T^d := \R^d/2\pi\Z^d.$$
For the sake of notational simplicity, any function $f \in L^1(\T^d)$ will be identified to its $2\pi\mathbb Z^d$-periodic extension to $\rr^d$ belonging to $L^1_{\text{loc}}(\rr^d)$.
\medskip
\begin{theorem}\label{thm:obs_R2}
Assume $d =2$. Let $b \in L^1(\mathbb T^2)\setminus\{0\}$ be a non-negative function.

For every $T>0$ and compact subset $\mathcal K \subset L^{\infty}(\mathbb T^2)$, there exists a positive constant $C=C(b,T, \mathcal{K}) > 0$ such that for every $2\pi\mathbb Z^2$-periodic potential $V \in \mathcal{K}$ and every $u_0 \in L^2(\R^2)$, the solution $u$ of \eqref{eq:SchrodingerObs} satisfies
\begin{equation}
    \label{eq:ObservabilitySchrodingerRd}
    \norme{u_0}_{L^2(\R^2)}^2 \leq C \int_0^T \int_{\mathbb R^2} b(z)|u(t,z)|^2 dzdt.
\end{equation}
\end{theorem}
\medskip
Contrary to Taüffer's result, Theorem~\ref{thm:obs_R2} is limited to the two-dimensional case. However, it allows to consider more general observable functions $b$ and to deal with rough potentials $V$. Even if our proof borrows some ingredients of the proof given by Taüffer, such as the use of the Floquet-Bloch transform, the remainder of the proof is very different. In a nutshell, ours uses the notion of semiclassical measures, whereas Taüffer's proof follows from an Ingham's type inequality and explicit computations on the spectrum of the operator $-\Delta+2i\theta\cdot\nabla + |\theta|^2$ on $L^2(\mathbb T^d)$, based on previous results of \cite{Jaf90}, \cite{KL05}. In particular, its proof seems not to be easily adaptable to the case of Schrödinger equations with potential.

By the well-known Hilbert Uniqueness Method \cite[Theorem 2.42]{Cor07}, one can deduce from Theorem~\ref{thm:obs_R2} an exact controllability result for the Schrödinger equation
\begin{equation}
	\label{eq:SchrodingerControl}
		\left\{
			\begin{array}{ll}
				i  \partial_t y  = (-\Delta + V(x)) y + h 1_{\omega} & \text{ in }  (0,+\infty) \times \R^d, \\
				y(0, \cdot) = y_0 & \text{ in } \R^d.
			\end{array}
		\right.
\end{equation}
In \eqref{eq:SchrodingerControl}, at time $t \in [0,+\infty)$, $y(t, \cdot) : \R^d \to \C$ is the state and $h(t,\cdot) : \omega \to \C$ is the control.
\medskip
\begin{corollary}
Assume $d =2 $.
For every non-empty $2\pi\mathbb Z^2$-periodic measurable subset $\omega \subset \rr^2$, $T>0$ and compact subset $\mathcal K \subset L^{\infty}(\mathbb T^2)$, there exists a positive constant $C=C(\omega,T, \mathcal K)>0$ such that for every $V \in \mathcal K$ and $y_1 \in L^2(\R^2)$, there exists a control $h \in L^2(0,T;L^2(\omega))$ satisfying
\begin{equation*}
    \label{eq:CostSchrodinger}
    \norme{h}_{L^2(0,T;L^2(\omega))} \leq C \norme{y_1}_{L^2(\R^2)},
\end{equation*}
and such that the solution $y$ of \eqref{eq:SchrodingerControl} with $y_0 = 0$ satisfies
\begin{equation*}
    \label{eq:exactcontrol}
    y(T,\cdot) = y_1.
\end{equation*}
\end{corollary}
\medskip

A key ingredient in the proof of Theorem~\ref{thm:obs_R2} is borrowed from \cite{Tau22} and consists in applying the Floquet-Bloch transform, which is introduced in Section~\ref{section:floquet_bloch}. As detailed in Section~\ref{section:floquet_bloch}, performing the Floquet-Bloch transform reduces the study of \eqref{eq:SchrodingerControl} to a family of Schrödinger equations posed on the torus $\mathbb T^d$:
\begin{equation}
	\label{eq:Schrodinger_torusObs}\tag{$E_{\theta}$}
		\left\{
			\begin{array}{ll}
				i  \partial_t u  = (-\Delta+2i\theta \cdot \nabla + |\theta|^2+ V(x)) u & \text{ in }  (0,+\infty) \times \T^d, \\
				u(0, \cdot) = u_0 & \text{ in } \T^d,
			\end{array}
		\right.
\end{equation}
with $\theta \in [0,1]^d$. The observability of \eqref{eq:Schrodinger_torusObs} with $\theta=0$ has been widely studied over the last two decades. In \cite{AM14}, the observability of \eqref{eq:Schrodinger_torusObs} is shown to hold from any open subset $\omega \subset \mathbb T^d$ in any time $T>0$ when $\theta=0$ and $V$ belongs to a class of potentials slightly larger than the class of continuous potentials. In the two-dimensional setting, the authors of \cite{BBZ13} established that the same result holds true for $V \in L^2(\mathbb T^2)$ but still from open subset. More recently, it has been shown that the regularity assumption on $b$ can also be relaxed. Indeed, the main result of \cite{BZ19} ensures the observability of \eqref{eq:Schrodinger_torusObs} (with $\theta=0$ and $V=0$) as soon as $b \in L^2(\T^2)$. Regarding the equations \eqref{eq:Schrodinger_torusObs} with non-trivial $\theta$, let us mention that an observability result for the one-dimensional case is given by \cite[Proposition~3.1]{BBZ13} for $V \in L^p(\T^1) $ with $p>1$ and $b=\un_{\omega}$ where $\omega\subset \T^1$ is a non-empty open subset.

Our second main result ensures the observability of the Schrödinger equations \eqref{eq:Schrodinger_torusObs}. It is obtained as a by product of the proof of Theorem~\ref{thm:obs_R2}.
\medskip
\begin{theorem}\label{obs_tore_thm}
Let $b \in L^1(\mathbb T^2)\setminus\{0\}$ be a non-negative function, $T>0$ and $\mathcal K\subset L^{\infty}(\T^2)$ be a compact subset. There exists a positive constant $C=C(b,T,\mathcal K) > 0$ such that for every potential $V \in  \mathcal K$, $\theta \in [0,1]^2$ and $v_0 \in L^2(\T^2)$, the solution $v$ of \eqref{eq:Schrodinger_torusObs} satisfies
\begin{equation}
    \label{eq:ObservabilitySchrodingerTore}
    \norme{v_0}_{L^2(\T^2)}^2 \leq C \int_0^T \int_{\mathbb T^2} b(z) |v(t,z)|^2 dz dt.
\end{equation}
\end{theorem}
\medskip
Even if our proof closely follows the methodology introduced in \cite{BZ12}, \cite{BBZ13} and \cite{BZ19}, some new difficulties appear. Indeed, a key ingredient is the establishment of uniform Strichartz estimates for the equation \eqref{eq:Schrodinger_torusObs}. Even for fixed $\theta \in [0,1]^2$, as mentioned in Remark \ref{rmk:strichartztheta} below, Strichartz estimates cannot be obtained from the usual Zygmund inequality \cite{Zyg74}. This is why we need to first get uniform resolvent estimates for the operator $-\Delta+2i\theta \cdot \nabla + |\theta|^2+ V(x)$ in the spirit of \cite{BBZ13}. On the other hand, the use of semi-classical defect measures for proving the observability estimates \eqref{eq:ObservabilitySchrodingerTore} has to be performed by keeping track of the dependence of the parameter $\theta$ in all the procedure.
\subsection{Organization of the paper}
Our article is organised as follows: Section~\ref{section:floquet_obs} introduces the Floquet-Bloch transform and aims at showing that uniform observability estimates of \eqref{eq:Schrodinger_torusObs} with respect to $\theta$ leads to an observability estimate for \eqref{eq:SchrodingerObs}. Section~\ref{sec:propsemigroupHthetaV} is devoted to establish useful properties of the group generated by $\mathcal H_{\theta, V}=-\Delta +2i \theta\cdot \nabla+|\theta|^2+V$, such as resolvent estimates and Strichartz estimates. Section \ref{sec:refobsmultid} consists in establishing observability inequalities for $L^2$ initial data from observability inequalities for highly oscillating initial data. In Section~\ref{section:proof_obs_torus}, the proof of Theorem~\ref{obs_tore_thm} is presented. It uses the notion of semiclassical defect measure and follows the strategy developed by the authors of \cite{BZ12}, \cite{BBZ13} and \cite{BZ19}. Few facts about semiclassical analysis and semiclassical defect measures are recalled in the Appendix, in Section~\ref{section:appendix}.

\section{Proof of the observability inequality on $\R^2$}\label{section:floquet_obs}

This section aims at proving that Theorem~\ref{thm:obs_R2} is a consequence of uniform observability estimates for the family of Schrödinger equations \eqref{eq:Schrodinger_torusObs} posed on $\T^d$. The proof relies on the Floquet-Bloch transform which is presented in Section~\ref{section:floquet_bloch}. 

\subsection{The Floquet-Bloch transform}\label{section:floquet_bloch}

In this part, we give a definition and few facts about the Floquet-Bloch transform. This tool is instrumental in the proof of Theorem~\ref{thm:obs_R2}. We follow the presentation of \cite[Section 4]{Kuc93}.

Let us first introduce the definition of the Floquet transform.
\begin{definition}
Let $\mathcal F : L^2(\R^d) \to L^2([0,2\pi]^d \times [0,1]^d)$ the Floquet transform
\begin{equation*}
    \mathcal F u(y , \theta) = \sum_{k \in \Z^d} e^{2i\pi\theta \cdot k} u(y+2 \pi k)\qquad \forall u \in L^2(\R^d),\ \forall (y,\theta) \in [0,2\pi]^d \times [0,1]^d.
\end{equation*}
\end{definition}
The first proposition ensures that the Floquet-Bloch transform is an isometry from $L^2(\rr^d)$ to $L^2([0,2\pi]^d \times [0,1]^d)$.
\begin{proposition}\label{prop:floquet_bloch_isometry}
The map $\mathcal F$ is an isometric isomorphism from $L^2(\rr^d)$ to $L^2([0,2\pi]^d \times [0,1]^d)$.
\end{proposition}

In this paper, we will use a slightly modified Floquet-Bloch transform, denoted by $\tilde{\mathcal F}$, defined by 
\begin{equation*}\label{eq:floquet_bloch_tilde}
    \tilde{\mathcal F}u(y,\theta)= e^{i\theta\cdot y} \mathcal Fu(y,\theta)\qquad  \forall u \in L^2(\rr^d),\ \forall (y, \theta) \in [0,2\pi]^d\times[0,1]^d.
\end{equation*}
It is clear that $\tilde{\mathcal F}$ still defines an isometric isomorphism from $L^2(\rr^d)$ to $L^2([0,2\pi]^d\times [0,1]^d)$. The main advantage of taking this definition comes from the following result:
\begin{proposition}\label{prop:Floquet_laplacian}
Let $k \in \nn$. For all $u \in H^k(\rr^d)$ and for almost all $\theta \in [0,1]^d$, we have $\tilde{\mathcal F}(u)(\cdot, \theta) \in H^k(\mathbb T^d)$. 
Moreover, $$\tilde{\mathcal F}(-\Delta u)(\cdot, \theta)= (-\Delta+2i\theta\cdot\nabla +|\theta|^2)\tilde{\mathcal F}u(\cdot, \theta)\qquad \forall u \in H^2(\mathbb R^d),\ \forall \theta \in [0,1]^d.$$
\end{proposition}
In particular, it appears from Proposition~\ref{prop:Floquet_laplacian} that $\tilde{\mathcal F}u$ enjoys some periodicity property in the $y$-variable, contrary to $\mathcal Fu$.
An other elementary result is the following proposition which ensures that the Floquet-Bloch transform commutes with any periodic function:
\begin{proposition}\label{prop:commutationperiodic}
Let $V \in L^{\infty}(\rr^d)$ be a $2\pi\mathbb Z^d$-periodic function. Then,  $$ \tilde{\mathcal F}(Vu)=V \tilde{\mathcal F}u\qquad \forall u \in L^2(\rr^d).$$
\end{proposition}
Combining Proposition~\ref{prop:Floquet_laplacian} and \ref{prop:commutationperiodic}
leads to, for all $V \in L^{\infty}(\T^d)$, $u \in H^2(\rr^d)$ and for almost all $\theta \in [0,1]^d$,
$$\tilde{\mathcal F}(-\Delta u+Vu)(\cdot, \theta)= (-\Delta+2i\theta\cdot\nabla +|\theta|^2+V)\tilde{\mathcal F}u(\cdot, \theta).$$
Consequently, for $V \in L^{\infty}(\T^d)$, we obtain that for all $u \in L^2(\rr^d)$ and for almost all $\theta \in [0,1]^d$,
\begin{equation}\label{eq:group_floquet_bloch}
\tilde{\mathcal F}(e^{it(-\Delta+V)}u)(\cdot, \theta)= e^{it(-\Delta+2i\theta\cdot\nabla +|\theta|^2+V)}\tilde{\mathcal F}u(\cdot, \theta)\qquad \forall t \in \R,
\end{equation}
where $\left(e^{it(-\Delta+2i\theta\cdot\nabla +|\theta|^2+V)}\right)_{t \in \rr}$ is the one-parameter group generated by the self-adjoint operator $$-\Delta+2i\theta\cdot\nabla +|\theta|^2+V : H^2(\T^d) \longrightarrow L^2(\T^d).$$

\subsection{From the uniform observability inequality on $\T^2$ to the observability inequality on $\R^2$}


This part is devoted to establish that Theorem~\ref{thm:obs_R2} can be deduced from Theorem~\ref{obs_tore_thm}.

In the rest of the paper, we will use the notation
\begin{equation*}
    \label{eq:definitionopHtheta}
    \mathcal{H}_{\theta,V} v := (- \Delta + 2 i \theta \cdot \nabla + |\theta|^2 + V) v\qquad \forall v \in H^2(\T^d),
\end{equation*}
for $V \in L^{\infty}(\T^d)$ and $\theta \in [0,1]^d$. The following proposition shows that uniform observability inequalities for Schrödinger equations \eqref{eq:Schrodinger_torusObs} directly imply  observability inequalities for Schrödinger equations posed on the Euclidean space.

\begin{proposition}\label{prop:from_torus_to_Rd}
Let $V \in L^{\infty}(\T^d)$, $b \in L^1(\T^d)$, $T>0$ and $C>0$. If for all $\theta\in [0,1]^d$ and $v_0 \in L^2(\T^d)$, 
$$\|v_0\|^2_{L^2(\T^d)} \leq C \int_0^T \int_{\T^d} b(z) \left|e^{-it \mathcal H_{\theta, V}}v_0(z)\right|^2 dz\, dt,$$
then for all $u_0 \in L^2(\rr^d)$,
$$\|u_0\|^2_{L^2(\rr^d)} \leq C \int_0^T \int_{\rr^d} b(z) \left|e^{-it(-\Delta+V)}u_0(z)\right|^2 dz \, dt.$$
\end{proposition}

In the Section~\ref{section:proof_obs_torus}, we give the proof of Theorem \ref{obs_tore_thm} which provides uniform observability estimates with respect to $\theta \in [0,1]^d$ for the Schrödinger equations \eqref{eq:Schrodinger_torusObs}, in the two-dimensional case. Thanks to Proposition~\ref{prop:from_torus_to_Rd}, Theorem~\ref{thm:obs_R2} appears as a consequence of Theorem~\ref{obs_tore_thm}.

\begin{proof}[Proof of Proposition~\ref{prop:from_torus_to_Rd}]
Let $u_0 \in L^2(\R^d)$ and $u$ be the solution of \eqref{eq:SchrodingerObs} associated with $u_0$. Proposition \ref{prop:Floquet_laplacian} and \eqref{eq:group_floquet_bloch} are instrumental in this proof.

First, thanks to the isometry property, the left hand side of \eqref{eq:SchrodingerObs} becomes
\begin{equation}
    \label{eq:lefthandsideObsRd}
    \norme{u_0}_{L^2(\R^d)}^2 = \norme{\tilde{\mathcal{F}}u_0}_{L^2([0,1]^d \times \T^d)}^2 =  \int_{[0,1]^d} \int_{\T^d} | \tilde{\mathcal{F}}u_0 (y, \theta)|^2 dy\, d\theta
\end{equation}

Secondly, for the right hand side of \eqref{eq:SchrodingerObs}, we have, thanks to Proposition~\ref{prop:floquet_bloch_isometry} and \ref{prop:commutationperiodic},
\begin{align}
\int_{0}^T \int_{\mathbb R^d} b(z)|u(t,z)|^2 dz\, dt & =\int_{0}^T \langle \tilde{\mathcal{F}} (e^{-i t (-\Delta+ V(x))} u_0), \tilde{\mathcal{F}} (b \cdot e^{-i t (-\Delta+ V(x))} u_0 )\rangle_{L^2([0,1]^d \times \T^d)} dt\notag\\
     & =  \int_{0}^T \int_{[0,1]^d}\int_{\T^d} b(z) |\tilde{\mathcal{F}} ( e^{-  i t (-\Delta + V)} u_0)(z, \theta)|^2 d\theta \, dz\, dt. \notag
\end{align}     
It therefore follows from \eqref{eq:group_floquet_bloch} and the previous lines that
\begin{equation}
   \int_{0}^T \int_{\mathbb R^d} b(z)|u(t,z)|^2 dz\, dt  = \int_{[0,1]^d} \left(\int_{0}^T \int_{\T^d} b(z) |e^{-  i t \mathcal H_{\theta, V}} \tilde{\mathcal{F}}u_0(z, \theta)|^2 dz \, dt\right) d\theta \label{eq:righthandsideObsRd}
\end{equation}
On the other hand, by assumptions, we have for almost all $\theta \in [0,1]^d$,
\begin{equation}\label{eq:ObsUniformApplication}
\left\|\tilde{\mathcal F}u_0(\cdot, \theta)\right\|^2_{L^2(\T^d)} \leq C \int_0^T \int_{\T^d} b(z) \left|e^{-it\mathcal H_{\theta, V}}\tilde{\mathcal F}u_0(z, \theta)\right|^2 dz \, dt.
\end{equation}
By gathering \eqref{eq:lefthandsideObsRd}, \eqref{eq:righthandsideObsRd} and \eqref{eq:ObsUniformApplication}, we finally obtain the expected observability inequality \eqref{eq:ObservabilitySchrodingerRd}.
\end{proof}

\section{Properties of the group generated by $\mathcal{H}_{\theta,V}$}
\label{sec:propsemigroupHthetaV}

The goal of this section is to first derive resolvent estimates for $\mathcal{H}_{\theta,V}$, then deduce a priori estimates and stability results for solutions to the associated semi-group $(e^{-it \mathcal{H}_{\theta,V}})_{t \geq 0}$. In both parts, we separate the available results in the multi-dimensional case to the specific two-dimensional results.

\subsection{Resolvent estimates for $\mathcal{H}_{\theta,V}$}

In this part, we first establish standard $L^2$ resolvent estimates for $\mathcal{H}_{\theta,V}$ in $\T^d$. In the second subsection, we prove $L^4$ resolvent estimates for $\mathcal{H}_{\theta,V}$ in $\T^2$. This result is specific to the two-dimensional case and is instrumental in the proof of the Strichartz estimates provided by Proposition~\ref{prop:StrichartzEstimate}. 

\subsubsection{$L^2$ resolvent estimates in $\T^d$}

The first result concerns spectral properties of the operator $\mathcal{H}_{\theta,V}$. As it is standard, the proof is omitted.
\begin{proposition}
For every $\theta \in [0,1]^d$ and $V \in L^{\infty}(\T^d)$, $\mathcal{H}_{\theta,V}$ is a self-adjoint, with compact resolvent, operator on $L^2(\T^d)$. Let $(\psi_{k,\theta,V})_{k \in \N}$ be the orthonormal basis of eigenfunctions and $(\lambda_{k,\theta,V})$ be the associated eigenvalues. For every $M>0$ and $s \in \R$, there exist two positive constants $C,C{'}>0$ such that for every $V \in L^{\infty}(\T^d)$ with $\norme{V}_{L^{\infty}(\T^d)} \leq M$ and $\theta \in [0,1]^d$, we have
\begin{equation}
\label{eq:uhshthetaV}
C \norme{u}_{H^s(\T^d)}^2 \leq \sum_{k=0}^{+\infty} (1+|\lambda_{k,\theta,V}|^2)^{s/2} |u_k|^2 \leq C{'} \norme{u}_{H^s(\T^d)}^2,\ \quad \forall u = \sum_{k=0}^{+\infty} u_k \psi_{k,\theta,V} \in H^s(\T^d).
\end{equation}
\end{proposition}

The next result is about stability properties with respect to parameters of the resolvent of $\mathcal{H}_{\theta,V}$.
\begin{proposition}
\label{prop:convresolvent}
Assume that $\theta_n \to \theta$ in $[0,1]^d$ and $V_n \rightharpoonup^\star V$ in $L^{\infty}(\T^d)$ as $n \to +\infty$, then for every $\lambda \in \C$ such that $\Re(\lambda) > 0$, the following convergence holds
\begin{equation*}
    \label{eq:convergenceresolventwithouti}
   \norme{( \lambda I + i \mathcal{H}_{\theta_n,V_n})^{-1} f - (\lambda I + i\mathcal{H}_{\theta,V})^{-1} f}_{L^2(\T^2)} \underset{n \to +\infty}{\to} 0,\qquad \forall f \in L^2(\T^d).
\end{equation*}
\end{proposition}
\begin{proof}
The proof combines an energy estimate coupled with the Rellich theorem. We start from 
\begin{equation}
    \label{eq:resolventtrotterKaton}
    (-i \lambda I -\Delta + 2 i \theta_n \cdot \nabla + |\theta_n|^2 + V_n)u_n=-if\ \text{in}\ \T^d.
\end{equation}

After multiplying \eqref{eq:resolventtrotterKaton} by $\overline{u}_n$, integrating on $\T^d$ and performing an integration by parts, we obtain
\begin{equation}\label{eq:ipp_resolvent}
(-i \lambda+|\theta_n|^2) \|u_n\|^2_{L^2}+\|\nabla u_n\|^2_{L^2} +2 \langle i\theta_n \cdot \nabla u_n, u_n\rangle_{L^2} +\langle V_n u_n, u_n \rangle_{L^2}\\=-i \langle f, u_n \rangle_{L^2}.
\end{equation}
By taking the real and imaginary parts, it follows from \eqref{eq:ipp_resolvent} and Young's inequalities that there exists a positive constant $C>0$, depending on $\lambda$ and $\sup_{n \in \N} \norme{V_n}_{L^{\infty}(\T^d)}$, such that
$$ \norme{\nabla u_n}_{L^2(\T^d)}^2 \leq C \norme{f}_{L^2(\T^2)}^2 + C \norme{u_n}_{L^2(\T^2)}^2\ \text{and}\ \norme{ u_n}_{L^2(\T^d)}^2 \leq C \norme{f}_{L^2(\T^2)}^2.$$
This readily implies that
\begin{align*}
 \norme{u_n}_{H^1(\T^d)}^2 &\leq C \norme{f}_{L^2(\T^d)}^2.
\end{align*}
By the Rellich theorem, we have that there exists $v \in H^1(\T^d)$ such that, up to a subsequence,
\begin{equation*}
    u_n \rightharpoonup v\ \text{in}\ H^1(\T^d),\ u_n \to v\ \text{in}\ L^2(\T^d)\ \text{as}\ n \to +\infty.
\end{equation*} 
After multiplying \eqref{eq:resolventtrotterKaton} by $i \bar{\varphi} \in H^1(\T^d)$ and integrating by parts, we pass to the limit as $n \to +\infty$
\begin{equation*}
    \int_{\T^d} \left( \lambda v \bar{\varphi} + i \left(\nabla v \cdot \nabla \bar{\varphi} + 2 i \theta \cdot \nabla v \bar{\varphi} + |\theta|^2 v \bar{\varphi} + V v \bar{\varphi} \right)\right)dx= \int_{\T^d} f \bar{\varphi}\qquad \forall \varphi \in H^1(\T^d).
\end{equation*}
By uniqueness, we have that $v = (\lambda I + i\mathcal{H}_{\theta,V})^{-1} f$, then $(u_n)_{n \in \N}$ admits a unique accumulation point $(\lambda I + i\mathcal{H}_{\theta,V})^{-1} f \in H^1(\T^d)$, which concludes the proof.
\end{proof}

\subsubsection{$L^4$ resolvent estimates in $\T^2$}

The main result of this part is a $L^4$ resolvent estimate for $\mathcal{H}_{\theta,V}$ in $\T^2$.
\begin{proposition}
\label{prop:resolventV}
For every $M>0$, there exists $C>0$ such that for every $\theta \in [0,1]^2$, $V \in L^{\infty}(\T^2)$, $\norme{V}_{L^{\infty}(\T^2)} \leq M$, $f \in L^{4/3}(\T^2)$ and $\tau \in \C$ with $|\Im(\tau)| \geq 1$, 
\begin{equation}
    \label{eq:resolventthetaV}
    \norme{(-\Delta + 2 i \theta \cdot \nabla + |\theta|^2 + V - \tau )^{-1} f}_{L^4(\T^2)} \leq C \norme{f}_{L^{4/3}(\T^2)}.
\end{equation}
\end{proposition}
Proposition \ref{prop:resolventV} is an adaptation of \cite[Proposition 2.6]{BBZ13}. Two main differences appear. First, due to the presence of the parameter $\theta$ in the operator $\mathcal{H}_{\theta,V}$, we need to keep track of the independence of the constants, with respect to $\theta$ to get a uniform constant $C$ in \eqref{eq:resolventthetaV}. The second difference is the assumption on the potential. While Bourgain, Burq, Zworski are considering potential living in a compact set of $L^2(\T^2)$, here we are focusing on potentials living in a ball of $L^{\infty}(\T^2)$.

The proof of Proposition \ref{prop:resolventV} is crucially based on the following result which is Proposition \ref{prop:resolventV} in the particular case when $V=0$.
\begin{proposition}
\label{prop:resolvent}
There exists $C>0$ such that for every $\theta \in [0,1]^2$, $f \in L^{4/3}(\T^2)$ and $\tau \in \C$ with $|\Im(\tau)| \geq 1$, 
\begin{equation}
    \label{eq:resolventtheta}
    \norme{(-\Delta + 2 i \theta \cdot \nabla + |\theta|^2- \tau )^{-1} f}_{L^4(\T^2)} \leq C \norme{f}_{L^{4/3}(\T^2)}.
\end{equation}
\end{proposition}

\begin{proof}[Proof of Proposition \ref{prop:resolventV} from Proposition \ref{prop:resolvent}]
We start from $$(-\Delta + 2 i \theta \cdot \nabla + |\theta|^2 + V- \tau )u=f\ \text{in}\ \T^2.$$
Since $V$ is real-valued, after multiplying by $\bar{u}$ and integrating on $\T^2$, we obtain by taking the imaginary part and thanks to Hölder's inequality
\begin{align}
    |\Im(\tau)| \norme{u}_{L^2(\T^2)}^2 &\leq  \norme{u}_{L^4(\T^2)} \norme{f}_{L^{4/3}(\T^2)}.   \label{eq:estimateimpartV}
\end{align}
Since $|\Im(\tau)| \geq 1$, we get from \eqref{eq:estimateimpartV} that
\begin{equation*}
\norme{u}_{L^2(\T^2)}^2 \leq \norme{u}_{L^4(\T^2)} \norme{f}_{L^{4/3}(\T^2)}.
\end{equation*}
On the other hand, we also have
$$(-\Delta + 2 i \theta \cdot \nabla + |\theta|^2 - \tau )u=f - V u\ \text{in}\ \T^2.$$
So applying the resolvent estimate \eqref{eq:resolventtheta}, we get
$$ \norme{u}_{L^4(\T^2)} \leq C \norme{f}_{L^{4/3}(\T^2)} + C \norme{V u}_{L^{4/3}(\T^2)}.$$
By plugging the $L^2$-estimate on $u$ in the previous formula, using that $V \in L^{\infty}(\T^2)$ and performing Young's estimate, we get the expected result \eqref{eq:resolventthetaV}.
\end{proof}

All the end of this part is then devoted to the proof of Proposition \ref{prop:resolvent}.

The next result is a refinement of the Zygmund's inequality for the operator $\mathcal{H}_{\theta, 0}$ i.e. there exists $C>0$ such that for every $\theta \in [0,1]^2$ and $\lambda>0$, we have
\begin{equation}
    \label{eq:zygmund}
    \norme{\sum_{|n-\theta|^2 = \lambda} c_n e^{i n \cdot x}}_{L^4(\T^2)} \leq C \left(\sum_{|n-\theta|^2 = \lambda} |c_n|^2\right)^{1/2}.
\end{equation}
Note that inequality \eqref{eq:zygmund} comes from a straightforward adaptation of \cite{Zyg74}. 
\begin{proposition}
\label{prop:FlouZygmund}
There exists $C>0$ such that for all $\theta \in [0,1]^2$, $\kappa \geq 0$, $0 < h \leq 1$ and $u = \sum_{n \in \Z} \widehat{u}(n) e^{i n \cdot x}  \in L^2(\T^2)$ satisfying
\begin{equation*}
    \widehat{u}(n) = 0\ \text{for}\ n \notin \mathcal{B}_{\theta}(\kappa, h) := \{n \in \Z^2\ ;\ \left| h^2|n-\theta|^2 - 1 \right| \leq \kappa^2 h^2\},
\end{equation*}
we have
\begin{align*}
\norme{u}_{L^4(\T^2)} \leq
\left\{
    \begin{array}{ll}
        C (1+ \kappa)^{1/4} (1+\kappa^2h)^{1/4} \norme{u}_{L^2(\T^2)} & \mbox{if } \kappa \leq h^{-1} \\
        C (1+\kappa)^{1/2} \norme{u}_{L^2(\T^2)} & \mbox{if }\kappa \geq h^{-1}
    \end{array}
\right.
\end{align*}
\end{proposition}
\begin{remark}
It is worth mentioning that $\kappa = 0$ is simply the Zygmund's inequality \eqref{eq:zygmund}, while the other regimes have to be treated by different arguments, that are Sobolev embeddings for $\kappa \geq h^{-1}$ and an arithmetic proof of Sogge's estimate for spectral projectors for $\kappa \leq h^{-1}$.
\end{remark}

The proof that we give below is an adaptation of \cite[Proposition 2.4]{BBZ13}.
\begin{proof}
In the following proof, the constants $C>0$ that appear can vary from line to line but do not depend on $\theta$.

First, we have $\mathcal{B}_{\theta}(\kappa_1, h) \subset \mathcal{B}_{\theta}(\kappa_2, h)$ when $\kappa_1 \leq \kappa_2$, so one can safely assume that $\kappa \geq C$ where $C\geq 1$ is a positive numerical constant.

For a constant $0<\delta \leq 1$ that will be fixed later, we distinguish two regimes: $\kappa h \geq \delta$, and $\kappa h \leq \delta$. 

\textit{First regime: $\kappa h \geq \delta$.} The estimate comes from the Sobolev embedding $H^{1/2}(\T^2) \hookrightarrow L^4(\T^2)$ because $\widehat{u}(n) = 0$ unless $|n|^2 \leq |\theta|^2 + h^{-2} + \kappa^2 \leq 2+ (\delta^{-2} + 1)\kappa^2$, so taking $\kappa \geq \sqrt 2$, this implies
\begin{equation*}
    \norme{u}_{H^{1/2}(\T^2)} \leq C_{\delta} \kappa^{1/2} \norme{u}_{L^2(\T^2)}.
\end{equation*}

\textit{Second regime: $h \kappa \leq \delta$.} First observe that $\mathcal{B}_{\theta}(\kappa, h) \subset \mathcal{A}_{\theta}(\kappa, h)$ where
\begin{equation*}
    \mathcal{A}_{\theta}(\kappa, h) := \{n \in \Z^2\ ;\ \left| h |n-\theta| -1 \right| \leq \kappa^2 h^2\}.
\end{equation*}
Indeed, using $h \kappa \leq \delta \leq 1$, we have
\begin{align*}
    n \in \mathcal{B}_{\theta}(\kappa, h) \Rightarrow -\kappa^2 h^2 + 1 \leq h^2 |n-\theta|^2 \leq \kappa^2 h^2 + 1 \Rightarrow \sqrt{1-\kappa^2 h^2 } \leq h |n-\theta| \leq \sqrt{1+\kappa^2 h^2}\\
    \Rightarrow 1-\kappa^2 h^2 \leq h |n-\theta| \leq 1+\kappa^2 h^2 \Rightarrow n \in \mathcal{A}_{\theta}(\kappa, h).
\end{align*}

Note that we have
\begin{align*}
    \C &= \{z \in \C\ ;\ \Re (z-\theta) \geq 0,\ \Im (z-\theta) \geq 0\} \cup \{z \in \C\ ;\ \Re (z-\theta) \leq 0,\ \Im (z-\theta) \geq 0\}\\ &\qquad \cup \{z \in \C\ ;\ \Re (z-\theta) \leq 0,\ \Im (z-\theta) \leq 0\}  \cup \{z \in \C\ ;\ \Re (z-\theta) \geq 0,\ \Im (z-\theta) \leq 0\}\\
    & = \mathcal{C}^{\theta}_{++} \cup \mathcal{C}^{\theta}_{-+} \cup \mathcal{C}^{\theta}_{--} \cup \mathcal{C}^{\theta}_{+-}.
\end{align*}
We will only consider the situation where
\begin{equation*}
    u = \sum_{n \in \Z^2} u_n e^{i n \cdot x} = \sum_{n \in \Z^2 \cap \mathcal{C}^{\theta}_{++}} u_n e^{i n \cdot x}.
\end{equation*}
The general case easily follows.

We first introduce
\begin{equation*}
     \mathcal{A}_{\theta}(\kappa, h) = \bigcup_{\alpha=0}^{N_{\kappa,h}} \left(\Z^2 \cap \mathcal{A}_{\theta,\alpha}(\kappa, h)\right),\ \text{with}\ N_{\kappa,h} := \left\lfloor \frac{\pi}{2 hk}\right\rfloor,
\end{equation*}
where
\begin{equation*}
     \mathcal{A}_{\theta,\alpha}(\kappa, h) := \Big\{z \in \mathcal{C}_{++}^{\theta}\ ;\ |h|z-\theta| - 1 | \leq \kappa^2 h^2,\ \text{arg}(z-\theta) \in [\alpha h \kappa, (\alpha+1) h \kappa)\Big\}.
\end{equation*}
Then the proof relies on the following geometric lemma, that is a Corollary of \cite[Lemma 2.5]{BBZ13}.
\begin{lemma}
\label{lemma:arithmetic}
Fix $\delta > 0$ small enough. Then there exists $Q \in \N$ such that for all $\theta\in [0,1]^2$, $0 < h < 1$, $1 \leq \kappa \leq \delta/h$ and $\alpha, \beta, \alpha',\beta' \in \{0,1, \dots, N_{\kappa,h}\}$, if
\begin{equation}
\label{eq:conditionalphabetaetc}
    ( \mathcal{A}_{\theta,\alpha}(\kappa, h) + \mathcal{A}_{\theta,\beta}(\kappa, h)) \cap (( \mathcal{A}_{\theta,\alpha'}(\kappa, h) + \mathcal{A}_{\theta,\beta'}(\kappa, h)) \neq \emptyset,
\end{equation}
then
\begin{equation}
\label{eq:conditionalphabetaetcBis}
|\alpha - \alpha'| + |\beta - \beta'| \leq Q\ \text{or}\ |\alpha- \beta'| + |\beta - \alpha'| \leq Q.
\end{equation}
\end{lemma}
\begin{proof}[Proof of Lemma \ref{lemma:arithmetic}.]
From \cite[Lemma 2.5]{BBZ13}, we have that there exists $Q \in \N$ such that for any $0 < h < 1$ and any $1 \leq \kappa \leq \delta/h$, for every $\alpha, \beta, \alpha',\beta' \in \{0,1, \dots, N_{\kappa,h}\}$, if \eqref{eq:conditionalphabetaetc} holds with $\theta=0$ then \eqref{eq:conditionalphabetaetcBis} holds. Therefore, fixing $\Theta \in \T^2$ by remarking that $\mathcal{A}_{\Theta,\alpha}(\kappa, h)=F_{\Theta}(\mathcal{A}_{0,\alpha}(\kappa, h))$ with $F_{\Theta}(z) = z+\Theta$ so if \eqref{eq:conditionalphabetaetc} holds for $\theta=\Theta$ then \eqref{eq:conditionalphabetaetc} holds for $\theta=0$ so \eqref{eq:conditionalphabetaetcBis} holds. This concludes the proof.
\end{proof}
By decomposing as follows
\begin{equation}
\label{eq:decompositionu}
    u = \sum_{\alpha=0}^{N_{\kappa,h}} U_{\alpha}, \quad \text{with} \quad U_{\alpha} := \sum_{n \in \Z^2 \cap\mathcal{A}_{\theta,\alpha}(\kappa, h) } u_n e^{i n\cdot x},
\end{equation}
we obtain
\begin{equation}
\label{eq:developpementuL4}
    \norme{u}_{L^4(\T^2)}^4 = \norme{u^2}_{L^2(\T^2)}^2 = \int_{\T^2} u^2 \overline{u^2} dx =  \sum_{\alpha, \beta,\alpha',\beta' = 0}^{N_{\alpha,h}} \int_{\T^2} U_{\alpha} U_{\beta} \overline{U_{\alpha'}}\ \overline{U_{\beta'}} dx.
\end{equation}
Moreover, for $\alpha, \beta, \alpha',\beta' \in \{0,1, \dots, N_{\kappa,h}\}$, we have 
$$\int_{\T^2} U_{\alpha} U_{\beta} \overline{U_{\alpha'}}\ \overline{U_{\beta'}} dx
=\sum_{n \in \Z^2 \cap \mathcal A_{\theta, \alpha}} \sum_{m \in \Z^2 \cap\mathcal A_{\theta, \beta}}\sum_{p \in \Z^2 \cap\mathcal A_{\theta, \alpha'}}\sum_{q \in \Z^2 \cap \mathcal A_{\theta, \beta'}}u_nu_m\overline{u_p}\overline{u_q} \int_{\T^2} e^{ix\cdot(n+m-p-q)}dx.$$
In particular, if \eqref{eq:conditionalphabetaetc} does not hold, then $\int_{\T^2} U_{\alpha} U_{\beta} \overline{U_{\alpha'}}\ \overline{U_{\beta'}} dx=0$. Hence, from Lemma \ref{lemma:arithmetic}, we can restrict the sum in \eqref{eq:developpementuL4} to the subset of indices $(\alpha,\beta,\alpha',\beta')$ satisfying \eqref{eq:conditionalphabetaetcBis}, i.e.
\begin{equation}
    \label{eq:developpementuL4Mieux}
    \norme{u}_{L^4(\T^2)}^4 = \sum_{\alpha, \beta,\alpha',\beta'\ \text{satisfying}\ \eqref{eq:conditionalphabetaetcBis}} \int_{\T^2} U_{\alpha} U_{\beta} \overline{U_{\alpha'}}\ \overline{U_{\beta'}} dx
\end{equation}

For $(\alpha,\beta,\alpha',\beta')$ satisfying \eqref{eq:conditionalphabetaetcBis}, Hölder's inequality gives that
\begin{align*}
    \left| \int_{\T^2} U_{\alpha} U_{\beta} \overline{U_{\alpha'}}\  \overline{U_{\beta'}} dx\right| &\leq \norme{U_{\alpha}}_{L^4(\T^2)}\norme{U_{\beta}}_{L^4(\T^2)} \norme{U_{\alpha'}}_{L^4(\T^2)} \norme{U_{\beta}}_{L^4(\T^2)},
    \end{align*}
so
\begin{equation*}
    \left| \int_{\T^2} U_{\alpha} U_{\beta} \overline{U_{\alpha'}}\  \overline{U_{\beta'}} dx\right| \leq ( \norme{U_{\alpha}}_{L^4(\T^2)}^2 + \norme{U_{\alpha'}}_{L^4(\T^2)}^2) ( \norme{U_{\beta}}_{L^4(\T^2)}^2 + \norme{U_{\beta'}}_{L^4(\T^2)}^2)\ \text{if}\ |\alpha - \alpha'| + |\beta - \beta'| \leq Q,
\end{equation*}
or
\begin{equation*}
    \left| \int_{\T^2} U_{\alpha} U_{\beta} \overline{U_{\alpha'}}\  \overline{U_{\beta'}} dx\right| \leq ( \norme{U_{\alpha}}_{L^4(\T^2)}^2 + \norme{U_{\alpha'}}_{L^4(\T^2)}^2) ( \norme{U_{\beta}}_{L^4(\T^2)}^2 + \norme{U_{\beta'}}_{L^4(\T^2)}^2)\ \text{if}\ |\alpha- \beta'| + |\beta - \alpha'| \leq Q.
\end{equation*}
Therefore, we have from \eqref{eq:developpementuL4Mieux},
\begin{equation}
\label{eq:esetimationuL4Pasmal}
    \norme{u}_{L^4(\T^2)} \leq C Q^2 \left( \sum_{\alpha=0}^{N_{\alpha,h}} \norme{U_{\alpha}}_{L^4(\T^2)}^2 \right)^2.
\end{equation}

Let us estimate $\norme{U_{\alpha}}_{L^4(\T^2)}^2$. By using Hölder's inequality then Cauchy-Schwarz inequality we get
\begin{align}
 \notag   \norme{U_{\alpha}}_{L^4(\T^2)} &\leq C \norme{U_{\alpha}}_{L^{\infty}(\T^2)}^{1/2} \norme{U_{\alpha}}_{L^2(\T^2)}^{1/2}\\
\notag  & \leq \left(\sum_{n \in \Z^2 \cap \mathcal{A}_{\theta,\alpha}(\kappa,h)} |u_n| \right)^{1/2}
  \left(\sum_{n \in \Z^2 \cap \mathcal{A}_{\theta,\alpha}(\kappa,h)} |u_n|^2 \right)^{1/4}\\
\norme{U_{\alpha}}_{L^4(\T^2)}  & \leq C | \Z^2 \cap \mathcal{A}_{\theta,\alpha}(\kappa,h)|^{1/4} \norme{U_{\alpha}}_{L^2(\T^2)}.
\label{eq:estimationUalpha}
\end{align}

We now need to bound the number of integral points in $\mathcal{A}_{\theta,\alpha}(\kappa,h)$. It is not difficult to see that $\mathcal{A}_{\theta,\alpha}(\kappa,h)$ is included in a rectangle of height $1+\kappa$ and width $1 + 3 \kappa^2 h$. Moreover, the number of integral points in any rectangle of height $H$ and width $W$ is bounded by $C \max(H,1) \max(W,1)$. Hence recalling $\kappa h \leq \delta$, we have
\begin{equation}
\label{eq:calculnombrepointsentiers}
    | \Z^2 \cap \mathcal{A}_{\theta,\alpha}(\kappa,h)| \leq C (1+ \kappa)(1+3 \kappa^2 h) \leq C (1+\kappa)^2.
\end{equation}
By gathering \eqref{eq:esetimationuL4Pasmal}, \eqref{eq:estimationUalpha}, \eqref{eq:calculnombrepointsentiers} and \eqref{eq:decompositionu}, we obtain
\begin{equation*}
\label{eq:upperrightcadrantestimation}
    \norme{u}_{L^4(\T^2)}^2 \leq C (1+ \kappa) (1+\kappa^2 h) \norme{u}_{L^2(\T^2)}^4.
\end{equation*}
This concludes the proof of Proposition \ref{prop:FlouZygmund}.
\end{proof}

From Proposition \ref{prop:FlouZygmund}, we can now prove the resolvent estimate of Proposition \ref{prop:resolvent}.

\begin{proof}
We split the proof into two cases.

\textit{First case: $\Re(\tau) \leq C$.} The proof combines an energy estimate coupled to a Sobolev embedding. We start from $$(-\Delta + 2 i \theta \cdot \nabla + |\theta|^2- \tau )u=f\ \text{in}\ \T^2.$$
After multiplying by $\overline{u}$, integrating on $\T^2$ and performing an integration by parts, we obtain
$$\frac 12 \|\nabla u\|^2_{L^2(\T^d)} +2i\langle \theta\cdot\nabla u, u\rangle_{L^2(\T^2)}+ (|\theta|^2-\tau)\|u\|^2_{L^2(\T^2)}=\langle f, u\rangle_{L^2(\T^2)}.$$
By applying Hödler's inequality and Young's inequality to the real and imaginary parts, we deduce that 
\begin{align}
   \label{eq:estimaterealpart} \norme{\nabla u}_{L^2(\T^2)}^2 + |\theta|^2 \norme{u}_{L^2(\T^2)}^2 - \Re(\tau) \norme{u}_{L^2(\T^2)}^2 &\leq C \norme{u}_{L^4(\T^2)} \norme{f}_{L^{4/3}(\T^2)} + C  \norme{u}_{L^2(\T^2)}^2,\\
    |\Im(\tau)| \norme{u}_{L^2(\T^2)}^2 &\leq  \norme{u}_{L^4(\T^2)} \norme{f}_{L^{4/3}(\T^2)}.   \label{eq:estimateimpart}
\end{align}
Since $|\Im(\tau)| \geq 1$, we get from \eqref{eq:estimateimpart} that
\begin{equation}
\label{eq:estimateimpartBis}
\norme{u}_{L^2(\T^2)}^2 \leq \norme{u}_{L^4(\T^2)} \norme{f}_{L^{4/3}(\T^2)}.
\end{equation}
Plugging \eqref{eq:estimateimpartBis} in \eqref{eq:estimaterealpart}, we obtain
\begin{equation*}
    \norme{u}_{H^1(\T^2)}^2 \leq C \norme{u}_{L^4(\T^2)} \norme{f}_{L^{4/3}(\T^2)}.
\end{equation*}
The Sobolev embedding $H^1(\T^2) \hookrightarrow L^4(\T^2)$ enables us to conclude.

\textit{Second case: $\Re\tau > C$.} First, let us define
$$ \forall f=\sum_{n\in \mathbb Z^2} f_n e^{i n \cdot x} \in L^2(\T^2), \quad P_{\theta,\tau}f := \sum_{n\in \mathbb Z^2} \frac{f_n}{(|n-\theta|^2 - \tau)^{1/2}} e^{i n \cdot x}.$$

Let us prove that 
\begin{equation}
\label{eq:boundedP}
   \forall f \in L^2(\T^2),\quad \norme{P_{\theta,\tau}f}_{L^4(\T^2)} \leq C \norme{f}_{L^2(\T^2)}.
\end{equation}
By calling $u=P_{\theta,\tau}f$, we decompose $u$ as follows
$$u = \underbrace{\sum_{||n-\theta|^2 - \Re(\tau)| < 1} \frac{f_n}{(|n-\theta|^2-\tau)^{-1/2}} e^{i n \cdot x}}_{u_0} +  \sum_{j=1}^{+\infty}\underbrace{\sum_{2^{j-1} \leq ||n-\theta|^2 - \Re(\tau)| < 2^j} \frac{f_n}{(|n-\theta|^2 - \tau)^{1/2}} e^{i n\cdot x}}_{u_j}.$$
From Proposition \ref{prop:FlouZygmund} with $\kappa=1$ and $h=1$, we get
\begin{equation*}
    \norme{u_0}_{L^4(\T^2)} \leq C \norme{u_0}_{L^2(\T^2)} \leq C \norme{f}_{L^2(\T^2)}.
\end{equation*}
For $1 \leq j < +\infty$, from Proposition \ref{prop:FlouZygmund} with $\kappa=2^{j/2}$, $h=(\Re(\tau))^{-1/2}$, we get
\begin{equation*}
    \norme{u_j}_{L^4(\T^2)} \leq C (1+\kappa)^{1/2} \leq C 2^{j/4} \norme{u_j}_{L^2(\T^2)}.
\end{equation*}
Therefore, by combining the three previous equations, we obtain
\begin{align*}
  \norme{u}_{L^4(\T^2)} &\leq \norme{u_0}_{L^4(\T^2)} +  \sum_{j=1}^{+\infty} \norme{u_j}_{L^4(\T^2)}\\
  &\leq C \norme{f}_{L^2(\T^2)} + C \sum_{j=1}^{+\infty} 2^{j/4} \norme{u_j}_{L^2(\T^2)} \\
   & \leq C \norme{f}_{L^2(\T^2)} + \left(\sum_{j=1}^{+\infty} 2^{-j/2}\right)^{1/2} \left(\sum_{j=1}^{+\infty} 2^j \sum_{2^{j-1} \leq ||n-\theta|^2-\Reelle(\tau)| < 2^j} \frac{|f_n|^2}{|n-\theta|^2-\tau}\right)^{1/2}\\
   & \leq C \norme{f}_{L^2(\T^2)},
\end{align*}
which concludes the proof of \eqref{eq:boundedP}.

So, we deduce from \eqref{eq:boundedP} and \cite[Lemme p.XVII-5]{GV95} that 
\begin{equation}
\label{eq:boundedPPstar}
   \forall f \in L^{4/3}(\T^2),\ \norme{P_{\theta,\tau}P_{\theta,\tau}^*f}_{L^4(\T^2)} \leq C \norme{f}_{L^{4/3}(\T^2)}.
\end{equation}
Moreover, it is not difficult to see that $P_{\theta,\tau}P_{\theta,\tau}^*$ coincides with $(-\Delta + 2 i \theta \cdot \nabla + |\theta|^2 - \tau )^{-1}$ on $C_c^{\infty}(\T^2)$ that is dense in $L^{4/3}(\T^2)$, hence \eqref{eq:boundedPPstar} leads to \eqref{eq:resolventtheta}.
\end{proof}

\subsection{A priori estimates for the group generated by $\mathcal{H}_{\theta,V}$}
This section is devoted to present a priori estimates for the group generated by $\mathcal{H}_{\theta, V}$. In the first subsection, we begin by introducing useful stability results holding in any dimension. In the second and third section, we state Strichartz estimates in the one and two dimensional cases.

\subsubsection{Stability results in $L^2$ in $\T^d$}

We first focus on the $d$-dimensional setting, i.e. we consider for $T>0$

\begin{equation}
	\label{eq:SchrodingerObsTorusdf}
		\left\{
			\begin{array}{ll}
				i  \partial_t v  = (-\Delta + 2 i \theta \cdot \nabla + |\theta|^2 + V ) v + f & \text{ in }  (0,T) \times \T^d, \\
				v(0, \cdot) = v_0 & \text{ in } \T^d.
			\end{array}
		\right.
\end{equation}

The next result relies on the conservation of the $L^2$-norm and the well-posedness of \eqref{eq:SchrodingerObsTorusdf}.

\begin{proposition}
\label{prop:wellposednessL2}
Let $T>0$. For every $v_0 \in L^2(\T^d)$, $\theta \in [0,1]^d$ and $V \in L^{\infty}(\T^d)$, the solution $v$ to \eqref{eq:SchrodingerObsTorusdf} with $f=0$ satisfies
\begin{equation}
    \label{eq:conservationL2norm}
    \norme{v(t)}_{L^2(\T^d)} = \norme{v_0}_{L^2(\T^d)}\qquad \forall t \in [0,T].
\end{equation}
Moreover, for every $v_0 \in L^2(\T^d)$, $f \in L^1(0,T;L^2(\T^d))$, $\theta \in [0,1]^d$ and $V \in L^{\infty}(\T^d)$, the solution $v$ of \eqref{eq:SchrodingerObsTorusdf} satisfies
\begin{equation}
\label{eq:L2normsource}
    \norme{v}_{L^{\infty}(0,T;L^2(\T^d))} \leq \norme{v_0}_{L^2(\T^d)} + \norme{f}_{L^1(0,T;L^2(\T^d))}.
\end{equation}
\end{proposition}

The following result is a stability result according to the parameters $\theta \in [0,1]^d$ and $V \in L^{\infty}(\T^d)$.
\begin{proposition}
\label{prop:stabilityparameters}
Assume that $\theta_n \to \theta$ in $[0,1]^d$ and $V_n \rightharpoonup^\star V$ in $L^{\infty}(\T^d)$ as $n \to +\infty$, then we have
\begin{equation}
\label{eq:convergencesemigroup}
    e^{-it(\mathcal{H}_{\theta_n,V_n})} \varphi \underset{n \to +\infty}{\longrightarrow} e^{-it(\mathcal{H}_{\theta,V})} \varphi \quad \text{in} \quad L^2(\T^d),\qquad \forall t \in [0,T],\ \forall \varphi \in L^2(\T^d).
\end{equation}
\end{proposition}
\begin{proof}
This is a direct application of Trotter-Kato approximation theorem, see \cite[Theorem 4.2]{Paz83}, using the convergence of the resolvent already stated in Proposition \ref{prop:convresolvent}.
\end{proof}

\subsubsection{Strichartz estimates in $\T^1$} 

Now we focus on the $1$-dimensional setting, i.e. we consider for $T>0$

\begin{equation}
	\label{eq:SchrodingerObsTorus1f}
		\left\{
			\begin{array}{ll}
				i  \partial_t v  = (-\partial_{x}^2 + 2 i \theta \partial_{x} + |\theta|^2 + V ) v + f & \text{ in }  (0,T) \times \T^1, \\
				v(0, \cdot) = v_0 & \text{ in } \T^1.
			\end{array}
		\right.
\end{equation}

Instrumental in the proof of the one-dimensional observability estimates for Schrödinger type equations are the following Strichartz type estimates, taken from \cite[Proposition~2.1]{BBZ13}:

\begin{proposition}[{\cite[Proposition~2.1]{BBZ13}}]\label{strichartz_1D_prop}
Let $T>0$ and $M>0$. There exists $C=C(T,M)>0$ such that for all $\theta \in [0,1]$, $V\in L^{\infty}(\mathbb T)$ with $\|V\|_{L^{\infty}} \leq M$ and $u_0 \in L^2(\mathbb T)$,
\begin{equation}\label{estimation_strichartz_1D}
    \left\|e^{-it\mathcal H_{\theta, V}}u\right\|_{L^{\infty}(\mathbb T; L^2(0,T))} \leq C \|u\|_{L^2(\mathbb T)}.
\end{equation}
\end{proposition}
Notice that Proposition~\ref{strichartz_1D_prop} is a slightly modified version of \cite[Proposition~2.1]{BBZ13}. First, the potential is assumed to be in $L^{\infty}(\mathbb T)$, which is a sufficient assumption for our purpose. On the other hand, we claim that the constant in \eqref{estimation_strichartz_1D} can be taken uniformly with respect to $V$ in $B_{L^{\infty}}(0,M)$. This fact is clearly contained in the proof given by the authors of \cite[Proposition~2.1]{BBZ13}.

The next result concerns Strichartz estimates for solutions to \eqref{eq:SchrodingerObsTorus1f}.
\begin{proposition}
\label{prop:StrichartzEstimate1d}
Let $T>0$ and $M>0$. There exists $C=C(T,M)>0$ such that for all $\theta \in [0,1]$, $V\in L^{\infty}(\mathbb T)$ with $\|V\|_{L^{\infty}} \leq M$, $v_0 \in L^2(\T^1)$ and $f \in L^1(0,T;L^2(\T^1))$, the solution $v$ to \eqref{eq:SchrodingerObsTorus1f} satisfies
\begin{equation}
\label{eq:strichartz1Dsource}
    \norme{v}_{L^{\infty}(0,T;L^2(\T^1)) \cap L^{\infty}(\mathbb T^1; L^2(0,T))} \leq C \left( \norme{v_0}_{L^2(\T^1)} + \norme{f}_{L^1(0,T;L^2(\T^1))}\right).
\end{equation}
\end{proposition}

\begin{proof}
On the first hand, thanks to \eqref{eq:L2normsource}, we have for all $v_0 \in L^2(\T^1)$,
$$\|v\|_{L^{\infty}(0,T; L^2(\T^1))} \leq \|v_0\|_{L^2(\T^1)}+ \| f\|_{L^1(0,T; L^2(\T^1)}.$$

On the other hand, by using Duhamel's formula, we have that
$$v(t) = e^{-it \mathcal{H}_{\theta,V}} v_0 + \int_{0}^T 1_{s < t} e^{-i(t-s)\mathcal{H}_{\theta,V}} f(s) ds.$$
Then, by \eqref{eq:conservationL2norm} and \eqref{estimation_strichartz_1D}, we have
\begin{align*}
    \norme{v}_{L^{\infty}(\mathbb T; L_t^2(0,T))} &\leq C \left( \norme{v_0}_{L^2(\T^2)} + \int_0^T \norme{ e^{-it\mathcal{H}_{\theta,V}} e^{is\mathcal{H}_{\theta,V}}  f(s) }_{L^{\infty}(\mathbb T; L_t^2(0,T))} ds\right)\\ 
    & \leq C  \left( \norme{v_0}_{L^2(\T^2)} + \int_0^T \norme{e^{is\mathcal{H}_{\theta,V}}  f(s) }_{L^{2}(\T^1)} ds\right)\\
& = C \left( \norme{v_0}_{L^2(\T^1)} + \norme{f}_{L^1(0,T;L^2(\T^1))}\right),
\end{align*}
which leads to the second estimate of \eqref{eq:strichartz1Dsource}.
\end{proof}

\subsubsection{Strichartz estimates in $\T^2$} 

Now we focus on the $2$-dimensional setting, i.e. we consider for $T>0$

\begin{equation}
	\label{eq:SchrodingerObsTorus2f}
		\left\{
			\begin{array}{ll}
				i  \partial_t v  = (-\Delta + 2 i \theta \cdot \nabla + |\theta|^2 + V ) v + f & \text{ in }  (0,T) \times \T^2, \\
				v(0, \cdot) = v_0 & \text{ in } \T^2.
			\end{array}
		\right.
\end{equation}

The next result concerns Strichartz estimates for solutions to \eqref{eq:SchrodingerObsTorus2f}.
\begin{proposition}
\label{prop:StrichartzEstimate}
Let $T>0$ and $M>0$. There exists $C=C(T,M)>0$ such that for every $\theta \in [0,1]^2$, $v_0 \in L^2(\T^2)$ and $f \in L^{4/3}(\T^2;L^2(0,T)) \cap L^1(0,T;L^2(\T^2))$, the solution $v$ to \eqref{eq:SchrodingerObsTorus2f} satisfies
\begin{equation*}
    \norme{v}_{L^{\infty}(0,T;L^2(\T^2)) \cap L^{4}(\T^2;L^2(0,T))} \leq C \left( \norme{v_0}_{L^2(\T^2)} + \norme{f}_{L^{4/3}(\T^2;L^2(0,T)) \cap L^1(0,T;L^2(\T^2))}\right).
\end{equation*}
\end{proposition}
Proposition \ref{prop:StrichartzEstimate} has already been proved in \cite[Proposition 2.2]{BBZ13} for $\theta = 0$. Here, the main novelty is to get a uniform a priori estimate with respect to the parameter $\theta \in [0,1]^2$. 
\begin{remark}
\label{rmk:strichartztheta}
It is worth mentioning that the homogeneous case, i.e. $f=0$ cannot be proved as \cite[Remark (3) p.333]{BZ19}. Indeed, the starting point for obtaining such a case is the Zygmund's inequality \eqref{eq:zygmund}. By setting $v_0 = \sum_{\lambda > 0} v_\lambda$ with $v_{\lambda}=\sum_{|n-\theta|^2 = \lambda} c_n e^{i n\cdot x}$, we have that for $v$, solution to \eqref{eq:SchrodingerObsTorus2f} with $f=0$,
\begin{equation}
\label{eq:solutionL4homogeneous}
    \norme{v}_{L^4(\T^2;L^2(0,2\pi))}^4 = \int_{\T^2}\left(\int_{0}^{2 \pi}\left|\sum_{\lambda} e^{it \lambda} v_{\lambda}(z)\right|^2 dt\right)^2 dz.
\end{equation}

For $\theta=0$, the right hand side of \eqref{eq:solutionL4homogeneous} can then be bounded as follows, using Parseval equality, Hölder's inequality and Zygmund's inequality \eqref{eq:zygmund}
\begin{multline*}
  \int_{\T^2}\left(\int_{0}^{2 \pi}\left|\sum_{\lambda} e^{it \lambda} v_{\lambda}(z)\right|^2 dt\right)^2 dz =   (2 \pi)^2 \int_{\T^2}\left(\sum_{\lambda} |v_\lambda(z)|^2\right)^2 dz = (2 \pi)^2 \int_{\T^2}\sum_{\lambda, \mu} |v_\lambda(z)|^2 |v_{\mu}(z)|^2 dz\\
  \leq (2 \pi)^2 \sum_{\lambda, \mu} \norme{v_{\lambda}}_{L^4(\T^2)}^2 \norme{v_{\mu}}_{L^4(\T^2)}^2 \leq C \sum_{\lambda, \mu} \norme{v_{\lambda}}_{L^2(\T^2)}^2\norme{v_{\mu}}_{L^2(\T^2)}^2 \leq C \left(\sum_{\lambda} \norme{v_{\lambda}}_{L^2(\T^2)}^2\right)^2
  \leq C \norme{v_0}_{L^2(\T^2)}^2.
\end{multline*}

On the other hand, for $\theta \neq 0$, one cannot use Parseval equality to first estimate the right hand side of \eqref{eq:solutionL4homogeneous} because $\lambda>0$ is not an integer anymore. A natural strategy would be to employ an Ingham's inequality, see \cite{Ing36}. But for obtaining such an estimate, one needs to prove a gap condition on the set $\{|n-\theta|^2\ ;\ n \in \Z^2\}$, uniformly in $\theta$. This turns to be false because $| (1,0) - (\theta_1,\theta_2)|^2 - | (0,1) - (\theta_1,\theta_2)|^2 =  2(\theta_2-\theta_1) \to 0$ as $\theta_2-\theta_1 \to 0$.

Actually, even for fixed $\theta$, this gap fails to hold. Indeed, let us consider $\theta=(\theta_1, \theta_2)$, where $\theta_1$ and $\theta_2$ are $\mathbb Q$-linearly independent. It is well-known that, in that case, the set $\theta_1 \mathbb Z+\theta_2 \mathbb Z$ is dense in $\mathbb R$. For any $\varepsilon>0$, we can therefore find two integers $k, l \in \mathbb Z$ such that $0<|k \theta_1+l \theta_2| < \frac{\varepsilon}{4}$. By defining $n=(k+2l, l-2k)$ and $m=(2l-k, -l-2k)$, we obtain 
$$0<\left||n-\theta|^2-|m-\theta|^2\right|= 4|k\theta_1+l\theta_2|< \varepsilon.$$
This prevents $\{|n-\theta|^2\ ;\ n \in \Z^2\}$ to satisfy a gap.
\end{remark}

Instrumental in the proof of Proposition \ref{prop:StrichartzEstimate} are the resolvent estimates \eqref{eq:resolventthetaV} given by Proposition \ref{prop:resolventV}.
\begin{proof}[Proof of Proposition \ref{prop:StrichartzEstimate}]
Let $\theta \in [0,1]^2$, $v_0 \in L^2(\T^2)$ and $f \in L^{4/3}(\T^2;L^2(0,T)) \cap L^1(0,T;L^2(\T^2))$. Let $v$ be the solution of \eqref{eq:SchrodingerObsTorus2f}. We split the solution into two parts, i.e. 
$$v(t) = e^{- i t \mathcal H_{\theta,V}} v_0 + \int_{0}^t e^{- i (t-s) \mathcal H_{\theta,V}} f(s) ds =: v_1(t) + v_2(t).$$

\textit{First step: Bound on $v_1(t)$.} The goal of this step is to obtain
\begin{equation}
\label{eq:boundsemigroupdata}
    \norme{v_1(t)}_{L^{\infty}(0,T;L^2(\T^2)) \cap L^{4}(\T^2;L^{2}(0,T))} \leq C \norme{v_0}_{L^2(\T^2)}.
\end{equation}

Firstly, from \eqref{eq:conservationL2norm}, we have $\norme{v_1(t)}_{L^{\infty}(0,T;L^2(\T^2))} = \norme{v_0}_{L^2(\T^2)}$.

The second bound of \eqref{eq:boundsemigroupdata} comes from a “$TT^*$ argument”. More precisely, we set
$$ T : v_0 \in L^2(\T^2) \mapsto v_1.$$
In order to prove that $T$ is a bounded operator from $L^{2}(\T^2)$ to $L^{4}(\T^2;L^2(0,T))$, it suffices to prove that $TT^*$ is a bounded operator from $L^{4/3}(\T^2;L^2(0,T))$ to $L^4(\T^2;L^2(0,T))$, see for instance \cite[Lemme p.XVII-5]{GV95}.

First, a straightforward computation gives that
\begin{equation*}
    T^{*} f = \int_{0}^T e^{is\mathcal{H}_{\theta,V}} f(s) ds \in L^2(\T^2),\qquad \forall f \in L^{4/3}(\T^2;L^2(0,T)),
\end{equation*}
so
\begin{align*}
    TT^*f = \int_0^T e^{-i(t-s) \mathcal{H}_{\theta,V}} f(s) ds 
    = \int_0^t e^{-i(t-s) \mathcal{H}_{\theta,V}} f(s) ds + \int_t^T e^{-i(t-s) \mathcal{H}_{\theta,V}} f(s) ds
     =: T_1 f + T_2 f.
\end{align*}
It remains to prove that $T_1$ and $T_2$ are bounded operators from $L^{4/3}(\T^2;L^2(0,T))$ to $L^4(\T^2;L^2(0,T))$. We will only do it for $T_1$ because the same arguments apply for treating $T_2$.

Setting $y=T_1 f$, we observe that $y(t,\cdot)$ solves the Schrödinger equation \eqref{eq:SchrodingerObsTorus2f} with $v_0=0$, so setting $Y(t):= e^{-t} 1_{(0,+\infty)}(t) y(t)$ and $F(t) := e^{-t} f(t) 1_{(0,T)}(t)$, the following equation is satisfied in the distributional sense
\begin{equation}
\label{eq:equationY}
    i  \partial_t Y + i Y  = (-\Delta + 2 i \theta \cdot \nabla + |\theta|^2 +V ) Y + F \qquad \text{ in }  (0,+\infty) \times \T^2.
\end{equation}
Therefore, taking the Fourier transform in the time variable of \eqref{eq:equationY}, we obtain the equation satisfied by $\widehat{Y}$,
\begin{equation}
\label{eq:equationwidehatY}
  (-\Delta + 2 i \theta \cdot \nabla + |\theta|^2 + V + \tau -i ) \widehat{Y}(\tau)  = -\widehat{F}(\tau)\qquad \forall \tau \in \R.
\end{equation}
Therefore one can apply \eqref{eq:resolventthetaV} to \eqref{eq:equationwidehatY}, using that $\Ima(\tau-i) = 1$ to get
\begin{equation}
\label{eq:FourierEstimateYtau}
    \norme{\widehat{Y}(\tau)}_{L^4(\T^2)} \leq C \norme{\widehat{F}(\tau)}_{L^{4/3}(\T^2)}\qquad \forall \tau \in \R.
\end{equation}
Hence, using Parseval equality, and reversing time integration and space integration because $4 \geq 2 \geq 2/3$, we get from \eqref{eq:FourierEstimateYtau} that
\begin{multline*}
    \norme{T_1 f}_{L^4(\T^2;L^2(0,T))} \leq C \norme{Y}_{L^4(\T^2;L^2(\R_t))} = C \norme{\widehat{Y}}_{L^4(\T^2;L^2(\R_{\tau}))} \leq C\norme{\widehat{Y}}_{L^2(\R_{\tau};L^{4}(\T^2))} \\
  \leq C\norme{\widehat{F}}_{L^2(\R_{\tau};L^{4/3}(\T^2))}  \leq  C\norme{\widehat{F}}_{L^{4/3}(\T^2;L^2(\R_{\tau}))} 
    \leq C \norme{F}_{L^{4/3}(\T^2;L^2(\R_t))} \leq C \norme{f}_{L^{4/3}(\T^2;L^2(0,T))},
\end{multline*}
which proves that $T_1$ is bounded from $L^{4/3}(\T^2;L^2(0,T))$ to $L^4(\T^2;L^2(0,T))$. This concludes the proof of the second bound of \eqref{eq:boundsemigroupdata} hence the first step.

\textit{Second step: Bound on $v_2(t)$.} The goal of this step is to obtain
\begin{equation*}
    \norme{v_2}_{L^{\infty}(0,T;L^2(\T^2)) \cap L^{4}(\T^2;L^{2}(0,T))} \leq C \norme{f}_{L^1(0,T;L^2(\T^2) \cap L^{4/2}(\T^2;L^{2}(0,T))}.
\end{equation*}
The first bound is obtained as follows, for every $t \in [0,T]$,
\begin{equation*}
    \norme{v_2(t)}_{L^2(\T^2)) } \leq \int_{0}^T \norme{e^{-i(T-s)(\mathcal{H}_{\theta,V})} f(s)}_{L^2(\T^2)} \leq \int_{0}^T \norme{ f(s)}_{L^2(\T^2)} \leq \norme{f}_{L^1(0,T;L^2(\T^2))}.
\end{equation*}
The second bound exactly comes from the boundedness of $T_1$ from $L^{4/3}(\T^2;L^2(0,T))$ to $L^4(\T^2;L^2(0,T))$, see the previous step.

This concludes the proof of the second step then the proof of Proposition \ref{prop:StrichartzEstimate}.
\end{proof}

\section{Observability estimates in $\T^d$: from highly oscillating initial data to $L^2$ initial data}
\label{sec:refobsmultid}

The goal of this section is to prove that it is sufficient to establish observability inequalities for highly oscillating initial data to deduce the observability of $L^2$ initial data. Note that this step holds for all multi-dimensional torus $\T^d$. The first part consists in obtaining a weaker observability estimate, i.e. the expected observability estimate up to a compact term by using a dyadic decomposition, see Proposition \ref{frequency_cutoff_obs_prop} below. Then, the second part is devoted to remove this compact term by using a compactness-uniqueness argument, see Proposition \ref{prop:from_weakobs_to_strongobs} below. Here, because we are considering measurable observation sets, some arguments have to be changed as the unique continuation property satisfied by the elliptic operator $\mathcal{H}_{\theta,V}$. It is worth mentioning that while Proposition \ref{frequency_cutoff_obs_prop} focuses on a subset of $L^{\infty}$ potentials, Proposition \ref{prop:from_weakobs_to_strongobs} considers a compact subset of $L^{\infty}$ potentials.

\subsection{Weaker observability estimate}

Let $d \geq 1$. For $\theta \in [0,1]^d$, $V\in L^{\infty}(\mathbb T^d)$ and $\rho >0$, we define for all $h>0$,
\begin{equation*}\label{frequency_cutoff}
    \Pi_{h, \rho,\theta, V}u_0= \chi\left(\frac{h^2\mathcal{H}_{\theta,V}-1}{\rho}\right)u_0,
\end{equation*}
where $\chi \in C_c^{\infty}((-1,1), \rr)$ is equal to $1$ in a neighborhood of $0$.
\begin{proposition}\label{frequency_cutoff_obs_prop}
Let $b \in L^{\infty}(\mathbb T^d)$ be a non-negative function and $S\subset L^{\infty}(\T^d)$ be a bounded subset. If for any $T>0$ there exist some positive constants $C=C(T,S)>0$, $\rho_0=\rho_0(T, S)>0$ and $h_0=h_0(T, S)>0$ such that for all $\theta \in [0,1]^d$ and $V \in S$, we have
\begin{multline}\label{frequency_cutoff_obs}
   \forall 0<h\leq h_0, \forall 0< \rho \leq \rho_0, \forall u_0 \in L^2(\mathbb T^d), \\ \left\| \Pi_{h, \rho,\theta, V}u_0 \right\|_{L^2(\mathbb T^d)}^2 \leq C \int_0^T \int_{\mathbb T^d} b(z) \left| e^{-it\mathcal{H}_{\theta,V}} (\Pi_{h, \rho,\theta, V}u_0)(z)\right|^2 dz dt,
\end{multline}
then for any $T>0$ there exists a positive constant $C{'}=C{'}(T, S)>0$ such that for all $\theta \in [0,1]^d$ and $V \in S$,
\begin{equation*}
    \forall u_0 \in L^2(\mathbb T^d), \quad \left\| u_0 \right\|_{L^2(\mathbb T^d)}^2 \leq C{'} \left(\int_0^T \int_{\mathbb T^d} b(z) \left| e^{-it\mathcal{H}_{\theta,V}} u_0(z)\right|^2 dz dt+ \|u_0\|^2_{H^{-2}(\T^d)}\right).
\end{equation*}
\end{proposition}

\begin{proof}
Let us assume that \eqref{frequency_cutoff_obs} holds for any time $T>0$.\\

Let us show that the following weaker observability estimates hold: for any $T>0$ there exists a constant $C{'}>0$ such that for all $\theta \in [0,1]^d$ and $V \in S$,
\begin{equation}\label{weak_obs_estimate}
    \forall u_0 \in L^2(\mathbb T^d), \quad \left\| u_0 \right\|_{L^2(\mathbb T^d)}^2 \leq C \int_0^T \int_{\mathbb T^d} b(z) \left| e^{-it\mathcal{H}_{\theta,V}} u_0(z)\right|^2 dz dt+C\|u_0\|^2_{H^{-2}(\mathbb T^d)}.
\end{equation}
The weaker observability estimate \eqref{weak_obs_estimate} will be deduced from the observability estimate for highly oscillating initial data \eqref{frequency_cutoff_obs}. The strategy is crucially inspired by \cite[Section 3.1]{BZ19}. Let us describe in an heuristic way the strategy. 

We first decompose dyadically the initial data, the low frequencies are then putting into the $H^{-2}$-norm of the initial data, whereas \eqref{frequency_cutoff_obs} is applied to high frequencies. Moreover, in order to reconstruct the solution, in the right hand side of \eqref{weak_obs_estimate}, the natural idea is to commute $\sqrt{b}$ and $\mathcal{H}_{\theta,V}$ that does not work. This is why here we will crucially use that $D_t e^{-it\mathcal{H}_{\theta,V}} = -\mathcal{H}_{\theta,V} e^{-it\mathcal{H}_{\theta,V}}$ with $D_t = \partial_t /i$. Therefore, one can use $D_t \sqrt{b} = \sqrt{b} D_t$. Finally, because we are observing the solution in the interval time $(0,T)$, one should commute $1_{(0,T)}$ and $D_t$, that does not work. This is why, before applying the observability estimate, one may introduce a cut-off function $\psi$, whose support is contained in $(0,T)$, then apply semi-classical calculus in the time variable in order to commute $\psi$ and $D_t$. The remainder terms would be put in the $H^{-2}$-norm of the initial data.

Let $T>0$. By assumptions, there exist $\rho_0>0$, $h_0>0$ such that \eqref{frequency_cutoff_obs} holds. Fix $R>1$ such that 
$(R^{-1},R) \subset \{r \in \R\ ;\ \chi((r-1)/\rho_0) = 1\}$.
Then from \cite[Proposition 2.10]{BCD11}, one can find a dyadic partition of the unity as follows, there exist $\varphi_0 \in C_c^{\infty}((-1,1);[0,1])$ and $\varphi \in C_c^{\infty}((R^{-1},R);[0,1])$ such that, denoting $\varphi_k(r)=\varphi^2(R^{-k}r)$ for $k \geq 1$,
\begin{equation}
\label{eq:decompositiondyadique}
    \forall r \in \R^+,\ \varphi_0^2(r) + \sum_{k=1}^{+\infty} \varphi_k(r) = 1.
\end{equation}
Notice that, since $S$ is a bounded subset of $L^{\infty}(\T^d)$, there exists $M>0$ such that 
$$\forall V \in S, \quad \|V\|_{L^{\infty}(\T^d)} \leq M.$$
We therefore deduce from \eqref{eq:uhshthetaV} that for every $s \in \R$, there exist $c_1',c_2'>0$ such that for every $\theta \in [0,1]^d$ and $V \in S$,
\begin{equation}
\label{eq:comparaisonvarphiknorme}
  c_1' \sum_{k=0}^{+\infty} R^{2ks} \norme{\varphi_{k}(\mathcal{H}_{\theta,V})u_0}_{L^2(\T^d)}^2 \leq  \norme{ u_0}_{H^s(\T^d)}^2
 \leq c_2' \sum_{k=0}^{+\infty} R^{2ks} \norme{\varphi_{k}(\mathcal{H}_{\theta,V})u_0}_{L^2(\T^d)}^2.
\end{equation}

Let $\psi \in C_c^{\infty}((0,T);[0,1])$ that satisfies $\psi(t) =1$ on $(T/3,2T/3)$. Let us choose $K\geq 1$ large enough such that $R^{-K} \leq h_0^2$. Then, we have that for every $k \geq K+1$, setting $h=R^{-k/2}$, $\varphi_{k}(\mathcal{H}_{\theta,V})$ coincides with $\Pi_{h,\rho_0,\theta,V} \varphi_{k}(\mathcal{H}_{\theta,V})$. One can then use \eqref{frequency_cutoff_obs} on the time interval $(T/3,2T/3)$ to obtain
\begin{equation}
\label{eq:EstimationFrequencyCutOfft3}
    \norme{\varphi_{k}(\mathcal{H}_{\theta,V})u(T/3)}_{L^2(\T^2)}^2 \leq C \int_{\R} \psi(t)^2 \norme{\sqrt{b} \varphi_{k}(\mathcal{H}_{\theta,V}) u(t)}_{L^2(\T^2)}^2 dt\qquad \forall k \geq K+1.
\end{equation}
Now using that $D_t u = -\mathcal{H}_{\theta,V} u$ and $D_t \sqrt{b} = \sqrt{b} D_t$, we deduce from \eqref{eq:EstimationFrequencyCutOfft3} that
\begin{align}
\label{eq:EstimationFrequencyCutOfft3Bis}
    \norme{\varphi_{k}(\mathcal{H}_{\theta,V})u(T/3)}_{L^2(\T^2)}^2 \leq C \norme{\sqrt{b} \psi(t) \varphi_k(D_t) u}_{L^2(\R_t \times  \T^2)}^2 \qquad \forall k \geq K+1.
\end{align}

Let $\widetilde{\psi} \in C_c^{\infty}((0,T);[0,1])$ such that $\widetilde{\psi}=1$ on $\text{supp}(\psi)$. Setting $h=R^{-k/2}$, the semi-classical parameter, from the semi-classical calculus on $\R$, see Theorem \ref{th:semiclassicalcomposition}, \eqref{eq:expansioncommutatorRd} and \eqref{eq:suppdisjointssemiclassical}, the asymptotic expansion holds
\begin{align}
    \psi(t) \varphi_k(D_t) &= \psi(t) \varphi^2(hD_t)\notag\\
    &=\psi(t) \varphi^2(hD_t) \widetilde{\psi}(t)+\psi(t) \varphi^2(hD_t) (1- \tilde{\psi}(t)) \notag \\
   &= \psi(t) \varphi^2(hD_t) \widetilde{\psi}(t) + E(t,hD_t)(1+|t|^2)^{-1}(1+|hD_t|^2)^{-1},\label{eq:psitildepsisemiclassical}
\end{align}
where $$E(t,hD_t) = op_h(c),\ c \in \mathcal{S}(\R^2)\ \text{and}\ \sup_{(t,\tau) \in \R\times \R} |(1+t^2)^{\alpha} (1+\tau^2)^{\beta} \partial_{t}^\gamma \partial_{\tau}^{\delta} c(t,\tau)| \leq C_{\alpha,\beta,\gamma,\delta} h^3,\ \alpha,\beta,\gamma,\delta \in \N.$$
Then, by using \eqref{eq:EstimationFrequencyCutOfft3Bis}, \eqref{eq:psitildepsisemiclassical}, Theorem \ref{th:continuitypseudoRd}, \eqref{eq:EstimationCalderonVaillancourt} and again $D_t u = -\mathcal{H}_{\theta,V} u$ in $(0,T)$ we have 
\begin{align*}
  \notag & \norme{\varphi_{k}(\mathcal{H}_{\theta,V})u(T/3)}_{L^2(\T^d)}^2 \\
 \notag   & \leq C \norme{\sqrt{b} \varphi_k(D_t) \widetilde{\psi}(t) u}_{L^2(\R_t \times\T^d)}^2 + C  h^6 \norme{(1+t^2)^{-1} (1+|hD_t|^2)^{-1} u(t)}_{L^2(\R_t \times \T^d)}^2 \\
 & \leq C \norme{ \varphi_k(D_t) \widetilde{\psi}(t) \sqrt{b} u}_{L^2(\R_t \times\T^d)}^2 + C  h^6 \norme{(1+t^2)^{-1} (1+|h\mathcal{H}_{\theta,V}|^2)^{-1} u(t)}_{L^2(\R_t \times \T^d)}^2 \qquad \forall k \geq K+1.
\end{align*}
Therefore, by summing for $k \geq K+1$ the preceding estimate, remembering that $h=R^{-k/2}$, we get from  \eqref{eq:decompositiondyadique}, \eqref{eq:conservationL2norm} and \eqref{eq:uhshthetaV}
\begin{align}
\sum_{k=K+1}^{+\infty} \norme{\varphi_{k}(\mathcal{H}_{\theta,V})u(T/3)}_{L^2(\T^d)}^2  &\leq \int_{\R} \widetilde{\psi}(t)^2 \norme{\sqrt{b}  u(t)}_{L^2(\T^d)}^2 dt + C \norme{(1+|\mathcal{H}_{\theta,V}|^2)^{-1} u(t)}_{L^{\infty}(\R_t;L^2(\T^d))}^2\notag \\
 & \leq C \int_{0}^T \norme{\sqrt{b}  u(t)}_{L^2(\T^d)}^2 dt + C \norme{(1+|\mathcal{H}_{\theta,V}|^2)^{-1} u_0}_{L^2(\T^d)}^2 \notag \\
 & \leq C \int_{0}^T \norme{\sqrt{b}  u(t)}_{L^2(\T^d)}^2 dt + C \norme{u_0}_{H^{-2}(\T^d)}^2
\label{eq:estimatehighfrequenciescutoff}
\end{align}

To sum up, we then have from \eqref{eq:estimatehighfrequenciescutoff} and \eqref{eq:comparaisonvarphiknorme}
\begin{multline*}
     \norme{u_0}_{L^2(\T^d)}^2 = \sum_{k=0}^{+\infty} \norme{\varphi_k(\mathcal{H}_{\theta,V})u_0}_{L^2(\T^d)}^2
      \leq C \norme{u_0}_{H^{-2}(\T^d)}^2 + \sum_{k=K+1}^{+\infty} \norme{e^{-i(T/3)\mathcal{H}_{\theta,V}}\varphi_{k}(\mathcal{H}_{\theta,V})u_0}_{L^2(\T^d)}^2\\
     \leq C \norme{u_0}_{H^{-2}(\T^d)}^2 + \sum_{k=K+1}^{+\infty} \norme{\varphi_{k}(\mathcal{H}_{\theta,V})u(T/3)}_{L^2(\T^d)}^2
     \leq C \int_{0}^T \norme{\sqrt{b}  u(t)}_{L^2(\T^d)}^2 dt + C \norme{u_0}_{H^{-2}(\T^d)}^2,
\end{multline*}
which concludes the proof of \eqref{weak_obs_estimate}.
\end{proof}

\subsection{Remove the compact term in the weaker observability estimate}

\begin{proposition}\label{prop:from_weakobs_to_strongobs}
Let $b \in L^{\infty}(\mathbb T^d)$ be a non-negative function and $S\subset L^{\infty}(\T^d)$ be a compact subset for the $\text{weak}^\star$ topology. If for any $T>0$ there exist some positive constants $C=C(T,S)>0$ such that for all $\theta \in [0,1]^d$ and $V \in S$, we have
\begin{equation}\label{weak_obs_estimate2}
    \forall u_0 \in L^2(\mathbb T^d), \quad \left\| u_0 \right\|_{L^2(\mathbb T^d)}^2 \leq C{'} \left(\int_0^T \int_{\mathbb T^d} b(z) \left| e^{-it\mathcal{H}_{\theta,V}} u_0(z)\right|^2 dz dt+ \|u_0\|^2_{H^{-2}(\T^d)}\right).
\end{equation}
then for any $T>0$ there exists a positive constant $C{'}=C{'}(T, S)>0$ such that for all $\theta \in [0,1]^d$ and $V \in S$,
\begin{equation*}
    \forall u_0 \in L^2(\mathbb T^d), \quad \left\| u_0 \right\|_{L^2(\mathbb T^d)}^2 \leq C{'} \int_0^T \int_{\mathbb T^d} b(z) \left| e^{-it\mathcal{H}_{\theta,V}} u_0(z)\right|^2 dz dt.
\end{equation*}
\end{proposition}
\begin{proof}
Let us assume that \eqref{weak_obs_estimate2} holds for any time $T>0$.\\

\textit{First step: an unique continuation property.}\\
Let us consider the following space
\begin{equation*}
    \mathcal N_{T, \theta, V}=\left\{u \in L^2([0,T]\times \mathbb T^d); \ b(x)e^{-it\mathcal{H}_{\theta,V}}u(x)=0\ \text{on}\ [0,T]\times \mathbb T^d\right\},
\end{equation*}
with $\theta \in [0,1]^d$ and $V \in S$. Let us show that $\mathcal N_{T, \theta, V}=\left\{0\right\}$. To that end, we proceed by contradiction and assume $N_{T, \theta, V}\neq \left\{0\right\}$. We begin by noticing that $N_{T, \theta, V}$ is invariant by the action of $\mathcal{H}_{\theta,V}$. Indeed, if $u \in N_{T, \theta, V}$, then for all $0<\varepsilon< T$, $u_\varepsilon= \frac{e^{-i\varepsilon\mathcal{H}_{\theta,V}}u-u}{\varepsilon}$ belongs to $N_{T-\varepsilon, \theta, V}$. Thus, by applying \eqref{weak_obs_estimate} at time $T-\varepsilon$, we obtain that for all $0< \varepsilon < T$,
\begin{equation*}
    \|u_{\varepsilon}\|_{L^2(\mathbb T^d)} \leq C{'} \|u_{\varepsilon} \|_{H^{-2}(\mathbb T^d)},
\end{equation*}
for some positive constant $C'>0$.
Since $u_{\varepsilon} \underset{\varepsilon\to 0^+}{\longrightarrow} u$ in $\mathcal D'(\mathbb T^d)$ and $H^{-2}(\mathbb T^d)$, it follows that $u \in D\left(\mathcal{H}_{\theta,V}\right)$ (see for instance \cite[Section 1.1]{Paz83}) and 
$$\|\mathcal{H}_{\theta,V}u\|_{L^2(\mathbb T^d)} \leq C" \|u\|_{L^2(\mathbb T^d)},$$
for some positive constant $C''>0$.
Moreover, since $u_{\varepsilon}$ belongs to $N_{T-\varepsilon, \theta, V}$ for all $0< \varepsilon < T$, we deduce that $\mathcal{H}_{\theta,V}u \in N_{T-\delta, \theta, V}$ for all $0< \delta <T$. Then, $\mathcal{H}_{\theta,V}u \in N_{T, \theta, V}$ and $N_{T, \theta, V}$ is invariant by the action of $\mathcal{H}_{\theta,V}$. As a consequence, there exists a nontrivial function $\phi \in L^2(\mathbb T^d)$ and $\lambda \in \rr$ such that
$$(-\Delta+2i\theta\cdot\nabla+|\theta|^2+V)\phi = \lambda \phi \quad \text{and} \quad \phi \in N_{T, \theta, V}.$$
In particular, $\phi$ is an eigenfunction of $(-\Delta+2i\theta\cdot\nabla+|\theta|^2+V)$ which vanishes on a set of positive measure. By the unique continuation result from \cite[Theorem 1.2]{Reg01}, we deduce that $\phi\equiv 0$. This provides a contradiction and consequently, $N_{T, \theta, V}=\left\{0\right\}.$\\

\textit{Second step: we remove the $H^{-2}$-norm in the weak observability estimate.}

By now, we establish that there exists a positive constant $C{'}=C{'}(T,M)>0$ such that for all $\theta \in [0,1]^d$ and $V \in S$,
\begin{equation*}
    \forall u_0 \in L^2(\mathbb T^d), \quad \left\| u_0 \right\|_{L^2(\mathbb T^d)}^2 \leq C{'} \int_0^T \int_{\mathbb T^d} a(z) \left| e^{-it\mathcal{H}_{\theta,V}} u_0(z)\right|^2 dz dt.
\end{equation*}
Once again, we proceed by contradiction and it provides sequences $(u_n)_{n \in \nn} \subset L^2(\mathbb T^d)$ with $\|u_n\|_{L^2(\mathbb T^d)}=1$ for all $n \in \nn$, $(\theta_n)_{n \in \nn} \subset [0,1]^d$ and $(V_n)_{n \in \nn} \subset S$, such that for all $n \in \nn^*$,
\begin{equation}\label{observability_contradiction}
    \int_0^T \int_{\mathbb T^d} b(z) \left| e^{-it\mathcal{H}_{\theta_n,V_n}} u_n(z)\right|^2 dz dt \leq \frac1n.
\end{equation}
Since $(u_n)_{n \in \nn}$ is bounded in $L^2(\mathbb T^d)$, there exists $f\in L^2(\mathbb T^d)$ such that, up to a subsequence, $(u_n)_{n \in \nn}$ weakly converges to $f$ in $L^2(\mathbb T^d)$ and strongly converges to $f$ in $H^{-2}(\mathbb T^d).$ Thanks to the weak observability estimate \eqref{weak_obs_estimate}, it follows that 
\begin{equation}
    \label{eq:fnonnul}
    1\leq C \|f\|^2_{H^{-2}(\mathbb T^d)}.
\end{equation}
On the other hand, since $(\theta_n)_{n \in \nn}$ is bounded and $S$ is $\text{weakly}^\star$ compact in $L^{\infty}(\mathbb T^d)$, there exist $\theta \in [0, 2\pi]^d$ and $V \in S$ such that, up to a subsequence:
$$\theta_n \underset{n \to +\infty}{\to} \theta \quad \text{and} \quad V_n \underset{n \to +\infty}{\rightharpoonup^\star} V \ \text{in} \ L^{\infty}(\mathbb T^d). $$
We then get that for all $t\in [0,T]$,
\begin{equation}
\label{eq:convweaktrotterkato}
    e^{-it\mathcal{H}_{\theta_n,V_n}} u_n \underset{n \to +\infty}{\rightharpoonup} e^{-it\mathcal{H}_{\theta,V}} f \ \text{weakly in} \ L^2.
\end{equation}
Indeed, for obtaining \eqref{eq:convweaktrotterkato}, we proceed as follows, for $\varphi \in L^2(\T^d)$,
\begin{multline*}
    \langle e^{-it\mathcal{H}_{\theta_n,V_n}} u_n - e^{-it\mathcal{H}_{\theta,V}} f, \varphi \rangle =  \langle (e^{-it\mathcal{H}_{\theta_n,V_n}} - e^{-it\mathcal{H}_{\theta,V}}) u_n , \varphi \rangle + \langle e^{-it\mathcal{H}_{\theta,V}} u_n - e^{-it\mathcal{H}_{\theta,V}} f,  \varphi \rangle\\
    =\langle u_n, (e^{it\mathcal{H}_{\theta_n,V_n}} - e^{it\mathcal{H}_{\theta,V}}) \varphi \rangle + \langle u_n -f , e^{it\mathcal{H}_{\theta,V}} \varphi \rangle,
\end{multline*}
the first term goes to $0$ as $n \to +\infty$ according to the stability result \eqref{eq:convergencesemigroup} and the second term goes to $0$ as $n \to +\infty$ by weak convergence. 

From \eqref{eq:convweaktrotterkato} and \eqref{observability_contradiction}, we then deduce that for all $t \in [0,T]$, 
\begin{equation*}
    \int_0^T \int_{\mathbb T^d} b(z) \left| e^{-it\mathcal{H}_{\theta,V}} f(z)\right|^2 dz dt \leq \liminf_{n \to +\infty}  \int_0^T \int_{\mathbb T^d} b(z) \left| e^{-it\mathcal{H}_{\theta_n,V_n}} u_n(z)\right|^2 dz dt = 0,
\end{equation*}
that implies that $f \in \mathcal N_{T, \theta, V}=\left\{0\right\}$. This contradicts \eqref{eq:fnonnul} and ends the proof of Proposition~\ref{frequency_cutoff_obs_prop}.\end{proof}

\section{Proof of the uniform observability inequality on $\T^2$}\label{section:proof_obs_torus}
This section is devoted to the proof of Theorem~\ref{obs_tore_thm}. It is adapted from the one given by Burq and Zworski in \cite{BZ19}, in the case $V=0$ and $\theta=0$. Let us recall that two main differences appear in comparison with \cite{BZ19}: the presence of the parameter $\theta \in [0,1]^2$ and the bounded potential $V \in L^{\infty}(\T^2)$ in the operator $\mathcal{H}_{\theta,V}$. These difficulties have already been handled in Sections \ref{sec:propsemigroupHthetaV},\ref{sec:refobsmultid}. Here, we continue keeping track of $\theta$ to ensure that the observability constants do not depend on $\theta$. 

In a first part, we prove one-dimensional observability estimates (for highly oscillating data) thanks to semi-classical defect measures. The second part consists in proving two-dimensional observability estimates (also for highly oscillating data) using semi-classical defect measures and a reduction of the dimension argument based on ergodicity arguments.

Before continuing, let us state a useful easy lemma, that will be used in the next two parts. It enables us to pass from $L^{\infty}$-observable sets to $L^{1}$-observable sets in the multi-dimensional setting.
\begin{lemma}
\label{lemma:linftyl1}
Let $T>0$ and $S \subset L^{\infty}(\T^d)$. Assume that for every $B \in L^{\infty}(\T^d) \setminus \{0\}$, there exists a positive constant $C>0$ such that for all $\theta \in [0, 2\pi]^d$ and $V\in S$,
\begin{equation}
\label{eq:ddimensionalobservabilityestimateB}
   \forall u_0 \in L^2(\mathbb T^d), \quad \|u_0\|^2_{L^2(\mathbb T^d)} \leq C \int_0^T \int_{\mathbb T^1} B(x)\left| e^{-it\mathcal{H}_{\theta,V}}u_0(x)\right|^2 \ dxdt.
\end{equation}
Then, for every $b \in L^{1}(\T^d) \setminus \{0\}$, there exists a positive constant $C>0$ such that for all $\theta \in [0, 2\pi]^d$ and $V\in S$,
\begin{equation}
\label{eq:ddimensionalobservabilityestimateb}
   \forall u_0 \in L^2(\mathbb T^d), \quad \|u_0\|^2_{L^2(\mathbb T^d)} \leq C \int_0^T \int_{\mathbb T^d} b(x)\left| e^{-it\mathcal{H}_{\theta,V}}u_0(x)\right|^2 \ dxdt.
\end{equation}
\end{lemma}
\begin{proof}
Indeed, if $b \in L^1(\T^d)\setminus\{0\}$ then, there exists $R>0$ such that $\left|\{x \in \T^d; \ b(x)<R\}\right|>0$. Thus, $B= \un_{(0,R)}(b) b \in L^{\infty}(\T^d)\setminus\{0\}$ and for all $u_0 \in L^2(\T^d)$, we have that
$$\int_0^T \int_{\mathbb T^d} B(x)\left| e^{-it\mathcal{H}_{\theta,V}}u_0(x)\right|^2 \ dxdt \leq \int_0^T \int_{\mathbb T^d} b(x)\left| e^{-it\mathcal{H}_{\theta,V}}u_0(x)\right|^2 \ dxdt.$$
Conjugating the previous estimate with \eqref{eq:ddimensionalobservabilityestimateB} leads to \eqref{eq:ddimensionalobservabilityestimateb}.
\end{proof}

\subsection{One dimensional observability estimates}\label{Obs_1D_section}
This section is devoted to prove the following one dimensional weak observability estimates for smooth potentials.
\begin{proposition}\label{obs_schro_1D}
Let $b \in L^{\infty}(\mathbb T^1)\setminus\{0\}$ be a non-negative function, $M>0$ and $T>0$. There exists a positive constant $C>0$ such that for all $\theta \in [0, 1]$ and $V\in \mathcal C^{\infty}(\mathbb T^1)$ with $\| V\|_{L^{\infty}} \leq M$,
\begin{equation*}
\label{eq:onedimensionalobservabilityestimate}
   \forall u_0 \in L^2(\mathbb T^1), \quad \|u_0\|^2_{L^2(\mathbb T^1)} \leq C \int_0^T \int_{\mathbb T^1} b(x)\left| e^{-it\mathcal{H}_{\theta,V}}u_0(x)\right|^2 \ dxdt + C \|u_0\|^2_{H^{-2}(\T^1)}.
\end{equation*}
\end{proposition}
Recall the notation of the operator $\mathcal{H}_{\theta,V} = -\partial_{x}^2 + 2 i \theta \cdot \partial_{x} + |\theta|^2  +V(x)$ that would be used in the sequel of this part.
\begin{proof}
Thanks to Proposition~\ref{frequency_cutoff_obs_prop}, it is sufficient to establish that there exist some positive constants $C=C(T,M)>0$, $\rho_0=\rho_0(T)>0$ and $h_0=h_0(T)>0$ such that for all $\theta \in [0,1]$ and $V \in \mathcal C^{\infty}(\mathbb T^1)$ with $\| V\|_{L^{\infty}} \leq M$, we have
\begin{multline*}
   \forall 0<h\leq h_0, \forall 0< \rho \leq \rho_0, \forall u_0 \in L^2(\mathbb T^1), \\ \left\| \Pi_{h, \rho,\theta, V}u_0 \right\|_{L^2(\mathbb T^1)}^2 \leq C \int_0^T \int_{\mathbb T^1} b(x) \left| e^{-it\mathcal{H}_{\theta,V}} (\Pi_{h, \rho,\theta, V}u_0)(x)\right|^2 dx dt.
\end{multline*}
To that end, let us proceed by contradiction. Assume that there exist some sequences $(\rho_h)_{h >0} \subset \rr_+^*$, $(V_h)_{h>0} \subset \mathcal C^{\infty}(\mathbb T^1)$, $(\theta_h)_{h >0} \subset [0,1]$ and $(u_h)_{h >0} \subset L^2(\mathbb T^1)$ satisfying:
\begin{equation*}
    \rho_h \underset{h\to 0^+}{\longrightarrow} 0, \quad \| V_h\|_{L^{\infty}} \leq M, \quad v_h=\Pi_{h, \rho_h,\theta_h, V_h}u_h\quad \text{with}\quad \|v_h\|_{L^2}=1
\end{equation*}
and
\begin{equation}
\label{eq:convsupportmuzero1D}
    \int_0^T \int_{\mathbb T^1} b(x) \left| e^{-it\mathcal{H}_{\theta_h,V_h}} v_h(x)\right|^2 dx dt \underset{h \to 0^+}{\longrightarrow}0.
\end{equation}
By definition, the family $(u_h)_{h>0}$ satisfies the $h$-oscillating property \eqref{h-oscillating}. By applying Proposition~\ref{semiclassical_measure_prop}, there exists a finite measure $\mu \in L^{\infty}(\rr, \mathcal M_+(T^*\mathbb T))$ such that, up to a subsequence: for all $\varphi \in L^1(\mathbb R)$ and $a \in \mathcal C^{\infty}_c(\mathbb T\times \rr)$,
\begin{equation}\label{defect_measure_1D}
    \lim_{h \to 0^+} \int_{\rr} \varphi(t) \langle \ops ha e^{-i t\mathcal H_{\theta_h, V_h}} v_h, e^{-it \mathcal H_{\theta_h, V_h}} v_h\rangle_{L^2(\mathbb T^d)} dt= \int_{\rr \times T^*\mathbb T^1} \varphi(t) a(x,\xi) \mu(t, dx, d\xi)dt,
\end{equation}
and the measure $\mu$ satisfies
\begin{equation}\label{flow_invariance}
    \forall s \in \rr, \quad \partial_s\int_{\rr}\int_{T^*\mathbb T^1}\varphi(t) a(x+s\xi, \xi) \mu(t, dx, d\xi)dt=0,
\end{equation}
for all $\varphi \in L^1(\mathbb R)$ and $a \in \mathcal C^{\infty}_c(\mathbb T\times \rr)$.
Moreover, the following properties hold
\begin{equation*}\label{support}
   \forall t_0,t_1 \in \rr,\ \mu([t_0,t_1]\times \mathbb T^1 \times \rr)=|t_1-t_0|\quad \text{and} \quad \Supp \mu \subset \rr \times \mathbb T^1 \times \{-1, 1\}. 
\end{equation*}
Let us define the measure $$\mu_T(dx)=\int_{0}^T\int_{\rr} \mu(t, dx, d\xi)dt.$$

As a consequence of Proposition~\ref{strichartz_1D_prop}, we show that 
$$\int_{\mathbb T} b(x) \mu_T(dx)=0.$$ To that end, let us first check that there exists $f\in L^{\infty}(\mathbb T)$ such that $\mu_T=f(x)dx$. Indeed, we obtain from \eqref{defect_measure_1D} and \eqref{estimation_strichartz_1D} that for all $a\in \mathcal C_c^{\infty}(\mathbb T)$ 
$$\int_{\mathbb T} a(x) d\mu_T(dx) = \lim_{h \to 0^+} \int_{\rr}\int_{\mathbb T} a(x) \left| e^{-i t\mathcal H_{\theta_h, V_h}} u_h\right|^2(x) dxdt \leq C \| a\|_{L^1(\mathbb T)},$$ and then, $\mu_T \in \left(L^1(\mathbb T)\right)' = L^{\infty}(\mathbb T)$. Now, let $(b_n)_{n \in \mathbb N} \in \mathcal C_c^{\infty}(\mathbb T)^{\mathbb N}$ be a sequence converging to $b$ in $L^1(\mathbb T).$ Notice that since $\mu_T=f(x)dx \in L^{\infty}$, we readily have $\int_{\mathbb T} b(x) \mu_T(dx) = \lim_{n \to +\infty} \int_{\mathbb T} b_n(x) \mu_T(dx) $. On the other hand, 
\begin{align*}
    \left|\int_{\mathbb T} b_n(x) \mu_T(dx)\right| & = \lim_{h \to 0^+} \left|\int_{0}^T\int_{\mathbb T} b_n(x) \left| e^{-i t\mathcal H_{\theta_h, V_h}} v_h\right|^2(x) dxdt\right| \\ & \leq \lim_{h \to 0^+} \int_{0}^T\int_{\mathbb T} b(x) \left| e^{-i t\mathcal H_{\theta_h, V_h}} v_h\right|^2(x) dxdt + C \|b-b_n\|_{L^1(\mathbb T)}.
\end{align*}
We deduce from the last inequality and \eqref{eq:convsupportmuzero1D} that $$\int_{\mathbb T} b(x) \mu_T(dx)=\lim_{n \to +\infty}  \left|\int_{\mathbb T} b_n(x) \mu_T(dx)\right|=0.$$
Finally, let us notice that the invariance property \eqref{flow_invariance} leads to 
$$\partial_x \mu_T=0.$$
Let us check this fact. Thanks to the fact that $$\Supp \mu \subset \rr \times \mathbb T^1 \times \{-1, 1\},$$ 
it is sufficient to prove that for all $a\in \mathcal C^{\infty}_c(\mathbb T^1 \times \rr)$ with $$(\Supp a) \cap (\mathbb T^1 \times \{-1, 1\})\subset \mathbb T^1 \times \{-1\} \quad \text{or} \quad (\Supp a) \cap (\mathbb T^1 \times \{-1, 1\})\subset \mathbb T^1 \times \{ 1\},$$
and for all $\phi \in L^1(\rr)$,
$$\int_{\rr} \int_{\mathbb T^1\times \rr} \phi(t) \partial_x a(x, \xi) \mu(t, dx, d\xi) =0.$$
For example, let us deal with the case when $(\Supp a) \cap (\mathbb T^1 \times \{-1, 1\})\subset \mathbb T^1 \times \{1\}$. 
Thanks to the invariance property \eqref{flow_invariance}, we have for all $\phi \in L^1(\rr)$ and for all $s \in \rr$,
\begin{align*} 
\int_{\rr} \int_{\mathbb T^1\times \rr} \phi(t) \partial_x a(x, \xi) d\mu(t, dx, d\xi) & =\int_{\rr} \int_{\mathbb T^1\times \{1\}} \phi(t) \partial_x a(x+s\xi, \xi) \mu(t, dx, d\xi)\\
& = \int_{\rr} \int_{\mathbb T^1\times \{1\}} \phi(t) \partial_x a(x+s,1) \mu(t, dx, d\xi).
\end{align*}
By using anew the invariance property \eqref{flow_invariance} and the fact that $\partial_x a(x+s, 1)=\partial_s a(x+s,1)$, it follows that
\begin{align*}\int_{\rr} \int_{\mathbb T^1\times \rr} \phi(t) \partial_x a(x, \xi) d\mu(t, dx, d\xi) & = \int_{\rr} \int_{\mathbb T^1\times \{1\}} \phi(t) \partial_s a(x+s, \xi) \mu(t, dx, d\xi)\\
& =\partial_s\int_{\rr} \int_{\mathbb T^1\times \{1\}} \phi(t) a(x+s\xi, \xi) \mu(t, dx, d\xi)\\
\int_{\rr} \int_{\mathbb T^1\times \rr} \phi(t) \partial_x a(x, \xi) d\mu(t, dx, d\xi)&=0.
\end{align*}
This proves that 
$\partial_x \mu_T=0,$
which means that $\mu_T=c dx$, with $c>0$ since $\mu_T(\mathbb T^1)= T$.

This implies, in particular, that $$c \|b\|_{L^1(\mathbb T)} = \int_{\mathbb T} b(x) \mu_T(x)= 0.$$
This is a contradiction and this ends the proof of Proposition~\ref{obs_schro_1D}.
\end{proof}
Thanks to Proposition~\ref{obs_schro_1D}, we can establish the following result which provides one-dimensional observability estimates for Schrödinger equations with $L^{\infty}$-potential and source term:

\begin{corollary}\label{cor:obs_1D}
Let $b \in L^1(\mathbb T^1)\setminus \{0\}$ be a non-negative function, $M>0$, and $T>0$. There exists a positive constant $C$ such that for all $\theta \in [0,1],$ $V \in L^{\infty}(\mathbb T^1)$ with $\|V\|_{L^{\infty}(\mathbb T^1)} \leq M$, $u_0 \in L^2(\mathbb T^1)$ and $f \in L^1((0,T), L^2(\mathbb T^1))$, the mild solution $u$ to 
\begin{equation*}
		\left\{
			\begin{array}{ll}
				i  \partial_t u  = (-\partial_{x}^2 + 2 i \theta \cdot \partial_{x} + |\theta|^2  +V(x)) u +f & \text{ in }  (0,T) \times \T^1, \\
				u(0, \cdot) = u_0 & \text{ in } \T^1.
			\end{array}
		\right.
\end{equation*}
satisfies the observability estimate
\begin{equation*}
    \|u\|^2_{L^{\infty}((0,T);L^2(\T^1))} \leq C\left(\int_0^T \int_{\T^1} b(x) |u(t,x)|^2 dxdt + \|f\|^2_{L^1((0,T), L^2(\mathbb T^1))}\right).
\end{equation*}
\end{corollary}
\begin{proof}
Let us first deal with the case $f=0$ and $b \in L^{\infty}(\T^1)$. According to Proposition~\ref{prop:from_weakobs_to_strongobs}, since $B_{L^{\infty}}(0,M)$ is $\text{weakly}^\star$ compact in $L^{\infty}(\T^1)$, it is sufficient to show that there exists $C>0$ such that for all $\theta \in [0,1]$ and $V \in L^{\infty}(\T^1)$ with $\| V\|_{L^{\infty}(\T^1)} \leq M$,
$$\forall u_0 \in L^2(\T^1), \quad \|u_0\|^2_{L^2(\T^1)} \leq C\int_0^T\int_{\T^1} b(x) |u(t,x)|^2 dx\, dt + C\|u_0\|^2_{H^{-2}(\T^1)}.$$
From Proposition~\ref{obs_schro_1D}, there exists a positive constant $C>0$ such that for all $\theta \in [0, 1]$, $V\in \mathcal C^{\infty}(\mathbb T^1)$ with $\| V\|_{L^{\infty}} \leq M$ and $u_0 \in L^2(\T^1)$,
\begin{equation}
\label{eq:onedimensionalobservabilityestimate2}
  \|u_0\|^2_{L^2(\mathbb T^1)} \leq C \int_0^T \int_{\mathbb T^1} b(x)\left| e^{-it(-\partial_{x}^2 + 2 i \theta \cdot \partial_{x} + |\theta|^2  +V(x))}u_0(x)\right|^2 \ dxdt + C \|u_0\|^2_{H^{-2}(\T^1)}.
\end{equation}
Let $V \in L^{\infty}(\T^1)$. One can find a sequence $(V_n)_{n \in \nn} \subset \mathcal C^{\infty}(\T^1)$ satisfying 
$$\forall n \in \nn, \quad \| V_n\|_{L^{\infty}(\T^1)} \leq M \quad \text{and} \quad V_n \underset{n \to +\infty}{\longrightarrow} V \quad \text{for the weak}^\star \text{ topology of } L^{\infty}.$$
Moreover, we obtain from \eqref{eq:onedimensionalobservabilityestimate2} that for all $n \in \nn$, $\theta \in [0,1]$ and $u_0 \in L^2(\T^1)$,
\begin{equation*}
\label{eq:onedimensionalobservabilityestimate3}
  \|u_0\|^2_{L^2(\mathbb T^1)} \leq C \int_0^T \int_{\mathbb T^1} b(x)\left| e^{-it(-\partial_{x}^2 + 2 i \theta \cdot \partial_{x} + |\theta|^2  +V_n(x))}u_0(x)\right|^2 \ dxdt + C \|u_0\|^2_{H^{-2}(\T^1)}.
\end{equation*}
We therefore deduce from the above estimates, together with the stability result given by Proposition~\ref{eq:convergencesemigroup} and the dominated convergence Theorem, that for all $\theta \in [0,1]$ and $u_0 \in L^2(\T^1)$,
$$\|u_0\|^2_{L^2(\T^1)} \leq C \int_0^T \int_{\T^1} b(x) \left| e^{-it(-\partial_{x}^2 + 2 i \theta \cdot \partial_{x} + |\theta|^2  +V(x))}u_0(x)\right|^2 \ dxdt+ C\|u_0\|_{H^{-2}(\T^1)}.$$
Since $B_{L^{\infty}}(0,M)$ is $\text{weakly}^\star$ compact in $L^{\infty}(\T^1)$, we are now able to conclude from Proposition~\ref{prop:from_weakobs_to_strongobs} that for all $\theta \in [0,1]$, $V \in L^{\infty}(\T^1)$ with $\|V\|_{L^{\infty}}\leq M$ and $u_0 \in L^2(\T^1)$,
\begin{equation}\label{eq:obs_1D_without_source_term}\|u_0\|^2_{L^2(\T^1)} \leq C \int_0^T \int_{\T^1} b(x) \left| e^{-it(-\partial_{x}^2 + 2 i \theta \cdot \partial_{x} + |\theta|^2  +V(x))}u_0(x)\right|^2 \ dxdt.
\end{equation}
Notice that, thanks to Lemma~\ref{lemma:linftyl1}, the observability estimates \eqref{eq:obs_1D_without_source_term} holds true for $b \in L^1(\T^1)$.

Let us now consider the general case. We split $u$ into two terms, the one coming from the initial data and the second coming from the source term, we have
\begin{equation}
\label{eq:duhamelformulaproofObsInhom}
  u(t)= e^{-i t \mathcal{H}_{\theta,V}} u_0 + \int_0^t e^{-i(t-s)\mathcal{H}_{\theta,V}} f(s) ds.  
\end{equation}
Taking the square of the $L^2(\T^1)$-norm on both sides of the previous equality, we have that
\begin{equation*}
    \norme{u(t)}_{L^2(\T^1)}^2 \leq 2 \norme{e^{-i t \mathcal{H}_{\theta,V}} u_0}_{L^2(\T^1)}^2 + 2 \norme{\int_0^t e^{-i(t-s)\mathcal{H}_{\theta,V}} f(s) ds}_{L^2(\T^1)}^2.
\end{equation*}
Then, from the Minkowski inequality, the conservation of the $L^2$-norm \eqref{eq:conservationL2norm}, the observability inequality for the homogeneous equation \eqref{eq:obs_1D_without_source_term}, we deduce that 
\begin{align*}
    \norme{u(t)}_{L^2(\T^1)}^2 & \leq 2 \norme{u_0}_{L^2(\T^1)}^2 + 2 \left(\int_{0}^t \norme{ f(s)}_{L^2(\T^1)} ds\right)^2
   \notag\\ & \leq C \int_{0}^T \int_{\T^1} b(x) |e^{-i t' \mathcal{H}_{\theta,V}} u_0|^2 dt' dx + C  \|f\|^2_{L^1((0,T),L^2(\T^1))}.
\end{align*}
Now we plug the Duhamel formula \eqref{eq:duhamelformulaproofObsInhom} in the right hand side of the previous estimate, using Fubini theorem, we obtain for all $t \in (0,T)$
\begin{align*}
   & \norme{u(t)}_{L^2(\T^1)}^2\\
    &\leq C \int_{0}^T \int_{\T^1} b(x) |u(t')|^2 dt' dx 
   + C \int_{0}^T \int_{\T^1} b(x) \left|\int_0^{t'} e^{-i(t'-s)\mathcal{H}_{\theta,V}} f(s) ds\right|^2 dt' dx + C \|f\|^2_{L^1((0,T),L^2(\T^1))}\\
   & \leq C \int_{0}^T \int_{\T^1} b(x) |u(t')|^2 dt' dx + C \norme{b}_{L^1(\T^1)} \norme{\int_0^{t'} e^{-i(t'-s)\mathcal{H}_{\theta,V}} f(s) ds}_{L^{\infty}(\T^1;L^2(0,T))}^2 + C \|f\|^2_{L^1((0,T),L^2(\T^1))}.
\end{align*}
Moreover,
the Strichartz estimates \eqref{eq:strichartz1Dsource} given by Proposition~\ref{prop:StrichartzEstimate1d} shows that 
\begin{equation*}
\norme{\int_0^{t'} e^{-i(t'-s)\mathcal{H}_{\theta,V}} f(s) ds}_{L^{\infty}(\T^1;L^2(0,T))}^2 \leq C \|f\|^2_{L^1((0,T), L^2(\mathbb T^1))}.
\end{equation*}
Then, it follows that
\begin{equation*}
\norme{u}^2_{L^{\infty}(0,T; L^2(\T^1))}\leq C  \int_{0}^T \int_{\T^1} b(x) |u(t)|^2 dt dx + C \|f\|^2_{L^1((0,T), L^2(\mathbb T^1))},
\end{equation*}
which concludes the proof.
\end{proof}

\subsection{Two-dimensional observability estimates}
This section is devoted to the proof of Theorem~\ref{obs_tore_thm}. It is divided into two parts. The first one establishes uniform observability estimates for smooth potentials, belonging to a relative compact subset of $L^4(\T^2)$. In the second part, we deduce the observability estimates for $L^{\infty}$-potentials from the smooth case and the stability result of Proposition \ref{prop:stabilityparameters}. Without loss of generality, according to Lemma \ref{lemma:linftyl1}, we can assume that $b \in L^{\infty}(\T^2)$.

\subsubsection{Observability estimates for smooth potentials}
In this section, we prove the following proposition which provides weak observability estimates for smooth potentials:
\begin{proposition}\label{obs_tore_smooth_potential}
Let $T>0$, $M>0$, $b \in L^{\infty}(\mathbb T^2) \setminus \{0\}$ be a non-negative function and $\mathcal K \subset \mathcal C^{\infty}(\T^2)\cap B_{L^{\infty}}(0,M)$ be a relatively compact subset of $L^4(\T^2)$. There exists a positive constant $C=C(T,b, M, \mathcal K) > 0$ such that for every potential $V \in \mathcal K$, $\theta \in [0,1]^2$ and $v_0 \in L^2(\T^2)$,
\begin{equation*}
    \label{eq:ObservabilitySchrodingerToreWeak}
    \norme{v_0}_{L^2(\T^2)}^2 \leq C \int_0^T \int_{\mathbb T^2} b(z) |e^{-it\mathcal H_{\theta, V}}v(t,z)|^2 dz dt +C \|u_0\|^2_{H^{-2}(\T^2)}.
\end{equation*}
\end{proposition}

Let $M>0$, $T>0$ and $\mathcal K \subset \mathcal C^{\infty}(\T^2)\cap B_{L^{\infty}}(0,M)$ be a relatively compact subset of $L^4(\T^2)$. Thanks to Proposition~\ref{frequency_cutoff_obs_prop}, it is sufficient to establish that there exist some positive constants $C=C(T,b, M, \mathcal K)>0$, $\rho_0=\rho_0(T)>0$ and $h_0=h_0(T)>0$ such that for all $\theta \in [0,1]^2$ and $V \in \mathcal K$, we have
\begin{multline*}
   \forall 0<h\leq h_0, \forall 0< \rho \leq \rho_0, \forall u_0 \in L^2(\mathbb T^2), \\ \left\| \Pi_{h, \rho,\theta, V}u_0 \right\|_{L^2(\mathbb T^2)}^2 \leq C \int_0^T \int_{\mathbb T^2} b(z) \left| e^{-it\mathcal H_{\theta, V}} (\Pi_{h, \rho,\theta, V}u_0)(z)\right|^2 dz dt,
\end{multline*}
where $\mathcal H_{\theta, V}=-\Delta+2i\theta\cdot\nabla + |\theta|^2 +V$.

To that end, let us proceed by contradiction. Assume that there exist some sequences $(h_n)_{n \in \nn} \subset (0,1]$, $(\rho_n)_{n \in \nn} \subset \rr_+^*$, $(V_n)_{n \in \nn} \subset \mathcal K$, $(\theta_n)_{n \in \nn} \subset [0,1]^2$ and $(u_n)_{n \in \nn} \subset L^2(\mathbb T^2)$ satisfying:
\begin{equation*}
    \rho_n \underset{n \to +\infty}{\longrightarrow} 0, \quad v_n=\Pi_{h_n, \rho_n,\theta_n, V_n}u_n\quad \text{with}\quad \|v_n\|_{L^2}=1
\end{equation*}
and
\begin{equation}\label{eq:contradiction_obs}
    \int_0^T \int_{\mathbb T^2} b(z) \left| e^{-it\mathcal H_{\theta_n, V_n}} v_n(z)\right|^2 dz dt \underset{n \to +\infty}{\longrightarrow}0.
\end{equation}
Since $\mathcal K$ is relatively compact in $L^4(\T^2)$, up to a subsequence, $(V_n)_{n \in \nn}$ is a Cauchy sequence in $L^4(\T^2)$. Let us consider a small parameter $\delta>0$, to be chosen later. There exists $p_{\delta} \in \nn$ such that
\begin{equation}\label{eq:Vn_cauchy_sequence}
\|V_n -V_{p_\delta}\|_{L^4(\T^2)} \leq \delta\qquad \forall n \geq p_{\delta}.
\end{equation}
By definition, the family $(v_n)_{n \in \nn}$ satisfies the $h_n$-oscillating property \eqref{h-oscillating}. We can therefore apply Proposition~\ref{semiclassical_measure_prop} which provides a finite measure $\mu \in L^{\infty}(\rr, \mathcal M_+(T^*\mathbb T^2))$ such that, up to a subsequence: for all $\varphi \in L^1(\mathbb R)$ and $a \in \mathcal C^{\infty}_c(\mathbb T^2\times \rr^2)$,
\begin{equation}\label{defect_measure_2D}
    \lim_{n \to +\infty} \int_{\rr} \varphi(t) \langle \ops {h_n}a e^{-i t\mathcal H_{\theta_n, V_n}} v_n, e^{-it \mathcal H_{\theta_n, V_n}} v_n\rangle_{L^2(\mathbb T^2)} dt= \int_{\rr \times T^*\mathbb T^2} \varphi(t) a(z,\xi) \mu(t, dz, d\zeta)dt.
\end{equation}
Moreover, the measure $\mu$ satisfies
\begin{equation}\label{flow_invariance_2D}
    \forall s \in \rr, \quad \partial_s\int_{\rr}\int_{T^*\mathbb T^2}\varphi(t) a(z+s\zeta, \zeta) \mu(t, dz, d\zeta)dt=0,
\end{equation}
for all $\varphi \in L^1(\mathbb R)$ and $a \in \mathcal C^{\infty}_c(\mathbb T^2\times \rr^2)$.
Furthermore, the following properties hold
\begin{equation*}\label{support_2D}
   \forall t_0,t_1 \in \rr,\ \mu([t_0,t_1]\times \mathbb T^2 \times \rr^2)=|t_1-t_0|\quad \text{and} \quad \Supp \mu \subset \rr \times \mathbb T^2 \times \mathbb S^1. 
\end{equation*}
Let us define the measure $\mu_T \in \mathcal M_+(T^*\mathbb T^2)$ by
$$\mu_T(dz, d\zeta)=\int_0^T \mu(t, dz, d\zeta) dt.$$
We divide the remainder of the proof into six steps.\\

\textit{First step: regularity property for $\mu_T$.}\\
We have from \eqref{defect_measure_2D}, together with Proposition~\ref{semiclassical_measure_prop} assertion \textit{(ii)}, that for all $a \in \mathcal C^{\infty}_c(\mathbb T^2)$,
$$\int_{\mathbb T^2\times \rr^2} a(z) \mu_T(dz, d\zeta) = \lim_{n \to +\infty} \int_{0}^T\int_{\mathbb T^2} a(z) \left| e^{-it\mathcal H_{\theta_n, V_n}} v_n(z)\right|^2 dzdt.$$
Consequently, we obtain for all $a \in \mathcal C^{\infty}_c(\mathbb T^2)$,
\begin{align*}\left|\int_{\mathbb T^2\times \rr^2} a(z) \mu_T(dz, d\zeta)\right| & = \lim_{n \to +\infty} \left|\int_{0}^T\int_{\mathbb T^2} a(z) \left| e^{-it\mathcal H_{\theta_n, V_n}} v_n(z)\right|^2 dz\right| \\
& \leq \|a\|_{L^2(\mathbb T^2)} \left\|e^{-it\mathcal H_{\theta_n, V_n}} v_n\right\|^2_{L^4(\mathbb T^2, L^2(0,T))} \\
& \leq C_T \|a \|_{L^2(\mathbb T^2)},
\end{align*}
where the last inequality follows from the Strichartz type estimates given by Proposition~\ref{prop:StrichartzEstimate}.
It follows from the Riesz representation Theorem that there exists a non-negative function $g_T \in L^2(\mathbb T^2)$ such that for all $a \in \mathcal C^{\infty}_c(\mathbb T^2)$,
$$\int_{\mathbb T^2\times \rr^2} a(z) \mu_T(dz, d\zeta) =\int_{\mathbb T^2} a(z) g_T(z) dz.$$

Now, let $(b_n)_{n \in \mathbb N} \in \mathcal C_c^{\infty}(\mathbb T^2)^{\mathbb N}$ be a sequence converging to $b$ in $L^2(\mathbb T^2).$ Notice that since $g_T\in L^{2}$, we readily have $\int_{\mathbb T^2} b(z) g_T(z) dz = \lim_{n \to +\infty} \int_{\mathbb T} b_n(z) g_T(z) dz $. On the other hand, 
\begin{align*}
    \left|\int_{\mathbb T^2} b_n(z) g_T(z) dz\right| & = \lim_{h \to 0^+} \left|\int_{0}^T\int_{\mathbb T} b_n(z) \left| e^{-i t\mathcal H_{\theta_h, V_h}} v_h\right|^2(z) dzdt\right| \\ & \leq \lim_{h \to 0^+} \int_{0}^T\int_{\mathbb T} b(z) \left| e^{-i t\mathcal H_{\theta_h, V_h}} v_h\right|^2(z) dzdt + C \|b-b_n\|_{L^1(\mathbb T^2)}.
\end{align*}
We deduce from the last inequality and \eqref{eq:contradiction_obs} that 
\begin{equation}\label{dfct_meas_null}\int_{\mathbb T^2\times \rr^2} b(z) \mu_T(dz, d\zeta)=\int_{\mathbb T^2} b(z) g_T(z) dz=0.
\end{equation}

\textit{Second step: Ergodicity property.}
We define the set $\Sigma_{\mathbb R \setminus \mathbb Q}$ of irrational directions on the torus $\mathbb T^2$
\begin{equation*}
    \Sigma_{\mathbb R \setminus \mathbb Q}:=\left\{(z, \zeta)\in T^*\mathbb T^2; \ |\zeta|=1, \, \mathbb Z^2 \cap \{\zeta\}^{\perp}=\{0\}\right\},
\end{equation*}
and $\Sigma_{\mathbb Q}=\Sigma_{\mathbb R \setminus \mathbb Q}^c$ the set of rational directions.
The set $\Sigma_{\mathbb R \setminus \mathbb Q}$ is clearly invariant by the flow:

\begin{equation}\label{irr_dir_flow}
\forall (z, \zeta) \in \Sigma_{\mathbb R \setminus \mathbb Q}, \forall s \in \rr, \quad (z+s\zeta,\zeta) \in \Sigma_{\mathbb R \setminus \mathbb Q}.
\end{equation}
Let us show that $\mu_T(\Sigma_{\mathbb R \setminus \mathbb Q})=0.$ Let $(b_p)_{p \in \nn} \subset \mathcal C_c^{\infty}(\mathbb T^2;\R^+)$ such that $b_p \underset{p \to +\infty}{\longrightarrow} b$ in $L^2(\mathbb T^2)$. Clearly, we have $$\langle b_p\rangle\underset{p \to +\infty}{\longrightarrow} \langle b \rangle,$$
where $\langle b \rangle = \int_{\mathbb T^2} b(x) dx.$
Since $\langle b \rangle =\|b\|_{L^1(\mathbb T^2)}>0$, for $p$ sufficiently large, we have $\langle b_p\rangle \geq\frac{\|b\|_{L^1(\mathbb T^2)}}{2}>0$. Furthermore, by unique ergodicity of the flow $z \longmapsto z+s \zeta$, the following convergence holds: for all $p \in \nn$,
$$\forall (z, \zeta) \in \Sigma_{\mathbb R \setminus \mathbb Q}, \quad \langle b_p\rangle_S(z, \zeta):=\frac 1S \int_{0}^S b_p(z+s \zeta)ds \underset{S \to +\infty}{\longrightarrow} \langle b_p\rangle.\footnote{To prove this convergence, one can check it for trigonometric polynomials and then conclude by a density argument.}$$
Consequently, we obtain from the Fatou's lemma that
\begin{align}\nonumber
  0  \leq \mu_T(\Sigma_{\mathbb R \setminus \mathbb Q}) \langle b_p \rangle  = \int_{\Sigma_{\mathbb R \setminus \mathbb Q}} \langle b_p\rangle \mu_T(dz, d\zeta)
    & = \int_{\Sigma_{\mathbb R \setminus \mathbb Q}} \liminf_{S \to +\infty} \langle b_p\rangle_S(z, \zeta) \mu_T(dz, d\zeta)\\
    & \leq \liminf_{S\to +\infty}\int_{\Sigma_{\mathbb R \setminus \mathbb Q}} \langle b_p\rangle_S(z, \zeta) \mu_T(dz, d\zeta).\label{dfct_meas_irr_dir1}
\end{align}
By \eqref{flow_invariance_2D} and \eqref{irr_dir_flow}, $\mu_T(dz, d\zeta)$ and $\Sigma_{\mathbb R \setminus \mathbb Q}$ are invariant by the flow, it follows that for all $S>0$,
\begin{equation}\label{dfct_meas_irr_dir2} \int_{\Sigma_{\mathbb R \setminus \mathbb Q}} \langle b_p\rangle_S(z, \zeta) \mu_T(dz, d\zeta)= \frac 1S\int_0^S \int_{\Sigma_{\mathbb R \setminus \mathbb Q}} b_p(z+s\zeta, \zeta) \mu_T(dz, d\zeta)= \int_{\Sigma_{\mathbb R \setminus \mathbb Q}} b_p(z) \mu_T(dz, d\zeta).
\end{equation}
On the other hand, \eqref{dfct_meas_null} shows that
\begin{equation}\label{dfct_meas_irr_dir3}
0 \leq \left|\int_{\Sigma_{\mathbb R \setminus \mathbb Q}} b_p(z) \mu_T(dz, d\zeta)\right|= \left|\int_{\Sigma_{\mathbb R \setminus \mathbb Q}} b_p(z)-b(z) \mu_T(dz, d\zeta)\right|\leq \int_{\mathbb T^2} |b_p(z)-b(z)| g_T(z)dz \underset{p\to+\infty}{\longrightarrow} 0.
\end{equation}
Finally, by gathering \eqref{dfct_meas_irr_dir1}, \eqref{dfct_meas_irr_dir2} and \eqref{dfct_meas_irr_dir3}, we obtain that 
$$0\leq  \mu_T(\Sigma_{\mathbb R \setminus \mathbb Q}) \leq \frac {2}{\|b\|_{L^1}} \left|\int_{\Sigma_{\mathbb R \setminus \mathbb Q}} b_p(z) \mu_T(dz, d\zeta)\right| \underset{p \to +\infty}{\longrightarrow} 0,$$
providing $\mu_T(\Sigma_{\mathbb R \setminus \mathbb Q})=0.$
Consequently, $\mu_T(\Sigma_{\mathbb Q})=T>0$ and since
$\left\{\zeta \in \mathbb S^1; \ \exists z\in \mathbb T^2, (z, \zeta) \in \Sigma_{\mathbb Q} \right\}$
is a countable set, there exists $\zeta_0 \in \rr^2$ such that
$$\mu_T(\mathbb T^2 \times \{\zeta_0\})>0 \quad \text{and} \quad \zeta_0=\frac{(n,m)}{\sqrt{n^2+m^2}} \quad \text{for some}\quad (n,m) \in \mathbb Z^2.$$

\textit{Third step: a change of variables.}
This step is devoted to show that, up to a change of variables, we can assume that $\zeta_0=(0,1)$. 
Let $F : \rr^2 \longrightarrow \rr^2$ be the isometry defined by $$\forall (x_1, x_2) \in \rr^2, \quad F(x_1, x_2)= x_1 \xi_0^{\perp} +x_2 \zeta_0,$$
where $\zeta_0^{\perp}=\frac{(-m,n)}{\sqrt{n^2+m^2}}$.
One can readily verify that for any function $u$ periodic with respect to $(2\pi \mathbb Z)^2$, the function $F^*u$ is periodic with respect to $(A \mathbb Z)^2$, with $A=2\pi\sqrt{n^2+m^2}$. Moreover, if $u(t,\cdot)$ is solution to the Schrödinger equation:
\begin{equation*}
		\left\{
			\begin{array}{ll}
				i  \partial_t u  = (-\Delta + 2 i \theta \cdot \nabla  + |\theta|^2+ V(y)) u & \text{ in }  (0,T) \times \T^2, \\
				u(0, \cdot) = u_0 & \text{ in } \T^2,
			\end{array}
		\right.
\end{equation*}
then, $v(t,\cdot)=F^*u(t, \cdot)$ is solution to the Schrödinger equation posed on $\mathbb R^2/(A \mathbb Z)^2$:
\begin{equation}\label{eq:SchrodingerEq_DilatedTorus}
		\left\{
			\begin{array}{ll}
				i  \partial_t v  = (-\Delta + 2 i F^{-1}(\theta) \cdot \nabla + |\theta|^2+ F^*V(y)) v & \text{ in }  (0,T) \times \mathbb R^2/(A \mathbb Z)^2, \\
				v(0, \cdot) = F^*u_0 & \text{ in } \mathbb R^2/(A \mathbb Z)^2.
			\end{array}
		\right.
\end{equation}
As a consequence, if we define $w_h= F^* v_h$ then, 
$$\forall t \in \rr, \quad F^*\left(e^{-it(-\Delta+2i\theta_n \cdot\nabla+ |\theta_n|^2+ V_n)} v_n\right)= e^{-it(-\Delta+2iF^{-1}(\theta_n) \cdot\nabla+ |\theta_n|^2+ F^*V_n)} w_n.$$
In the following, the new torus $\mathbb R^2/(A \mathbb Z)^2$ is still denoted $\mathbb T^2$. 
Up to a subsequence, associated to this family of solutions of the Schrödinger equation \eqref{eq:SchrodingerEq_DilatedTorus} with initial data $(w_n)_{n \in \nn}$ is a semiclassical defect measure $\nu \in L^{\infty}(\mathbb R, \mathcal M_+(T^*\mathbb T^2))$ satisfying:
\begin{itemize}
    \item  for all $t_0,t_1 \in \rr,\ \nu([t_0,t_1]\times \mathbb T^2 \times \rr^2)=A^2|t_1-t_0|$ and $\Supp \nu \subset \rr \times \mathbb T^2 \times \mathbb S^1$,
    \item $\nu_T(\mathbb T^2\times \{(0,1)\})>0$, with $\nu_T(dz, d\zeta)=\int_0^T \nu(t, dz, d\zeta)dt$.
\end{itemize}
In other words, we are in the same situation as in the second step, with $\zeta_0=(0,1)$. For the sake of conciseness, we keep in the remainder of the proof the notations adopted in the second step and assume that $\zeta_0=(0,1)$.\\

\textit{Fourth step: localization around the rational direction $\zeta_0$.}
In this step, we localize the semiclassical defect measure $\mu$ around the rational direction $\zeta_0=(0,1)$.
For $m \in \nn^*$, we define the following set of rational directions
\begin{equation*}\label{eq:rational_directions}
    \Sigma_{\mathbb Q}^m:=\left\{(z,\zeta)\in \mathbb T^2 \times \mathbb S^1; \ \zeta=\frac{(p,q)}{\sqrt{p^2+q^2}}, \ p^2+q^2 \leq m, \ \text{pgcd}(p,q)=1\right\}.
\end{equation*}
Let $\varepsilon>0$ be a small positive real number, to be chosen later. By denoting $\Sigma_{\mathbb R \setminus \mathbb Q}^m =\left(\Sigma_{\mathbb Q}^m\right)^c$, we have $\Sigma_{\mathbb Q} = \bigcup_{m \in \mathbb N^*} \Sigma_{\mathbb Q}^m$ so $\Sigma_{\mathbb R \setminus \mathbb Q}=\Sigma_{\mathbb Q}^c = \bigcap_{m \in \mathbb N^*} (\Sigma_{\mathbb Q}^m)^c = \bigcap_{m \in \nn^*} \Sigma_{\mathbb R \setminus \mathbb Q}^m$, with $\Sigma_{\mathbb R \setminus \mathbb Q}^{m+1} \subset \Sigma_{\mathbb R \setminus \mathbb Q}^m$ for all $m \in \N^*$. Then, the second step shows that $$\lim_{m \to +\infty} \mu_T\left(\Sigma_{\mathbb R \setminus \mathbb Q}^m\right)= \mu_T\left(\Sigma_{\mathbb R \setminus \mathbb Q}\right)=0.$$
In particular, this provides a positive integer $m_{\varepsilon} \in \nn^*$ such that $\mu_T\left(\Sigma_{\mathbb R \setminus \mathbb Q}^{m_{\varepsilon}}\right)<\varepsilon.$
Since $\Sigma_{\mathbb Q}^{m_{\varepsilon}}$ is a discrete set, we can choose $\chi_{\varepsilon} \in \mathcal{C}^{\infty}_c(\rr^2)$ such that 
\begin{equation*}\label{troncature}
    \chi_{\varepsilon}((0,1))=1, \quad \Supp \chi_{\varepsilon} \subset B((0,1); \varepsilon) \quad \text{and} \quad \Sigma^{m_\varepsilon}_{\mathbb Q} \cap \left(\mathbb T^2 \times \Supp \chi_{\varepsilon}\right)= \mathbb T^2\times\{(0,1)\},
\end{equation*}
where $B((0,1); \varepsilon)$ stands for the Euclidean ball of $\rr^2$ centered at $(0,1)$ with radius $\varepsilon$.\\

\textit{Fifth step: a normal form argument.}
This step consists in applying a normal form argument in order to reduce the study to a Schrödinger equation with a one-variable potential. The following proposition is an adaptation of  \cite[Proposition~4.4]{BBZ13} (see also \cite[Corollary~2.6]{BZ12}).
\begin{proposition}\label{prop:normal_form}
There exist three bounded families $(Q_{n,\delta})_{n \in \nn}, (R_{n,\delta})_{n \in \nn}, (W_{n,\delta})_{n \in \nn}$ in $\mathcal L (L^2(\mathbb T^2))$ such that for all $n \in \nn$,
\begin{equation*}
    (I+h_nQ_{n, \delta})\mathcal H_{\theta_n, V_{p_\delta}} \chi_{\varepsilon}(h_n D)
    =\Big(\mathcal H_{\theta_n, \langle V_{p_{\delta}}\rangle_y}(I+h_nQ_{n, \delta})+h_nW_{n, \delta} D_x\Big)\chi_{\varepsilon}(h_nD)+h_nR_{n, \delta},
\end{equation*}
where $\langle V_{p_{\delta}}\rangle_y(x)= \int_{\mathbb T^1} V_{p_\delta}(x,y) dy$ and $D=D_z=(D_x,D_y)=\frac 1i \nabla$. More precisely, there exist a positive constant $C>0$ independent on $\delta$ and a positive constant $C_{\delta}>0$ such that for all $n\in \nn$,
$$\|Q_{n, \delta}\|_{\mathcal L(L^2)} \leq C, \quad \|W_{n, \delta}\|_{\mathcal L(L^2)} \leq C_{\delta} \quad \text{and} \quad \|R_{n, \delta}\|_{\mathcal L(L^2)} \leq C_{\delta}.$$
\end{proposition}
Let us mention that, contrary to \cite[Proposition~4.4]{BBZ13}, our symbols depend on the semiclassical parameters $h_n$. For that reason, we need to carefully estimate each commutators appearing in the proof of Proposition~\ref{prop:normal_form}.
\begin{proof}[Proof of Proposition~\ref{prop:normal_form}]
Let $\sigma \in \mathcal{C}^{\infty}_c(\rr \setminus\{0\})$ be equal to $1$ near to the projection of $\Supp \chi_{\varepsilon}$ onto $\{0\}\times \rr$ (such a function can be chosen independently with respect to $\varepsilon$.).
Let us define 
$$Q_{n, \delta}=\frac i2\tilde{V}_{p_\delta} \frac{\sigma(h_nD_y)}{h_nD_y},$$
with $\tilde{V}_{p_\delta}(x,y)=\int_0^y V_{p_\delta}(x,y{'})-\langle V_{p_\delta}\rangle_y(x)dy'$. The family $(Q_{n,\delta})_{n \in \nn}$ is uniformly bounded (with respect to $\delta$) in $\mathcal L(L^2(\mathbb T^2))$ since $\|\tilde{V}_{p_\delta}\|_{L^{\infty}} \leq 2M$. With this choice, we obtain 
\begin{align*}
    &\ (I+h_nQ_{n,\delta})\mathcal H_{\theta_n, V_{p_\delta}}\chi_{\varepsilon}(h_nD)
     - \mathcal H_{\theta_n,\langle V_{p_\delta}\rangle_y}(I+h_nQ_{n, \delta})\chi_{\varepsilon}(h_nD)
    \\
    &=\big(h_n[Q_{n,\delta}, -\Delta]+ 2ih_n [Q_{n,\delta}, \theta_n\cdot\nabla]+  V_{p_\delta}-\langle V_{p_\delta}\rangle + h_nQ_{n,\delta} V_{p_\delta}-h_n \langle V_{p_\delta}\rangle Q_{n,\delta}\big)\chi_{\varepsilon}(h_nD)\\
    &=\left(-h_n i(D_x \tilde V_{p_\delta}) D_x \frac{\sigma(h_nD_y)}{h_nD_y}\right)\chi_{\varepsilon}(h_nD)+h_n(R_{n,\delta,1}+R_{n,\delta,2}),
\end{align*}
with $R_{n,\delta,1}=(Q_{n,\delta} V_{p_\delta}-\langle V_{p_\delta}\rangle Q_{n,\delta})\chi_{\varepsilon}(h_n D)$ and
    $$R_{n,\delta, 2}=\left(-\frac i2 D_x^2(\tilde{V}_{p_\delta})  \frac{\sigma(h_nD_y)}{h_nD_y}+\frac i2 [ \tilde{V}_{p_\delta}, D_y^2+2\theta_n\cdot\nabla] \frac{\sigma(h_nD_y)}{h_nD_y}+\frac{V_{p_\delta}-\langle V_{p_\delta}\rangle_y}{h_n}\right)\chi_{\varepsilon}(h_nD).$$
We can readily check that the family $(R_{n,\delta,1})_{n \in \nn}$ is bounded in $\mathcal L(L^2(\mathbb T^2))$. Furthermore, by direct computations, we have
\begin{align*} R_{n,\delta,2}
& =\left(\frac i2 (\Delta\tilde{V}_{p_\delta})\frac{\sigma(h_nD_y)}{h_nD_y}-2\theta_n\cdot(\nabla \tilde V_{p_\delta})\frac{\sigma(h_nD_y)}{h_nD_y}-\frac{i}{h_n} (D_y\tilde{V}_{p_\delta})\sigma(h_nD_y)+\frac{V_{p_\delta}-\langle V_{p_\delta}\rangle}{h_n}\right)\chi_{\varepsilon}(h_nD)\\
&=\left(\frac i2 (\Delta\tilde{V}_{p_\delta})\frac{\sigma(h_nD_y)}{h_nD_y}-2\theta_n\cdot(\nabla \tilde V_{p_\delta})\frac{\sigma(h_nD_y)}{h_nD_y}+\frac{V_{p_\delta}-\langle V_{p_\delta}\rangle}{h_n}(1-\sigma(h_nD_y))\right)\chi_{\varepsilon}(h_nD)\\
&=\left(\frac i2 (\Delta\tilde{V}_{p_\delta})\frac{\sigma(h_nD_y)}{h_nD_y}-2\theta_n\cdot(\nabla \tilde V_{p_\delta})\frac{\sigma(h_nD_y)}{h_nD_y}\right) \chi_{\varepsilon}(h_nD),
\end{align*}
since $i(D_y \tilde V_{p_\delta})=V_{p_\delta}-\langle V_{p_\delta} \rangle_y$ and $(1-\sigma) \chi_\varepsilon= 0.$
We obtain that the family $(R_{n,\delta,2})_{n \in \nn}$ is bounded in $\mathcal L(L^2(\mathbb T^2))$. Finally, we obtain that 
\begin{equation*}
    (I+h_nQ_{n,\delta})\mathcal H_{\theta_n, V_{p_\delta}}\chi_{\varepsilon}(h_nD)
     - \mathcal H_{\theta_n,\langle V_{p_\delta}\rangle_y}(I+h_nQ_{n, \delta})\chi_{\varepsilon}(h_nD)= h_n W_{n,\delta} D_x \chi_{\varepsilon}(h_nD) + h_n(R_{n,\delta,1}+R_{n,\delta,2}),
\end{equation*}
with $$W_{n,\delta}= -\frac i2 D_x(\tilde V_{p_\delta}) \frac{\sigma(h_nD_y)}{h_nD_y}.$$ Once again, the family $(W_{n,\delta})_{n \in \nn}$ is bounded in $\mathcal L (L^2(\mathbb T^2))$. This ends the proof of Proposition~\ref{prop:normal_form}.
\end{proof}

\textit{Sixth step: reduction to the one-dimensional case.}
This steps consists in reducing the study to the one-dimensional case, in order to conclude thanks to Corollary~\ref{cor:obs_1D}. We begin by defining for all $t \in \rr$ and $n \in \nn$, $w_{n,\delta}(t)=(I+h_nQ_{n,\delta})\chi_{\varepsilon}(h_nD) v_n(t)$. First of all, we show that there exists a constant $\gamma>0$, independent on $\varepsilon$ and $\delta$, such that $$0< \gamma \leq \liminf_{n \to +\infty} \|w_{n,\delta}\|_{L^2((0,T)\times \mathbb T^2)}.$$
Indeed, if $C>0$ is such that $\sup_{n \in \nn} \|Q_{n,\delta}\|_{\mathcal L(L^2)}\leq C$ then, we have 
\begin{align*}
    \liminf_{n \to +\infty} \|w_{n,\delta}\|_{L^2((0,T)\times \mathbb T^2)} & \geq \liminf_{n \to +\infty} \|\chi_{\varepsilon}(h_nD)v_n\|_{L^2((0,T)\times\mathbb T^2)} - \limsup_{n \to +\infty} h_n C \|v_n\|_{L^2((0,T)\times \mathbb T^2)}\\
    & = \liminf_{n \to +\infty} \|\chi_{\varepsilon}(h_nD)v_n\|_{L^2((0,T)\times\mathbb T^2)}.
\end{align*}
Moreover, thanks to \eqref{defect_measure_2D}, we have 
\begin{align*}
    \liminf_{n \to +\infty} \|\chi_{\varepsilon}(h_nD)v_n\|^2_{L^2((0,T)\times\mathbb T^2)} & =\liminf_{n \to +\infty} \int_{0}^T \langle \chi_{\varepsilon}^2(h_nD) v_n(t), v_n(t)\rangle_{L^2(\mathbb T^2)} dt\\
    & = \int_{\mathbb T^2 \times \rr^2} \chi_{\varepsilon}^2(\xi)
    \mu_T(dx, d\xi) \geq \mu_T\left(\mathbb T^2 \times \{(0,1)\}\right).
\end{align*}
This implies that $ \liminf_{n \to +\infty} \|w_{n,\delta}\|_{L^2((0,T)\times \mathbb T^2)} \geq \gamma>0$, with $\gamma = \mu_T\left(\mathbb T^2 \times \{(0,1)\}\right)^{\frac 12}>0$.

Furthermore, notice that, thanks to Proposition~\ref{prop:normal_form}, $w_{n,\delta}$ solves the following Schrödinger equation
\begin{equation*}
    i\partial_t w_{n,\delta}(t)= (-\Delta+2i\theta_n\cdot \nabla + |\theta_n|^2 + \langle V_{p_\delta}\rangle_y(x))w_{n,\delta}(t)+f_{n,\varepsilon, \delta}(t)+ g_{n,\varepsilon, \delta}(t)
\end{equation*}
with 
\begin{align*}
    f_{n, \varepsilon, \delta}(t)&=h_nR_{n, \delta} v_n(t)+(I+h_nQ_{n,\delta})[\chi_{\varepsilon}(h_nD), V_{p_\delta}]v_n(t),\\
    g_{n, \varepsilon, \delta}(t)&= h_nW_{n,\delta} D_x\chi_{\varepsilon}(h_nD) v_n(t)+(I+h_nQ_{n,\delta})\chi_{\varepsilon}(h_nD)(V_{p_\delta}-V_n)v_n(t).
\end{align*}

Let us show that $(f_{n, \varepsilon, \delta})_{n \in \nn}$ converges to $0$ in $L^2((0,T)\times\mathbb T^2)$. On the first hand, since $(v_n(t))_{n \in \nn}$ is bounded in $L^{\infty}((0,T), L^2(\mathbb T^2))$ and $(R_{n,\delta})_{n \in \nn}$ is bounded in $\mathcal L(L^2(\mathbb T^2))$, the first term on the right hand side goes to $0$ in $L^2((0,T)\times \mathbb T^2)$. On the other hand, to deal with the second term on the right hand side, it is sufficient to show that $([\chi_{\varepsilon}(h_nD), V_{p_\delta}])_{n \in \nn}$ converges to $0$ in $\mathcal L(L^2(\mathbb T^2))$.  This is a consequence of \eqref{eq:expansioncommutatorTd} together with the Calderon-Vaillancourt Theorem \eqref{eq:EstimationCalderonVaillancourtT^d}.

Regarding the sequence $(g_{n, \varepsilon, \delta})_{n \in \nn}$, we have that there exist a constant $C_{\delta}>0$, independent on $\varepsilon$, and a constant $C{'}>0$, independent on $\varepsilon$ and $\delta$, such that 
\begin{equation}
\label{eq:boundsgndelta}
\limsup_{n \to +\infty} \|g_{n, \varepsilon, \delta}\|_{L^2((0,T)\times \mathbb T^2)}^2 \leq C_{\delta} \varepsilon^2 +C{'}\delta^2.
\end{equation}
Indeed, since $\sup_{n \in \nn} \|Q_{n,\delta}\|_{\mathcal L(L^2)}\leq C$ and $(W_{n,\delta})_{n \in \nn}$ is bounded by a constant $C_{\delta}>0$ in $\mathcal L(L^2(\mathbb T^2))$, we have 
\begin{multline*}
\limsup_{n \to+\infty} \|g_{n, \varepsilon, \delta}\|^2_{L^2((0,T)\times \mathbb T^2)}\\
\leq 2C_{\delta}\limsup_{n \to +\infty} \int_{0}^T\|h_n D_x\chi_{\varepsilon}(h_nD) v_n(t)\|^2_{L^2(\mathbb T^2)} dt+2 \limsup_{n \to +\infty} \left\|(V_{p_\delta}-V_n)v_n\right\|^2_{L^2((0,T)\times\T^2)}.
\end{multline*}
On the first hand, by using the fact that $\Supp \chi_{\varepsilon} \subset \{(\xi_1, \xi_2)\in \rr^2; \ |\xi_1| \leq \varepsilon\}$, it follows that 
$$\limsup_{n \to +\infty} \int_{0}^T\|h_n D_x\chi_{\varepsilon}(h_nD) v_n(t)\|^2_{L^2(\mathbb T^2)} dt \leq \varepsilon^2 \|v_n\|^2_{L^2((0,T)\times\T^2)}=\varepsilon^2 T.$$
On the other hand, since $\limsup_{n\to+\infty}\|V_{p_{\delta}}-V_n\|_{L^4(\T^2)} \leq \delta$ by \eqref{eq:Vn_cauchy_sequence}, it follows from the Strichartz type estimates given by Proposition~\ref{prop:StrichartzEstimate} that
$$\limsup_{n \to +\infty} \left\|(V_{p_\delta}-V_n)v_n\right\|^2_{L^2((0,T)\times\T^2)} \leq \| V_{p_\delta}-V_n\|^2_{L^4(\T^2)}\|v_n\|^2_{L^4(\T^2,L^2(0,T))} \leq C{'}\delta^2,$$
where $C{'}>0$ is a new positive constant provided by Proposition~\ref{prop:StrichartzEstimate}.  We therefore obtain \eqref{eq:boundsgndelta}.

By now, we use the Fourier series Theory in the $y$-variable. We can write
\begin{equation*}
    w_{n,\delta}(t,z)=\sum_{k \in \mathbb Z} w_{n,\delta,k}(t,x) e^{iky}, \quad f_{n,\varepsilon, \delta}(t,z)=\sum_{k \in \mathbb Z} f_{n,\varepsilon, \delta,k}(t,x) e^{iky}, \quad g_{n,\varepsilon, \delta}(t,z)=\sum_{k \in \mathbb Z} g_{n,\varepsilon, \delta,k}(t,x) e^{iky},
\end{equation*}
for $z=(x,y)\in \mathbb T^2$.
Since $\langle V_{p_{\delta}}\rangle_y$ does not depend on $y$, we obtain that for all integers $k$ and $n$, $w_{n,\delta,k}$ solves the following one-dimensional Schrödinger equation posed on $\mathbb T^1$:
\begin{equation*}
    i\partial_t w_{n,\delta, k}(t)=(-\partial^2_x+2i\theta_{1,n}\cdot \partial_x + |\theta_n|^2 + \langle V_{p_\delta}\rangle(x)+k^2+2\theta_{2,n}k)w_{n,\delta,k}(t)+f_{n,\varepsilon, \delta,k}(t)+ g_{n,\varepsilon, \delta, k}(t),
\end{equation*}
where $\theta_n=(\theta_{1,n}, \theta_{2,n})$.

We can now apply the one-dimensional observability estimates, given by Corollary~\ref{cor:obs_1D}, to the solutions $w_{n,\delta,k}$, with potential $\langle V_{p_\delta}\rangle_y$ and observable $\langle b\rangle_y$. By Parseval's Theorem, we obtain
\begin{align*}
      \int_{0}^T\|w_{n,\delta}(t)\|^2_{L^2(\mathbb T^2)}dt&= \sum_{k \in \mathbb Z} \int_0^T\|w_{n,\delta,k}(t)\|^2_{L^2(\mathbb T^1)}dt\\
    & \leq \tilde C \sum_{k \in \mathbb Z} \int_0^T\int_{\mathbb T^1} \langle b\rangle_y(x) |w_{n,\delta,k}(t,x)|^2 dxdt + \tilde C \sum_{k \in \mathbb Z} \|f_{n,\varepsilon, \delta,k}+g_{n,\varepsilon, \delta,k}\|^2_{L^2((0,T)\times \mathbb T^1)},
\end{align*}
where $\tilde C>0$ is a positive constant, independent on $\delta>0$, provided by Corollary~\ref{cor:obs_1D}. It follows that 
\begin{align}\label{eq:gamma}
    0<\gamma^2 \leq\liminf_{n\to+\infty} \int_{0}^T\|w_{n,\delta}(t)\|^2_{L^2(\mathbb T^2)}dt &\leq  \limsup_{n\to+\infty} \int_{0}^T\|w_{n,\delta}(t)\|^2_{L^2(\mathbb T^2)}dt \\ \nonumber
    & \leq \tilde C \limsup_{n\to+\infty}\int_{0}^T \int_{\mathbb T^2} \langle b \rangle_y(x) |w_{n,\delta}(t,z)|^2 dzdt+ \tilde CC_{\delta}\varepsilon^2+\tilde C C{'}\delta^2.
\end{align}
It remains to estimate the first term of the right hand side. Let $(b_p)_{p\in\nn} \subset \mathcal{C}^{\infty}(\mathbb T^2, \rr_+)$ be a sequence converging to $b$ in $L^2(\mathbb T^2)$ and satisfying 
$\|b_p\|_{L^{\infty}}\leq \|b\|_{L^{\infty}},$ for all $p \in \nn$.
We have, for all $p,n \in \nn$,
\begin{multline}\label{eq:b_estimate}
\int_{0}^T \int_{\mathbb T^2} \langle b \rangle_y(x) |w_{n,\delta}(t,z)|^2 dzdt\\
=\int_{0}^T \int_{\mathbb T^2} \langle b-b_p \rangle_y(x) |w_{n,\delta}(t,z)|^2 dzdt+\int_{0}^T \int_{\mathbb T^2} \langle b_p \rangle_y(x) |w_{n,\delta}(t,z)|^2 dzdt.
\end{multline}
On the first hand, thanks to the Strichartz type estimates given by Proposition~\ref{prop:StrichartzEstimate}, there exists an other constant $C{'}>0$ such that for all $p, n \in \nn$,
\begin{multline*}
    \int_{0}^T \int_{\mathbb T^2} \langle b-b_p \rangle_y(x) |w_{n,\delta}(t,z)|^2 dzdt  \leq \left\|\langle b-b_p\rangle_y\right\|_{L^2(\mathbb T^1)} \|w_{n,\delta}\|^2_{L^4(\mathbb T^2, L^2(0,T))}\\
     \leq C{'} \left\|\langle b-b_p\rangle_y\right\|_{L^2(\mathbb T^1)} \left(\|w_{n,\delta}(0,\cdot)\|^2_{L^2(\mathbb T^2)}+ \|f_{n,\varepsilon,\delta}+g_{n,\varepsilon, \delta}\|^2_{L^2((0,T)\times\mathbb T^2)}\right).
\end{multline*}
Thus, we have for all $p \in \nn$,
\begin{equation}\label{eq:b_estimate2}
\limsup_{n \to+\infty} \int_{0}^T \int_{\mathbb T^2} \langle b-b_p \rangle_y(x) |w_{n,\delta}(t,z)|^2 dzdt \leq C'_{\varepsilon, \delta} \left\| b-b_p\right\|_{L^2(\mathbb T^2)},
\end{equation}
for a new positive constant $C'_{\varepsilon, \delta}>0$.
On the other hand, thanks to \eqref{defect_measure_2D}, we have for all $p \in \nn$,
\begin{multline*}
    \limsup_{n\to+\infty} \int_{0}^T \int_{\mathbb T^2} \langle b_p \rangle_y(x) |w_{n,\delta}(t,z)|^2 dzdt = \int_{\mathbb T^2\times \rr^2} \langle b_p\rangle_y(x) |\chi_{\varepsilon}(\zeta)|^2 \mu_T(dz, d\zeta)\\
     = \int_{\Sigma_{\mathbb Q}^{m_\varepsilon}}  \langle b_p\rangle_y(x) |\chi_{\varepsilon}(\zeta)|^2 \mu_T(dz, d\zeta)+ \int_{\Sigma_{\rr \setminus\mathbb Q}^{m_\varepsilon}}  \langle b_p\rangle_y(x) |\chi_{\varepsilon}(\zeta)|^2 \mu_T(dz, d\zeta).
\end{multline*}
Let us recall that $\Sigma_{\mathbb Q}^{m_{\varepsilon}} \cap (\mathbb T^2 \times \Supp \chi_{\varepsilon})= \mathbb T^2 \times \{(0,1)\}$ and $\mu_T(\Sigma_{\rr \setminus\mathbb Q}^{m_\varepsilon})<\varepsilon$.
We deduce that 
\begin{equation}\label{eq:b_estimate3}
    \limsup_{n\to+\infty} \int_{0}^T \int_{\mathbb T^2} \langle b_p \rangle(x) |w_{n,\delta}(t,z)|^2 dzdt \leq \int_{\mathbb T^2\times\{(0,1)\}}  \langle b_p\rangle_y(x) \mu_T(dz, d\zeta)+\|b_p\|_{L^{\infty}(\mathbb T^2)} \varepsilon.
\end{equation}
Moreover, since $\mathbb T^2\times \{(0,1)\}$ and $\mu_T(dz, d\zeta)$ are invariant under the flow $(z, \xi) \mapsto (z+s\zeta, \zeta)$ for $s \in \rr$, we have 
\begin{equation}\label{eq:b_estimate4}
\int_{\mathbb T^2\times\{(0,1)\}}  \langle b_p\rangle_y(x) \mu_T(dz, d\zeta)= \int_{\mathbb T^2\times\{(0,1)\}} b_p(z) d\mu_T(dz, d\zeta) \leq \int_{\mathbb T^2} b_p(z) g_T(z) dz,
\end{equation}
where $g_T \in L^2(\mathbb T^2)$ is introduced in the first step. Finally, by gathering \eqref{eq:gamma}, \eqref{eq:b_estimate}, \eqref{eq:b_estimate2}, \eqref{eq:b_estimate3} and \eqref{eq:b_estimate4}, we obtain 
\begin{equation*}
    0 < \gamma^2 \leq \tilde CC'_{\varepsilon, \delta} \|b-b_p\|_{L^2(\mathbb T^2)}+\tilde C\int_{\mathbb T^2} b_p(z) g_T(z) dz+ \tilde C\|b_p\|_{L^{\infty}(\mathbb T^2)}\varepsilon+ \tilde CC_{\delta}\varepsilon^2+\tilde C C{'}\delta^2.
\end{equation*}
Since $\sup_{p\in \nn}\|b_p\|_{L^{\infty}} \leq \|b\|_{L^{\infty}}$, $b_p \underset{p\to+\infty}{\longrightarrow} b$ in $L^2(\mathbb T^2)$ and $\int_{\mathbb T^2} b(z) g_T(z)dz=0$, by passing to the limit when $p\to+\infty$, we first have
$$0<\gamma^2 \leq \tilde C\|b\|_{L^{\infty}} \varepsilon+\tilde C C_{\delta} \varepsilon^2+\tilde C C{'} \delta^2.$$ Then, letting $\varepsilon$ go to $0$ leads to 
$$0< \frac{\gamma^2}{2} \leq \tilde{C} C{'}\delta^2.$$
Since $\delta>0$ can be chosen arbitrary small, this last inequality provides a contradiction. 
This ends the proof of Proposition~\ref{obs_tore_smooth_potential}.

\subsubsection{From smooth to rough potentials}
In this section, we end the proof of Theorem~\ref{obs_tore_thm}. Let $T>0$ and $\mathcal K \subset L^{\infty}(\T^2)$ be a compact subset. Since $\mathcal K$ is compact, there exists $M>0$ such that $\mathcal K \subset B_{L^{\infty}}(0,M)$. Let $(\rho_n)_{n \in \nn} \subset \mathcal C^{\infty}(\T^2)$ be an approximation of unity. From $\mathcal K$, we define the following set of smooth potentials:
\begin{equation*}
    \mathcal K':=\left\{\rho_n *V; \ V \in \mathcal K, n \in \nn\right\}.
\end{equation*}
It can be readily checked that $\mathcal K' \subset \mathcal C^{\infty}(\T^2)\cap B_{L^{\infty}}(0,M)$ is relatively compact in $L^4(\T^2)$. Thanks to Proposition~\ref{obs_tore_smooth_potential}, there exists a positive constant $C=C(T,M, \mathcal K')>0$ such that for all $n \in \nn$, $\theta \in [0,1]^2$, $V\in \mathcal K'$ and $u_0 \in L^2(\T^2)$,
\begin{equation}\label{eq:obs_smooth_potentialrhon}
    \|u_0\|^2_{L^2(\T^2)} \leq C \int_0^T \int_{\T^2} b(z) \left|e^{-it(-\Delta+2i\theta\cdot\nabla + |\theta|^2 +\rho_n*V)}u_0(z)\right|^2 dz dt+C \|u_0\|^2_{H^{-2}(\T^2)}.
\end{equation}
Let $V \in \mathcal K$. Since $\rho_n*V \underset{n\to+\infty}{\longrightarrow} V$ in $L^4(\T^2)$, we then have $\rho_n*V \rightharpoonup^\star V$ in $L^{\infty}(\T^2)$, so one can use Proposition \ref{prop:stabilityparameters} to pass to the limit as $n \to +\infty$ in \eqref{eq:obs_smooth_potentialrhon} to finally get for all $\theta \in [0,1]^2$ and $u_0\in L^2(\T^2)$,
$$\|u_0\|^2_{L^2(\T^2)} \leq C \int_0^T \int_{\T^2} b(z) \left|e^{-it(-\Delta+2i\theta\cdot\nabla+ |\theta|^2 +V)}u_0(z)\right|^2 dz dt+C\|u_0\|^2_{H^{-2}(\T^2)}.$$
Since $\mathcal K$ is compact in $L^{\infty}(\T^2)$, we can apply Proposition~\ref{prop:from_weakobs_to_strongobs} to conclude the proof of Theorem~\ref{obs_tore_thm}.

\appendix
\section{Appendix}\label{section:appendix}

\subsection{Semiclassical quantization on Euclidean spaces}
This section is devoted to present few facts about semiclassical analysis on Euclidean spaces. We follow the presentation of \cite[Chapter 4]{Zwo12}.

In this section, $h>0$ denotes a positive parameter and $a \in \mathcal{S}(\R^{2d})$, $a = a(x, \xi)$ is called a symbol. 
The Weyl quantization of $a$ is the operator $\ops ha$ acting on $u \in \mathcal{S}(\R^{d})$ by the formula
\begin{equation*}
    \ops ha u(x) = \frac{1}{2 \pi h^d} \int_{\R^d} \int_{\R^d} e^{\frac{i}{h} \langle x-y, \xi \rangle} a \left(\frac{x+y}{2}, \xi\right) u(y) dy d \xi.
\end{equation*}

We first present a result that tells us how the Weyl quantization acts on the Schwartz space $\mathcal{S}(\R^d)$, on tempered distributions $\mathcal{S}'(\R^d)$ and on $L^2(\R^d)$.

\begin{theorem}
\label{th:continuitypseudoRd}
We have the following continuity results.
\begin{enumerate}
    \item \cite[Theorem 4.1]{Zwo12}. Assume $a \in \mathcal{S}(\R^{2d})$. Then, $\ops ha$ can be extended as an operator mapping $\mathcal{S}'(\R^d)$ to $\mathcal{S}(\R^d)$.
    \item \cite[Theorem 4.2]{Zwo12}. Assume $a \in \mathcal{S}'(\R^{2d})$. Then, $\ops ha$ can be extended as an operator mapping $\mathcal{S}(\R^d)$ to $\mathcal{S}'(\R^d)$.
    \item \cite[Theorem 5.1]{Zwo12}. Assume $a \in \mathcal C_b^{\infty}(\R^{2d}) := \{a \in \mathcal C^{\infty}(\R^{2d}) \ ;\ \norme{\partial^{\alpha} a}_{L^{\infty}(\R^{2d})} \leq C_{\alpha}\ \forall \alpha \in \N^{2d}\}$. Then, $\ops ha$ can be extended as an operator mapping $L^2(\R^d)$ to $L^2(\R^d)$. Moreover, we have
    \begin{equation}
        \label{eq:EstimationCalderonVaillancourt}
            \|\ops ha\|_{\mathcal L (L^2(\mathbb R^d))} \leq C \norme{a}_{L^{\infty}(\R^{2d})} + O(h^{1/2}).
    \end{equation}
    \end{enumerate}
\end{theorem}

The third point of Theorem \ref{th:continuitypseudoRd} is usually called Calderon-Vaillancourt Theorem. Note that \cite[Theorem 4.23]{Zwo12} telling us that 
 \begin{equation}
        \label{eq:EstimationCalderonVaillancourtWeak}
            \|\ops ha\|_{\mathcal L (L^2(\mathbb R^d))} \leq C \sum_{|\alpha| \leq M d} \norme{\partial^{\alpha} a}_{L^{\infty}(\R^{2d})},
\end{equation}
would not be sufficient for our purpose. Actually, the useful bound \eqref{eq:EstimationCalderonVaillancourt} comes from \eqref{eq:EstimationCalderonVaillancourtWeak} and a tricky scaling argument. It is also worth mentioning that one can derive a better bound than \eqref{eq:EstimationCalderonVaillancourt}. Indeed, \cite[Theorem 13.13]{Zwo12} states that 
 \begin{equation*}
        \label{eq:EstimationCalderonVaillancourtStrong}
            \|\ops ha\|_{\mathcal L (L^2(\mathbb R^d))} =  \norme{a}_{L^{\infty}(\R^{2d})} + O(h).
    \end{equation*}
    
We have this straightforward result, that enables to compute explicitly the operator $\ops ha$ for particular symbols $a$.
\begin{lemma}
We have
\begin{enumerate}
    \item \cite[Equation (4.1.6)]{Zwo12}. If $a(x, \xi) = \xi^{\alpha}$, then $\ops ha u (x) = h^{\alpha} D^{\alpha} u (x)$, where $D^{\alpha} = (1/i) \partial^{\alpha}$,
    \item \cite[Theorem 4.3]{Zwo12}. If $a(x, \xi) = a(x) \in \mathcal{S}'(\R^d)$ then $\ops ha u (x) = a(x) u(x)$.
\end{enumerate}
\end{lemma}

We now state the well-known Gårding inequality that roughly indicates that if the symbol is non-negative then the associated operator is almost non-negative.

\begin{theorem}
\cite[Theorem 4.32]{Zwo12}. Assume $a \in \mathcal C_b^{\infty}(\R^{2d})$ and $a \geq 0$ in $\R^{2d}$. Then there exist constants $C \geq 0$ and $h_0 >0$ such that 
\begin{equation*}
    \label{eq:inequalitygardingRd}
    \langle \ops ha u, u \rangle \geq - C h \norme{u}_{L^2(\R^d)}^2\qquad \forall 0 < h < h_0,\ \forall u \in L^2(\R^d).
\end{equation*}
\end{theorem}

Roughly speaking, the next result tells us first that the composition of two pseudodifferential operators is also a pseudodifferential operator and we can give an exact formula. Secondly, we present the asymptotic expansion of such a symbol with respect to $h$.

Given $z = (x,\xi)$, $w = (y,\eta)$ in $\R^{2d}$, define their symplectic product
$$\sigma(z,w) = \langle \xi, y \rangle - \langle x, \eta \rangle.$$

\begin{theorem}\label{th:semiclassicalcomposition}
\cite[Theorem 4.11 and 4.12]{Zwo12}. 

Suppose that $a, b \in \mathcal{S}(\R^{2d})$. Then, there exists $c \in \mathcal{S}(\R^{2d})$ such that
\begin{equation*}
    \ops ha \ops hb = \ops hc,
\end{equation*}
where
\begin{equation*}
c(x,\xi) = e^{i \frac{h}{2} \sigma(D_x, D_{\xi}, D_{y}, D_{\eta})} (a(x, \xi) b(y,\eta))_{|(y,\eta)=(x,\xi)}.
\end{equation*}
Moreover, we have
\begin{equation*}
\label{eq:expansionproductRd}
    c= ab + \frac{h}{2i} \{a,b\} + \mathcal{O}_{\mathcal S}(h^2),
\end{equation*}
and
\begin{equation}
\label{eq:expansioncommutatorRd}
    [ \ops ha,  \ops hb] =  \frac hi \ops h{ \{a, b\}} + \mathcal{O}_{\mathcal S}(h^3).
\end{equation}
If we assume further that $\mathrm{supp}(a) \cap \mathrm{supp}(b) = \emptyset$, then 
\begin{equation}
\label{eq:suppdisjointssemiclassical}
    c = \mathcal{O}_{\mathcal S}(h^{\infty}).
\end{equation}
\end{theorem}
In the previous result, the notation $\varphi = \mathcal{O}_{\mathcal{S}}(h^k)$ means that for all multiindices $\alpha,\beta,\gamma,\delta$, 
$$ \sup_{(x, \xi) \in \R^{2d}} |(1+|x|^2)^{\alpha} (1+|\xi|^2)^{\beta} \partial_{x}^{\gamma} \partial_{\xi}^{\delta} a(x,\xi)| \leq C_{\alpha,\beta,\gamma, \delta} h^k.$$

\subsection{Semiclassical quantization on the torus}

We need to extend the semiclassical quantization to the torus $\T^d$. First, one can identify the torus $\T^d$ with the fundamental domain 
$$ \T^d \approx \{x = (x_1, \dots, x_n)\ ;\ 0 \leq x_i\leq 1,\, \forall 1 \leq i \leq d\}.$$
We likewise identify functions in $\T^d$ with periodic functions on $\R^d$
$$ u(x + k) = u(x)\qquad \forall k \in \Z^d.$$
Symbols $a$ on $T^* \T^d = \T^d \times \R^d$ are similarly identified with symbols $a \in \R^{2d}$ that are periodic with respect to the variable $x \in \R^d$,
$$ a(x+k, \xi) = a(x,\xi)\qquad \forall k \in \Z^d.$$
Operators obtained by quantizing such symbols satisfy
$$
(\mathrm{op}_h (a) u)(x+k) = (\mathrm{op}_h (a) u(\cdot + k))(x).
$$

From the previous results, established in the Euclidean case $\R^{2d}$, one can deduce the analogue in the torus.

\begin{theorem}\label{thm:sc_results_torus}
Suppose that $a, b \in \mathcal{C}_b^{\infty}(T^* \T^d)$.
\begin{enumerate}
\item The operator $\ops ha$ can be extended as an operator mapping $L^2(\T^d)$ to $L^2(\T^d)$. Moreover, we have
    \begin{equation}
        \label{eq:EstimationCalderonVaillancourtT^d}
            \|\ops ha\|_{\mathcal L (L^2(\mathbb T^d))} \leq C \norme{a}_{L^{\infty}(T^* \T^d)} + O(h^{1/2}).
    \end{equation}
\item If $a \geq 0$, then there exist constants $C \geq 0$ and $h_0 >0$ such that 
\begin{equation*}
    \label{eq:inequalitygardingTd}
    \langle \ops ha u, u \rangle \geq - C h \norme{u}_{L^2(\T^d)}^2\qquad \forall 0 < h < h_0,\ \forall u \in L^2(\T^d).
\end{equation*}
\item There exists $c \in \mathcal{S}(\R^{2d})$ such that
\begin{equation*}
    \ops ha \ops hb = \ops hc,
\end{equation*}
where
\begin{equation}
\label{eq:exactformulacompositionT}
c(x,\xi) = e^{i \frac{h}{2} \sigma(D_x, D_{\xi}, D_{y}, D_{\eta})} (a(x, \xi) b(y,\eta))_{|(y,\eta)=(x,\xi)}.
\end{equation}
Moreover, we have
\begin{equation*}
\label{eq:expansionproductTd}
    c= ab + \frac{h}{2i} \{a,b\} + \mathcal{O}_{\mathcal S}(h^2),
\end{equation*}
and
\begin{equation}
\label{eq:expansioncommutatorTd}
    [ \ops ha,  \ops hb] =  \frac hi \ops h{ \{a, b\}} + \mathcal{O}_{\mathcal S}(h^3).
\end{equation}
\item Given a polynomial function $P$ of degree $2$, we have
\begin{equation}\label{commutators_laplacian}
    [\ops ha, \ops h{P(\xi)}] = \frac hi \ops h{\nabla_{\xi}P(\xi) \cdot \nabla_x a}
\end{equation}
\end{enumerate} 
\end{theorem}
\begin{proof}
We only prove the fourth point. From the exact formula \eqref{eq:exactformulacompositionT}, the terms of order bigger than $h$ in the asymptotic expansion only involve derivatives of order more than three so they all vanish. Therefore, we have $[\ops ha, \ops h{P(\xi)}] = \frac hi \ops h{ \{a, P(\xi)\}} =  \frac hi \ops h {\nabla_{\xi}P(\xi) \cdot \nabla_x a},$ which concludes the proof.
\end{proof}

\subsection{Semiclassical measures for Schrödinger equations}
This section is devoted to recall useful results concerning semiclassical measures for Schrödinger equations. 

Let $(h_n)_{n \in \nn} \subset (0,1]$ be a sequence of real positive numbers tending to $0$. Let $(\theta_n)_{n \in \nn} \subset [0,1]^d$ and $(V_n)_{n \in \nn} \subset \mathcal{C}^{\infty}(\mathbb T^d)$ be a sequence of smooth potentials satisfying $$\sup_{n \in \nn} \|V_n\|_{L^{\infty}(\mathbb T^d)} <+\infty.$$ To simplify, we denote by $\mathcal H_{n} (= \mathcal{H}_{\theta_n,V_n})$ the Schrödinger operator given by
\begin{equation*}\label{schrod_operator}
\mathcal H_{n}=-\Delta+2i\theta_n\cdot\nabla + |\theta_n|^2 + V_n,
\end{equation*}
with $n \in \nn$.
\begin{definition}\label{h-oscillating}
A bounded family $(u_n)_{n \in \nn} \subset L^2(\mathbb T^d)$ is said to be \textit{$(h_n)$-oscillating} if and only if 
$$\lim_{R \to +\infty} \limsup_{n \to +\infty} \| \un_{(R,+\infty)}(h_n^2 \mathcal H_{n}) u_n \|_{L^2(\mathbb T^d)} =0.$$ 
\end{definition}
The following lemma shows that a $(h_n)$-oscillating family is actually also $(h_n)$-oscillating with respect to the free Schrödinger operator $-\Delta$.
\begin{lemma}\label{lemma:h_oscillating}
Let $(u_n)_{n \in \nn} \in L^2(\mathbb T^d)$ be a bounded family. This family is $(h_n)$-oscillating (with respect to $\mathcal H_{n}$) if and only if 
\begin{equation}\label{h-oscillating2}
\lim_{R\to+\infty} \limsup_{n \to +\infty} \| \un_{(R, +\infty)}(-h_n^2\Delta) u_n\|_{L^2(\mathbb T^d)}=0.
\end{equation}
\end{lemma}
\begin{proof}
Let us prove that if $(u_n)_{n \in \nn}$ is $(h_n)$-oscillating (with respect to $\mathcal H_{n}$) then \eqref{h-oscillating2} holds. The converse can be proved in the same manner. By using the facts that 
$$\un_{(R,+\infty)}(-h_n^2\Delta)=\un_{(R,+\infty)}(-h_n^2\Delta)\un_{(\sqrt R,+\infty)}(h_n^2 \mathcal H_{n})+\un_{(R,+\infty)}(-h_n^2\Delta)\un_{(-\infty, \sqrt R)}(h_n^2 \mathcal H_{n})$$
and for all $R> 4$ and $v \in L^2(\mathbb T^d)$,
$$\|\un_{(R,+\infty)}(-h_n^2\Delta) v\|^2_{L^2(\mathbb T^d)} \leq \frac1{R-2\sqrt R}\left\langle \un_{(R,+\infty)}(-h_n^2\Delta) h_n^2 (-\Delta +2i\theta_n \cdot \nabla) v,v\right\rangle_{L^2(\mathbb T^d)},$$
thanks to functional calculus, we deduce that for all $R>4$ and $n \in \nn$, taking $v = \un_{(-\infty, \sqrt R)}(h_n^2 \mathcal H_{n}) u_n$,
\begin{multline*}
     \|\un_{(R,+\infty)}(-h_n^2\Delta) u_n\|^2_{L^2(\mathbb T^d)} 
    \leq  \  2\|\un_{(\sqrt R,+\infty)}(h_n^2 \mathcal H_{n}) u_n\|^2_{L^2(\mathbb T^d)}\\
    +\frac2{R-2\sqrt R}\left\langle h_n^2 (-\Delta +2i\theta_n \cdot \nabla) \un_{(-\infty, \sqrt R)}(h_n^2 \mathcal H_{n}) u_n, \un_{(-\infty, \sqrt R)}(h_n^2 \mathcal H_{n}) u_n\right\rangle_{L^2(\mathbb T^d)}.
\end{multline*}
Moreover, we have that for all $n \in \nn$,
\begin{multline*}
    \left\langle h_n^2 (-\Delta +2i\theta_n \cdot \nabla) \un_{(-\infty, \sqrt R)}(h_n^2 \mathcal H_{n}) u_n, \un_{(-\infty, \sqrt R)}(h_n^2 \mathcal H_{n}) u_n\right\rangle_{L^2(\mathbb T^d)}\\
    \leq \sqrt R \|u_n\|_{L^2(\mathbb T^d)}^2 + h_n^2 \Big(d +  \|V_n\|_{L^{\infty}(\mathbb T^d)}\Big)\|u_n\|_{L^2(\mathbb T^d)}^2 .
\end{multline*}
We therefore deduce from the previous lines, together with the facts that $(V_n)_{n \in \nn}$ is $L^{\infty}$-bounded and $(u_n)_{n \in \nn}$ is $L^2$-bounded, denoting $M = \sup_{n \in \nn}\left\| u_n\right\|_{L^2(\mathbb T^d)}^2$,
\begin{equation*}
         \limsup_{n \to +\infty}\|\un_{(R,+\infty)}(-h_n^2\Delta) u_n\|_{L^2(\mathbb T^d)}^2
    \leq  \ \limsup_{n \to +\infty}\|\un_{(\sqrt R,+\infty)}(h_n^2 \mathcal H_{n}) u_n\|_{L^2(\mathbb T^d)}^2+\frac{\sqrt R M}{R-2\pi\sqrt R}.
\end{equation*}
We can now apply the $(h_n)$-oscillating assumption to conclude that 
\begin{equation*}
    \limsup_{n \to +\infty} \|\un_{(R,+\infty)}(-h_n^2\Delta) u_n\|_{L^2(\mathbb T^d)}^2 \leq \limsup_{n \to +\infty} \|\un_{(\sqrt R,+\infty)}(h_n^2 \mathcal H_{n}) u_n\|_{L^2(\mathbb T^d)}^2 + \frac{\sqrt R M}{R-2 \sqrt R} \underset{R\to+\infty}{\longrightarrow} 0.
\end{equation*}
This concludes the proof of Lemma~\ref{lemma:h_oscillating}.
\end{proof}

The following proposition is the main result of this section. It provides a generalisation of results established in \cite{Mac09} by Macià. More precisely, \cite[Theorem~1 \& 2]{Mac09} studied semiclassical defect measures for solutions to Schrödinger equations associated to an operator $\mathcal{H}=-\Delta+V$. The following result extends \cite[Theorem~1 \& 2]{Mac09} to the operators $\mathcal{H}_n=-\Delta+2i\theta_n \cdot \nabla + |\theta_n|^2+   V_n,$ which are allowed to vary according to $n \in \nn$. The proof follows the very same lines as the ones given in \cite{Mac09}. For the convenience of the reader, the proof is entirely recalled.
\begin{proposition}\label{semiclassical_measure_prop}
Let $(u_n)_{n \in \nn}$ be a bounded family in $L^2(\mathbb T^d)$. If $(u_n)_{n \in \nn}$ is $(h_n)$-oscillating then there exists a subsequence, still denoted $(u_n)_{n \in \nn}$ and a finite measure $\mu \in L^{\infty}(\rr, \mathcal M_+(T^*\mathbb T^d))$ satisfying:\\
\indent \textit{(i)} For every $\varphi \in L^1(\rr)$ and every $a \in \mathcal C^{\infty}_c(T^*\mathbb T^d)$,
\begin{equation*}
    \lim_{n \to +\infty} \int_{\rr} \varphi(t) \langle \ops {h_n}a e^{-i t\mathcal H_{n}} u_n, e^{-it \mathcal H_{n}} u_n\rangle_{L^2(\mathbb T^d)} dt= \int_{\rr \times T^*\mathbb T^d} \varphi(t) a(x,\xi) \mu(t, dx, d\xi)dt.
\end{equation*}
\indent \textit{(ii)} For every $\varphi \in L^1(\rr)$ and every $a \in \mathcal C^{\infty}_c(\mathbb T^d)$,
\begin{equation*}
    \lim_{n\to +\infty} \int_{\rr} \varphi(t) a(x) |e^{-it \mathcal H_{n}} u_n(x)|^2 dt= \int_{\rr \times T^*\mathbb T^d} \varphi(t) a(x) \mu(t, dx, d\xi)dt.
\end{equation*}
\indent \textit{(iii)} The measure $\mu$ is invariant under the geodesic flow: for every $\varphi \in L^1(\rr)$, $a \in \mathcal C^{\infty}_c(T^*\mathbb T^d)$,
\begin{equation*}
    \forall s \in \rr, \quad \int_{\rr}\int_{T^*\mathbb T^d}\varphi(t) a(x+s\xi, \xi) \mu(dt, dx, d\xi)=\int_{\rr} \int_{T^*\mathbb T^d} \varphi(t)  a(x, \xi) \mu(dt, dx, d\xi).
\end{equation*}
\end{proposition}
\begin{proof}
Let $\psi_n(t,x)= e^{-i t\mathcal H_{n}} u_n(x)$. \medskip\\
\indent \textit{First step: subsequence for time-space Wigner distribution.}\\
Let us introduce the time-space Wigner distribution $W_n \in \mathcal D'(\rr \times \mathbb T^d \times \rr \times \rr^d)$, given by 
\begin{equation*}
    \forall b \in \mathcal C^{\infty}_c(\rr\times \mathbb T^d \times \rr \times \rr^d), \quad \langle W_n, b \rangle = \langle \ops {h_n}{b_n} \psi_n, \psi_n \rangle_{L^2(\rr\times \mathbb T^d)},
\end{equation*}
where $b_n(t,x,\tau, \xi):=b(t, x, h_n \tau, \xi)$.

Thanks to the Caldéron-Vaillancourt Theorem \eqref{eq:EstimationCalderonVaillancourt}, for every $b \in \mathcal D'(\rr \times \mathbb T^d \times \rr \times \rr^d)$, the family $\langle W_n, b \rangle$ is bounded. Since the space $(\mathcal D'(\rr \times \mathbb T^d \times \rr \times \rr^d), \|\cdot\|_{\infty})$ is separable, it follows that, up to a subsequence, $(W_n)_{n \in \nn}$ converges to $\tilde{\mu} \in \mathcal D'(\rr \times \mathbb T^d \times \rr \times \rr^d)$. By using anew the Caldéron-Vaillancourt Theorem, we have for every $b \in \mathcal C^{\infty}_c(\rr \times \mathbb T^d \times \rr \times \rr^d)$,
$$\|\langle W_n, b \rangle \| \leq C \norme{b}_{\infty} + O(h_n^{1/2}),$$
and by passing to the limit as $n \to +\infty$, we obtain that
$$\forall b \in \mathcal C^{\infty}_c( \rr\times \T^d\times \rr \times \rr^d), \quad \langle \tilde{\mu}, b \rangle_{\mathcal D', \mathcal D} \leq C \|b \|_{\infty}.$$
The above estimate therefore ensures that $\tilde{\mu}$ is a Radon measure on $\rr\times\mathbb T^d \times \rr \times \rr^d$. Moreover, by the Gårding inequality, for a non-negative test function $b \in \mathcal C^{\infty}_c(\rr \times \mathbb T^d \times \rr \times \rr^d)$,
$$\langle W_n, b \rangle \geq - C h_n \norme{u_n}_{L^2(\T^d)}.$$
Then, as $n \to +\infty$,
$$\forall b \in C^{\infty}_c(\rr \times \mathbb T^d \times \rr \times \rr^d), \quad \langle \tilde{\mu}, b \rangle \geq 0,$$
and $\tilde{\mu}$ is therefore a positive Radon measure.\\

\indent \textit{Second step: semiclassical defect measure for $\psi_n$.}\\
By now, we define the following positive Radon measure on $\rr \times \mathbb T^d \times \rr^d$
$$\mu(dt, dx, d\xi):= \int_{\rr_{\tau}} \tilde{\mu}(dt, dx, d\tau, d\xi).$$
We first establish that $\mu$ is well-defined, $\mu \in L^{\infty}(\rr, \mathcal M_+(\mathbb T^d \times \rr^d))$ and that \textit{(i)} holds. Let $\varphi, \chi \in \mathcal C^{\infty}_c(\rr)$ with $0 \leq \chi \leq 1$ and $\chi_{(-1,1)} \equiv 1$. For all $a \in \mathcal C^{\infty}_c(T^*\mathbb T^d)$ and $R>0$, 
\begin{equation}\label{wm_schro}
    \int_{\rr} \varphi(t) \langle \ops {h_n} a \psi_n, \psi_n \rangle_{L^2(\mathbb T^d)} dt = \langle W_n, b \rangle + r_n(R),
\end{equation}
with $b(t,x,\tau, \xi)= \varphi(t) \chi(\tau/R) a(x, \xi)$ and
\begin{equation*}
    r_n(R) = \int_{\rr} \varphi(t) \langle \ops {h_n} a \psi_n, \psi_n \rangle_{L^2(\mathbb T^d)} dt- \langle \ops {h_n}{b_n} \psi_n, \psi_n \rangle_{L^2(\rr \times \mathbb T^d)}.
\end{equation*}
Moreover, from the third assertion in  Theorem~\ref{thm:sc_results_torus}, we have that
$$\ops {h_n}{b_n} = \ops {h_n^2}{\varphi(t) \chi(\tau/R)} \ops {h_n} a = \varphi(t) \chi(h_n^2 D_t/R)  \ops {h_n} a + O_{\mathcal L (L^2)}(h_n^2).$$
This implies that
\begin{equation*}
    r_n(R)  = \int_{\rr} \varphi(t) \left\langle (1- \chi(h_n^2 D_t/R))\ops {h_n} a \psi_n, \psi_n \right\rangle_{L^2(\mathbb T^d)} dt + O(h_n^2),
\end{equation*}
and
\begin{equation*}\label{RHS_rest}
    |r_n(R)| \leq \int_\rr \varphi(t) \|\ops {h_n}a (1-\chi(h_n^2 D_t/R)) \psi_n\|_{L^2(\mathbb T^d)} dt + O(h_n^2).
\end{equation*}
Moreover, since $$(1-\chi(h_n^2 D_t/R)) \psi_n= (1-\chi(h_n^2 \mathcal H_n/R)) \psi_n,$$ we deduce from the $(h_n)$-oscillation property that
\begin{equation*}\label{RHS_rest2}
\limsup_{n \to +\infty}\|\ops {h_n}a (1-\chi(h_n^2 D_t/R)) \psi_n\|_{L^2(\mathbb T^d)}
\leq C \limsup_{n\to+\infty} \|(1-\chi(h_n^2 \mathcal H_n/R)) \psi_n\|_{L^2(\mathbb T^d)} \underset{R \to +\infty}{\longrightarrow} 0.
\end{equation*}
This shows that 
\begin{equation}\label{eq:remainder}\lim_{R\to +\infty} \limsup_{n \to +\infty} |r_n(R)|=0.
\end{equation}
By gathering \eqref{wm_schro} and \eqref{eq:remainder}, we obtain
\begin{equation*}
    \lim_{n\to+\infty} \int_{\rr} \varphi(t) \langle \ops {h_n} a \psi_n, \psi_n \rangle_{L^2(\mathbb T^d)} dt =\int_{\rr \times \mathbb T^d \times \rr^d} \varphi(t) a(x, \xi) \mu(dt, dx, d\xi).
\end{equation*}
Observe that, since for every $n \in \nn$, $\|\psi_n\|_{L^{\infty}(\rr, L^2(\mathbb T^d))}=1$, we have
$$\forall n \in \nn, \forall \varphi \in \mathcal C^{\infty}_c(\rr), \forall a \in \mathcal C^{\infty}_c(T^*\mathbb T^d), \quad \int_{\rr} \varphi(t) \langle \ops {h_n} a \psi_n, \psi_n \rangle_{L^2(\mathbb T^d)} dt \leq C \|\varphi\|_{L^1(\rr)} \|a\|_{L^{\infty}(T^*\mathbb T^d)}.$$
It follows that $\mu \in L^{\infty}(\rr, \mathcal M_+(T^*\mathbb T^d))$ and the above convergence holds for all $\varphi \in L^1(\rr)$. Notice that, at this step, we have for all real numbers $t_0\leq t_1$, $\mu([t_0,t_1] \times T^*\mathbb T^d) \leq t_1-t_0$, since for all $t \in \rr$, $\| \psi_n(t,\cdot)\|_{L^2(\mathbb T^d)}=1$. \medskip\\
\indent \textit{Third step: proof of assertion (ii).} Let $\chi \in \mathcal{C}^{\infty}_c(\rr)$ with $\chi \equiv 1$ on $(-1,1)$. For $R>1$, we define $\chi_R=\chi(\cdot/R)$. Thanks to the previous step, we deduce that for all $R>1$, $\varphi \in L^1(\rr)$ and $a \in \mathcal C_c^{\infty}(\mathbb T^d)$,
\begin{equation}\label{eq:third_cv1}
    \lim_{n \to +\infty} \int_{\rr} \varphi(t) \langle \ops {h_n} {a_R} \psi_n, \psi_n \rangle_{L^2(\mathbb T^d)} dt =\int_{\rr \times \mathbb T^d \times \rr^d} \varphi(t) a(x) \chi_R(\xi) \mu(dt, dx, d\xi)
\end{equation}
with $a_R(x, \xi)=a(x) \chi_R(\xi)$.\\
By the dominated convergence theorem, we have the following convergence
$$\int_{\rr \times \mathbb T^d \times \rr^d} \varphi(t) a(x) \chi_R(\xi) \mu(dt, dx, d\xi) \underset{R \to +\infty}{\longrightarrow} \int_{\rr \times \mathbb T^d \times \rr^d} \varphi(t) a(x) \mu(dt, dx, d\xi).$$
On the other hand, thanks to the third assertion of Theorem~\ref{thm:sc_results_torus}, we have for all $R>1$
\begin{multline}\label{eq:third_cv2}
 \int_{\rr} \varphi(t) \langle \ops {h_n} {a_R} \psi_n, \psi_n \rangle_{L^2(\mathbb T^d)} dt \\
= \int_{\rr} \varphi(t) \langle a(x) \psi_n, \psi_n \rangle_{L^2(\mathbb T^d)} dt+ \int_{\rr} \varphi(t) \langle a(x) (\chi_R(h_n D_x)-1) \psi_n, \psi_n \rangle_{L^2(\mathbb T^d)} dt +O(h_n).
\end{multline}
Moreover, by using the $(h_n)$-oscillating property and Lemma~\ref{h-oscillating2}, we obtain
\begin{multline}\label{eq:third_cv3}
|\limsup_{n\to +\infty} \int_{\rr} \varphi(t) \langle a(x) (\chi_R(h_n D_x)-1) \psi_n, \psi_n \rangle_{L^2(\mathbb T^d)} dt|\\ \leq \limsup_{n\to +\infty} \int_{\rr} |\varphi(t)| dt \|(1-\chi_R(h_n D_x)) u_n\|_{L^2(\mathbb T^d)} \underset{R \to +\infty}{\longrightarrow} 0.
\end{multline}
Thanks to \eqref{eq:third_cv1}, \eqref{eq:third_cv2} and \eqref{eq:third_cv3}, we deduce that
$$\lim_{n \to +\infty} \int_\rr \varphi(t)\langle a(x) \psi_n, \psi_n\rangle_{L^2(\T^d)} dt =\int_\rr \varphi(t) a(x) \mu(dt, dx, d\xi).$$
This concludes the proof of assertion \textit{(ii)}. \medskip\\
\indent \textit{Fourth step: semiclassical defect measure and geodesic flow.}\\ It remains to show that the semiclassical defect measure is invariant under the geodesic flow. As a first step, let us notice that it is sufficient to establish that
\begin{equation}\label{poisson_bracket}
    \forall \varphi \in \mathcal C^{\infty}_c(\rr), \forall a \in \mathcal{C}^{\infty}_c(T^*\mathbb T^d), \quad \int_{\rr} \int_{T^*\mathbb T^d} \varphi(t) \xi \cdot \nabla_x a(x, \xi) \mu(t, dx, d\xi) dt=0. 
\end{equation}
Indeed, if \eqref{poisson_bracket} holds, then it follows that for all $\varphi \in \mathcal C^{\infty}_c(\rr), a \in \mathcal{C}^{\infty}_c(T^*\mathbb T^d)$, and $s \in \rr$,
\begin{equation*}
    \frac d{ds} \int_{\rr}\int_{T^*\mathbb T^d}\varphi(t) a(x+s\xi, \xi) \mu(dt, dx, d\xi)=\int_{\rr} \int_{T^*\mathbb T^d} \varphi(t) \xi \cdot \nabla_x a_s(x, \xi) \mu(t, dx, d\xi) dt=0,
\end{equation*}
where $a_s(x, \xi)=a(x+s\xi, \xi).$ To conclude, let us show that \eqref{poisson_bracket} holds. Let $\varphi \in \mathcal C^{\infty}_c(\rr)$ and $a \in \mathcal{C}^{\infty}_c(T^*\mathbb T^d)$. Thanks to \eqref{commutators_laplacian} and the second step of this proof, we obtain
\begin{align}\label{commutator_step4}
    \int_{\rr} \int_{T^*\mathbb T^d} \varphi(t) \xi \cdot \nabla_x a(x, \xi) \mu(t, dx, d\xi) dt & = \lim_{n\to +\infty} \int_{\rr} \varphi(t) \langle \ops {h_n} {\xi \cdot \nabla_x a} \phi_n, \phi_n \rangle_{L^2(\mathbb T^d)} dt\nonumber\\
    &= \lim_{n\to +\infty} h_n i \int_{\rr} \varphi(t) \left\langle [\ops {h_n} a, - \Delta] \phi_n, \phi_n \right\rangle_{L^2(\mathbb T^d)} dt.
\end{align}
Moreover, by using the evolution equation satisfied by $\phi_n$, we deduce from integration by parts, the following identity
\begin{align*}
    \int_{\rr} \varphi(t) \left\langle [\ops {h_n} a, - \Delta] \phi_n, \phi_n \right\rangle_{L^2(\mathbb T^d)} dt & = \int_{\rr} \varphi(t) \left\langle [\ops {h_n} a, \mathcal H_{n} ] \phi_n, \phi_n \right\rangle_{L^2(\mathbb T^d)} dt+O_{n\to+\infty}(1)\\ 
    & = \int_{\rr} \varphi (t) i\frac d{dt} \left\langle \ops {h_n} a \phi_n, \phi_n \right\rangle_{L^2(\mathbb T^d)} dt+ O_{n\to +\infty}(1)\\
    & = -i \int_{\rr} \varphi '(t) \left\langle \ops {h_n} a \phi_n, \phi_n \right\rangle_{L^2(\mathbb T^d)} dt+ O_{n\to +\infty}(1).
\end{align*}
In particular, this provides $$\int_{\rr} \varphi(t) \left\langle [\ops {h_n} a, - \Delta] \phi_n, \phi_n \right\rangle_{L^2(\mathbb T^d)} dt= O_{n\to +\infty}(1),$$
and therefore \eqref{commutator_step4} implies \eqref{poisson_bracket}.
\end{proof}
We end this section by a useful example:
\begin{exemple}\label{ex:measure_support}
Let $(h_n)_{n \in \nn} \subset (0,1]$ and $(\rho_n)_{n \in \nn} \subset \rr^*_+$  be two sequences tending to $0$ and $(u_n)_{n \in \nn}$ be a bounded sequence in $L^2(\mathbb T^d)$ such that 
$$\un_{[1-\rho_n,1+\rho_n]}(h_n^2\mathcal H_{n}) u_n=u_n.$$
Then, the sequence $(u_n)_{n \in \nn}$ is $(h_n)$-oscillating and any semiclassical defect measure $\mu \in L^{\infty}(\rr, \mathcal M_+(\T^d))$ provided by Proposition~\ref{semiclassical_measure_prop} satisfies 
$$\Supp \mu \subset \rr \times \T^d \times \mathbb S^{d-1}.$$
\end{exemple}
\begin{proof}[Proof of Example~\ref{ex:measure_support}]
Let $\chi \in \mathcal C^{\infty}_c(\rr^d)$ be a nonnegative test function supported in $\rr^d \setminus \mathbb S^{d-1}$ and $t_0,t_1 \in \rr$ with $t_0< t_1$. By assumption, there exists $\delta>0$ such that for all $\xi \in \Supp \chi$, $$||\xi|^2-1|\geq \delta.$$
Without loss of generality, we can assume that $|\xi|^2-1\geq \delta,$ for $\xi \in \Supp \chi$.
We have
\begin{multline*}
   0 \leq \int_{t_0}^{t_1}  \langle \chi(h_n  D) u_n(t), u_n(t)\rangle_{L^2(\T^d)} dt
    \leq \int_{t_0}^{t_1} \left\langle \chi(h_n D)\frac{(-h_n^2\Delta-1)}{\delta} u_n(t), u_n(t)\right\rangle_{L^2(\T^d)} dt \\ \leq \frac{\|\chi\|_{L^{\infty}}}{\delta} \int_{t_0}^{t_1} \langle (h_n^2 \mathcal H_n-1) \un_{[1-\rho_n, 1+\rho_n]}(h_n^2 \mathcal H_n) u_n(t), u_n(t)\rangle_{L^2(\T^d)} dt+ O(h_n)
    \leq O(\rho_n+h_n).
\end{multline*}
Eventually, this provides 
$$\int_{t_0}^{t_1} \int_{\mathbb T^d \times \mathbb R^d} \chi(\xi) \mu(t, dx, d\xi)dt=0,$$
and then, $\Supp \mu \subset \rr \times \mathbb T^d \times \mathbb S^{d-1}$.
\end{proof}

\bibliographystyle{alpha}
\small{\bibliography{FreeSchrodinger}}

\end{document}